\documentclass{article}
\usepackage{graphicx} 
\usepackage[a4paper,margin=1in]{geometry}

\title{On the Submodule Structure of Hook Specht Modules in Characteristic 2}
\author{Zain Ahmed Kapadia \\ Queen Mary, University of London, Mile End Road, London E1 4NS, UK}

\usepackage{amsmath, mathtools, physics, array, tabularx, multirow, genyoungtabtikz, booktabs}

\usepackage{amsthm, amssymb,ytableau} 

\usepackage{hyperref, cleveref}

\newtheorem{theorem}{Theorem}[section]
\newtheorem*{theorem*}{Theorem}
\newtheorem{conjecture}[theorem]{Conjecture}
\newtheorem*{conjecture*}{Conjecture}
\newtheorem{lemma}[theorem]{Lemma}
\newtheorem*{lemma*}{Lemma}
\newtheorem{proposition}[theorem]{Proposition}
\newtheorem*{proposition*}{Proposition}
\newtheorem{corollary}[theorem]{Corollary}
\newtheorem*{corollary*}{Corollary}

\theoremstyle{definition}
\newtheorem{definition}[theorem]{Definition}
\newtheorem*{definition*}{Definition}
\newtheorem{example}[theorem]{Example}
\newtheorem*{example*}{Example}

\theoremstyle{remark}
\newtheorem{remark}[theorem]{Remark}
\newtheorem*{remark*}{Remark}

\newtheorem*{statement*}{Statement}

\usepackage[many]{tcolorbox}
\tcolorboxenvironment{theorem}{
  enhanced,
  borderline={0.8pt}{0pt}{black},
  borderline={0.4pt}{2pt}{black},
  boxrule=0.4pt,
  colback=white,
  coltitle=black,
  sharp corners,
}
\tcolorboxenvironment{theorem*}{
  enhanced,
  borderline={0.8pt}{0pt}{black},
  borderline={0.4pt}{2pt}{black},
  boxrule=0.4pt,
  colback=white,
  coltitle=black,
  sharp corners,
}
\tcolorboxenvironment{conjecture}{
  enhanced,
  borderline={0.8pt}{0pt}{black},
  boxrule=0.4pt,
  colback=white,
  coltitle=black,
  sharp corners,
}
\tcolorboxenvironment{conjecture*}{
  enhanced,
  borderline={0.8pt}{0pt}{black},
  boxrule=0.4pt,
  colback=white,
  coltitle=black,
  sharp corners,
}

\tcolorboxenvironment{lemma}{
    colback=gray!10,
    boxrule=0.1pt,
    before skip=1em,after skip=1em
}
\tcolorboxenvironment{lemma*}{
    colback=gray!10,
    boxrule=0.1pt,
    before skip=1em,after skip=1em
}

\tcolorboxenvironment{proposition}{
    colback=gray!10,
    boxrule=0.1pt,
    before skip=1em,after skip=1em
}
\tcolorboxenvironment{proposition*}{
    colback=gray!10,
    boxrule=0.1pt,
    before skip=1em,after skip=1em
}
\tcolorboxenvironment{corollary}{
    colback=gray!10,
    boxrule=0.1pt,
    before skip=1em,after skip=1em
}
\tcolorboxenvironment{corollary*}{
    colback=gray!10,
    boxrule=0.1pt,
    before skip=1em,after skip=1em
}

\tcolorboxenvironment{definition}{
	breakable,
    colback=gray!10,
    boxrule=0.1pt,
    before skip=1em,after skip=1em
}
\tcolorboxenvironment{definition*}{
	breakable,
    colback=gray!10,
    boxrule=0.1pt,
    before skip=1em,after skip=1em
}

\tcolorboxenvironment{example}{
	breakable,
    colback=white,
    borderline={0.8pt}{0pt}{black},
    boxrule=0.1pt,
    before skip=1em,after skip=1em
}
\tcolorboxenvironment{example*}{
	breakable,
    colback=gray!10,
    boxrule=0.1pt,
    before skip=1em,after skip=1em
}

\tcolorboxenvironment{remark}{
	breakable,
	blanker,
    before skip=1em,after skip=1em
}
\tcolorboxenvironment{remark*}{
	breakable,
	blanker,
    before skip=1em,after skip=1em
}
\tcolorboxenvironment{statement}{
	breakable,
	blanker,
    before skip=1em,after skip=1em
}
\tcolorboxenvironment{statement*}{
	breakable,
	blanker,
    before skip=1em,after skip=1em
}

\usepackage{tikz, tikz-cd}
\usetikzlibrary{cd, graphs, positioning, fit} 
\newcommand\addvmargin[1]{
  \node[fit=(current bounding box),inner ysep=#1,inner xsep=0]{};
} 

\usepackage{float} 

\usepackage{biblatex} 
\DeclareMathOperator{\Hom}{Hom}
\DeclareMathOperator{\Ker}{Ker}

\DeclareMathOperator{\Dim}{Dim}
\DeclareMathOperator{\res}{Res}
\DeclareMathOperator{\Ind}{Ind}
\newcommand{\floor}[1]{\left \lfloor #1 \right \rfloor} 
\newcommand{\isom}{\cong} 

\begin{document}

\newcommand\mnote[1]{\noindent\color{red}Matt: #1\color{black}}
\newcommand\znote[1]{\noindent\color{blue}Zain: #1\color{black}}
\newcommand\olive[1]{\noindent\color{olive} #1\color{black}}

\maketitle

\begin{abstract}
    The submodule structure of general Specht modules in prime characteristic is a difficult open problem. Kleshchev and Sheth [1999] \cite{kleshchev1999extensions} gave a combinatorial description of the submodule structure of Specht modules labelled by $2$-part partitions in prime characteristic. Using this result, as well as filtrations of Specht modules labelled by hook partitions via $2$-part Specht modules in characteristic $2$, one can study the submodule structure of hook Specht modules. In particular, we classify which of these are uniserial.
\end{abstract}

\tableofcontents

\newpage
\section{Introduction}\label{Introduction}

Let $n$ be a positive integer and let $S_n$ be the symmetric group on $n$ letters. In characteristic $0,$ the irreducible representations of $S_n$ are called Specht modules, $S^\lambda_0,$ and are parameterised by partitions $\lambda$ of $n.$ These Specht modules can reduced mod $p$ for any prime $p,$ giving $S^\lambda_p$ a module over $\mathbb{F}S_n.$ It is known when $S^\lambda_p$ is irreducible in positive characteristic \cite{JamesMathas1999} \cite{Fayers2005}, but understanding the full submodule structure or even the decomposition numbers for reducible Specht modules remains elusive.

However, the decomposition numbers for the Specht modules are fully understood in the case when $\lambda$ is a $2$-part partition, that is, when $\lambda = (\lambda_1, \lambda_2);$ and also in characteristic $2$ when $\lambda$ is a `hook' partition, that is, of the form $\lambda = (n-r, 1^r) \vdash n$ for $0 \leq r \leq n-1$ \cite{murphy1982submodule}. Murphy wrote two papers \cite{murphy1980decomposability} \cite{murphy1982submodule} which looked deeply at the submodule structure of hook Specht modules in characteristic $2$ by studying their endomorphism rings, and also by linking the structure of hook Specht modules in characteristic $2$ to that of $2$-part Specht modules. Despite this, not much progress has been made on this topic since.

This paper takes Murphy's results, as well as results from \cite{kleshchev1999extensions} which give a combinatorial description of $2$-part Specht modules in positive characteristic, to study the submodule structure of hook Specht modules in characteristic $2.$

We now briefly indicate the layout of this paper. 
\Cref{Background} contains background definitions and results regarding Specht modules, while in \cref{Results About 2-Part Specht Modules} we recall some results from the literature, and prove some new results about $2$-part Specht modules. 
In \cref{An Exact Sequence of 2-Part Specht Modules}, we construct an exact sequence of $2$-part Specht modules in the case that $p=2$ and $n$ is even. 
In \cref{A Filtration of Hook Specht Modules by 2-Part Specht Modules}, we explicitly construct a filtration of Specht modules labelled by hooks of the form $(n-r,1^r)$ with $n-r \geq r$ via $2$-part Specht modules in characteristic $2,$ a result which was proven by Sutton \cite{sutton2017graded} using KLR algebras.
For the case $n$ is even and $n-r \leq n,$ we prove in \cref{Another Filtration of Hook Specht Modules} that there is a filtration of $S^\lambda_2$ where all of the factors are isomorphic to $2$-part Specht modules, except for the bottom factor which is a quotient of a $2$-part Specht module isomorphic to a submodule of a $2$-part Specht module.
\Cref{Classification of Uniserial Hook Specht Modules in Characteristic 2} ends with an application of our results to classify all uniserial hook Specht modules in characteristic $2.$
\Cref{Further Questions} concludes the paper with some closing remarks and potential further avenues of research.

The author acknowledges his PhD supervisor, \href{https://webspace.maths.qmul.ac.uk/m.fayers/}{Dr. Matthew Fayers}, for the ideas and guidance provided throughout every step of this paper, and the \href{https://www.ukri.org/councils/epsrc/}{Engineering and Physical Sciences Research Council (EPSRC)} for funding this ongoing PhD project.

\newpage
\section{Background}\label{Background}

In this section we state introductory definitions and results, directing the reader to \cite{James1978} for more details.

Fix a non-negative integer $n.$ A \emph{composition} $\lambda$ of $n$ is a sequence of non-negative integers $\lambda = (\lambda_1, \lambda_2, \ldots)$ such that $\sum_{i\geq 1}\lambda_i = n.$ If $\lambda_i \geq \lambda_{i+1}$ for all $i,$ we say that $\lambda$ is a \emph{partition}, denoted $\lambda \vdash n.$ We will typically write $\lambda$ without trailing zeroes, and in multiplicative notation, for example, the partition $(5,1,1,0, \ldots)$ will be denoted $(5,1^2).$ The \emph{Young diagram} for a partition $\lambda$ is the set $[\lambda] = \{(i,j) \in \mathbb{N} \times \mathbb{N} \, | \, 1 \leq i, 1 \leq j \leq \lambda_i\}.$ Elements of $[\lambda]$ are called \emph{nodes}. We will picture Young diagrams in `English convention': the $r$th row of $[\lambda]$ consists of nodes of $[\lambda]$ with first co-ordinate equal to $r$, and the $c$th column of $[\lambda]$ consists of nodes of $[\lambda]$ with second co-ordinate equal to $c.$ There is a partial order on the set of partitions of $n$ as follows: we say $\lambda \trianglerighteq \mu$ ($\lambda$ `\emph{dominates}' $\mu$) if and only if for all $j, \sum_{i = 1}^j \lambda_i \geq \sum_{i=1}^j \mu_i.$

A \emph{$\lambda$-tableau} is a bijection from $[\lambda]$ to $\{1, \ldots, n\}.$ The symmetric group on $n$ elements, $S_n,$ acts on the set of $\lambda$-tableaux in the natural way. If $t$ is a tableau, \emph{a row stabiliser} (resp. \emph{column stabiliser}) is an element $\sigma \in S_n$ such that $i$ and $\sigma(i)$ are in the same row (resp. column) of $t,$ for all $i.$ We denote the subgroup of row (resp. column) stabilisers by $R_t$ (resp. $C_t$). Note that James \cite{James1978} refers to the subgroups $R_t$ (resp. $C_t$) as the row (resp. column) stabiliser.

We define an equivalence relation on the set of $\lambda$-tableaux by $t \sim s$ if and only if $\exists \sigma \in R_{t}$ such that $\sigma t = s.$ The \emph{tabloid} $\{t\}$ containing $t$ is the equivalence class of $t$ under this equivalence relation. $S_n$ acts on the set of $\lambda$-tabloids: for any $\pi \in S_n, \pi\{t\} := \{\pi t\},$ which is well-defined.

For an arbitrary field $\mathbb{F}$ and a partition $\lambda,$ let $M_\mathbb{F}^\lambda$ be the vector space over $\mathbb{F}$ whose basis elements are the $\lambda$-tabloids, with the $S_n$ action on $M_\mathbb{F}^\lambda$ being a linear extension of the $S_n$ action on tabloids.

For a tableau $t,$ the \emph{signed column sum}, $\kappa_t,$ is the element of the group algebra $\mathbb{F}S_n$ given by \[\kappa_t := \sum_{\sigma \in C_t} sgn(\sigma) \sigma.\] The \emph{polytabloid} $e_t$ is defined as \[e_t := \kappa_t \{t\}.\] If $\sigma$ is a column stabiliser of $t,$ it is clear that $(1+sgn(\sigma)\sigma)e_t = 0.$ The \emph{Specht module} $S_\mathbb{F}^\lambda$ for the partition $\lambda$ is the submodule of $M_\mathbb{F}^\lambda$ spanned by polytabloids. We may write $M^\lambda$ or $S^\lambda$ if the ground field is clear from context, and/or if the result is true for arbitrary ground fields. If the choice of ground field is arbitrary up to the characteristic $p$ of $\mathbb{F},$ we may write $M_p^\lambda$ and $S_p^\lambda.$ If the characteristic of $\mathbb{F}$ is $0,$ then $\{S^\lambda \, | \, \lambda \vdash n\}$ is a complete set of non-repeating irreducible representations for $S_n.$

Let $t$ be a fixed $\lambda$-tableau. Let $X$ be a subset of the $i$th column of $t,$ and $Y$ a subset of the $(i+1)$th column of $t.$ We will take $X$ to be taken from the end of the $i$th column of $t$ and $Y$ will be at the beginning of the $(i+1)$th column. Let $\sigma_1, \ldots, \sigma_k$ be coset representatives for $S_X \times S_Y$ in $S_{X \cup Y},$ and let $G_{X,Y} := \sum_{j=1}^k sgn(\sigma_j)\sigma_j.$ $G_{X,Y}$ is called a \emph{Garnir element}. For practical purposes, we may take $\sigma_1, \ldots, \sigma_k$ so that $\sigma_1 t, \ldots, \sigma_k t$ are all the tableaux which agree with $t$ except in the positions occupied by $X \cup Y,$ and whose entries increase vertically downwards in the positions occupied by $X\cup Y.$ If $| X \cup Y | > \mu'_i$ then $G_{X,Y} e_t = 0$ for any base field. Hence, these Garnir elements allow us to express a polytabloid as a sum of other polytabloids, and these relations are called \emph{Garnir relations}. 

Let $\lambda \vdash n$ and let $M$ be an $\mathbb{F}S_n$-module. Then there is a well-defined $\mathbb{F}S_n$-homomorphism $f : S^\lambda \to M$ such that $f(e_t) \mapsto m$ for some $m \in M$ if and only if $\pi m = 0$ for all $\pi$ of the form
$\pi=1+sgn(\sigma)\sigma$ where $\sigma$ is a column stabiliser of $t,$ and for all 
$\pi$ such that $\pi$ is a Garnir element on neighbouring columns of $t.$ That is, a linear map is an $\mathbb{F}S_n$-module homomorphism if and only if $f$ satisfies all column relations and all Garnir relations.

We define a \emph{standard tableau} $t$ to be a tableau such that the numbers are increasing along the rows (from left to right) and down the column (from top to bottom) of $t.$ We say that $\{t\}$ is a \emph{standard tabloid} if there is a standard tableau in the equivalence class $\{t\}.$ We say that $e_t$ is a \emph{standard polytabloid} if $t$ is standard. Using Garnir elements, one can rewrite any polytabloid $e_t$ as a sum of standard polytabloids. In particular, the standard polytabloids form a basis for the Specht modules, defined over any field, hence the dimension of $S^\lambda$ is independent of the ground field, and equals the number of standard $\lambda$-tableaux.

Let $\langle \,,\, \rangle$ be the unique bilinear form on $M^\lambda$ for which $\langle \{s\}, \{t\} \rangle = 1$ if $\{s\} = \{t\}, 0 $ if $\{s\} \neq \{t\}.$ This is a symmetric, $S_n$-invariant, non-singular bilinear form on $M^\lambda$, regardless of the choice of $\mathbb{F}$. In particular, if $\mathbb{F} = \mathbb{Q},$ then the form is an inner product. Denote $S^{\lambda\perp} = \{u \in M^\lambda \mid \langle u,v \rangle = 0, \forall v \in S^\lambda\}.$

Let $D_\mathbb{F}^\lambda := S_\mathbb{F}^\lambda/(S_\mathbb{F}^\lambda \cap S_\mathbb{F}^{\lambda\perp})$ (sometimes referred to as the `\emph{James module}' of $\lambda$). We have the following major result:

\begin{theorem}\label{thm1}
    $D_\mathbb{F}^\lambda$ is zero or absolutely irreducible. Further, if this is non-zero, then $S_\mathbb{F}^\lambda \cap S_\mathbb{F}^{\lambda\perp}$ is the unique maximal submodule of $S_\mathbb{F}^\lambda$, and $D_\mathbb{F}^\lambda$ is self-dual.
\end{theorem}

We say that a partition $\lambda$ is \emph{$p$-regular} if there is \emph{no} $i \in \mathbb{Z}$ such that $\lambda_i = \lambda_{i+1} = \cdots = \lambda_{i+p-1} > 0$, and we say that $\lambda$ is \emph{$p$-singular} if there is such an $i$. If the characteristic of $\mathbb{F}$ is $p > 0$, one can use \Cref{thm1} to show that $D^\lambda_\mathbb{F} \neq 0$ if and only if $\lambda$ is $p$-regular. If the characteristic of $\mathbb{F}$ is $p > 0,$ then $\{D_p^\lambda \, |, \, \lambda \vdash n, \lambda \text{ is $p$-regular}\}$ is a complete set of non-repeating irreducible modular representations for $S_n.$

If, for $\lambda \vdash n,$ we have $\lambda_3 = 0,$ we say that $\lambda$ is a \emph{$2$-part partition}, and we call $S^\lambda$ a `\emph{$2$-part Specht module}'. Similarly, if $\lambda \vdash n$ is of the form $(n-r, 1^r)$ for $0 \leq r \leq n-1,$ we say that $\lambda$ is a \emph{hook partition}, and we call $S^\lambda$ a `\emph{hook Specht module}'. Note that the partition $\lambda = (n)$ is considered both a $2$-part partition and a hook partition, and so $S^{(n)}$ is also a $2$-part Specht module and a hook Specht module. This paper will focus primarily on $2$-part Specht modules, hook Specht modules, and the relationship between them.

\newpage
\section{Results About 2-Part Specht Modules}\label{Results About 2-Part Specht Modules}

In this section, we recall and also prove some results on the structure of $2$-part Specht modules in characteristic $2.$ More specifically, we recall the combinatorial description of $2$-part Specht modules in positive characteristic; study when these Specht modules are uniserial in characteristic $2;$ calculate their socle; and provide a sufficient condition for when the submodule lattices of two $2$-part Specht modules are isomorphic in characteristic $2.$ 

The decomposition numbers for $2$-part Specht modules in positive characteristic were given by James in \cite[Theorem 24.15]{James1978}, and a full combinatorial description of the submodule structure for $p$-regular $2$-part Specht modules was given by Kleshchev and Sheth in \cite[Corollary 3.4]{kleshchev1999extensions}. We recall the details in this section with some slight notational differences from the literature.

Let $\lambda = (\lambda_1, \lambda_2) \vdash n,$ set $\alpha := \lambda_1 - \lambda_2 + 1,$ and let $\alpha = \sum_{i = 0}^\infty \alpha_i p^i$ be the $p$-adic expansion of $\alpha.$ Write $B_\lambda^- := \{i \, |\,  \alpha_i \neq 0\}, B_\lambda^+ := \{i \, |\,  \alpha_i \neq p-1\},$ and $[i,j) := \{k \in \mathbb{Z} \, |\,  i \leq k < j\}.$ Let $\hat{A}_\lambda$ be the family of sets which comprises the empty set along with any set $I$ of the form \[I = [i_1, i_2) \cup \cdots \cup [i_{2t-1}, i_{2t})\] with $i_1 < i_2 < \ldots < i_{2t}$ and $i_{2j-1} \in B_\lambda^-, i_{2j} \in B_\lambda^+$ for any $j = 1, 2, \ldots, t.$

For every $I \in \hat{A}_\lambda$ define \[\delta_I := \sum_{i \in I}(p-1-\alpha_i)p^i + \sum_{j=1}^t p^{i_{2j-1}},\] and note that $\delta_\emptyset = 0.$

Now set $A_\lambda := \{I \in \hat{A}_\lambda \, |\,  \delta_I \leq \lambda_2\},$ and for any $I \in A_\lambda$ define $\nu_I(\lambda) := (\lambda_1+ \delta_I, \lambda_2-\delta_I).$

\begin{theorem}\label{2_part_submod_structure} \cite[Corollary 3.4]{kleshchev1999extensions}
    Let $\lambda = (\lambda_1, \lambda_2) \vdash n$ be $2$-regular (that is, assume $\lambda_1 > \lambda_2$). Let $S^\lambda_p$ be the Specht module corresponding to $\lambda$ in characteristic $p.$ 
    Then 
    \begin{enumerate}
        \item $S^\lambda_p$ is multiplicity-free, and the set of composition factors is $\{D^{\nu_I(\lambda)}_p \, |\,  I \in A_\lambda\}.$
        \item Write $M_{\nu_I(\lambda)}$ for the smallest submodule of $S^\lambda_p$ that has $D^{\nu_I(\lambda)}_p$ as a composition factor. Then $M_{\nu_J(\lambda)} \supseteq M_{\nu_I(\lambda)}$ if and only if $J \subseteq I.$
    \end{enumerate}
    
\end{theorem}

We note that this theorem does not apply to the case that $p=2$ and $\lambda = (\lambda_1, \lambda_2) \vdash n$ is $2$-singular. We will do this case separately below, but we will first need the submodule structure of $S_2^{(a+1,a)}.$

\begin{example}\label{S_2^{(a+1_a)}}
Let $p = 2$ and let $\lambda = (a+1,a)\vdash n.$ Then $\alpha = 2$ which in binary is $10_2.$ For $i \geq 2,$ let $I_i$ be the interval $[1,i),$ and let $I_1$ denote the empty set. Then $\Tilde{A}_\lambda = \{I_i \, | \, i \geq 1\},$ and $\delta_{I_i} = 2^i - 2,$ so $A_\lambda = \{I_i \, | \, 2^i - 2 \leq a\}.$ Clearly $A_\lambda$ is a totally ordered set, with $I_i \supseteq I_j$ if and only if $i \geq j.$ 

So the set of composition factors is $\{D_2^{(a+1 + 2^i - 2, a - 2^i + 2)} \, | \, 2^i - 2 \leq a\},$ and submodules $\{ M_{(a+1+2^i-2, a-2^i + 2)} \, | \, 0 \leq 2^i-2 \leq a\}$ where $M_{(a+1+2^i-2, a-2^i + 2)} \supseteq M_{(a+1+2^j-2, a-2^j + 2)}$ if and only if $i \leq j.$ Hence $S_2^{(a+1,a)}$ is uniserial.
\end{example}

We will also need some definitions and results regarding restricting to a specific residue, which we take from \cite[p. 11-13]{brundan2000translation}.

For the Young diagram of a partition $\lambda,$ one can label a node $A \in [\lambda]$ with its \emph{residue mod $p$}, that is, for a node in column $c$ and row $r,$ we label the node by $rsd(A) := c-r \mod p.$ For $i \in \mathbb{Z}/p\mathbb{Z}$ we define the $i$-content of $\lambda$ to be the integer $cont_i(\lambda) := |\{A \in [\lambda] \, | \, rsd(A) = i\}|.$ 

The Nakayama conjecture, proved by Brauer and Robinson in 1947 \cite{brauer1947conjecture}\cite{robinson1947conjecture}, provides the block classification of Specht modules and James modules

\begin{theorem}[Nakayama's Conjecture]\label{Nakayama's_Conjecture}\cite{brauer1947conjecture}\cite{robinson1947conjecture}
\begin{itemize}
    \item If $\lambda, \mu \vdash n,$ then $S_p^\lambda$ and $S_p^\mu$ lie in the same block if and only if $cont_i(\lambda) = cont_i(\mu)$ for all $i \in \mathbb{Z}/p\mathbb{Z}.$ \cite[p. 347]{littlewood1951modular}
    \item If $\lambda$ and $\mu$ are $p$-regular partitions of $n,$ then $D_p^\lambda$ and $D_p^\mu$ lie in the same block if and only if $cont_i(\lambda) = cont_i(\mu)$ for all $i \in \mathbb{Z}/p\mathbb{Z}.$ 
\end{itemize}
\end{theorem}

So we can label each block by a list of residue contents $c_0, c_1, \ldots, c_{p-1},$ and a Specht module or James modules lies in this block if and only if the partition which labels it has the same residue contents.

We define the functors $e_i : \mathbb{F}S_n$-mod $\to$ $\mathbb{F}S_{n-1}$-mod and $f_i : \mathbb{F}S_n$-mod $\to$ $\mathbb{F}S_{n+1}$-mod by defining them first on a module $M$ in a fixed block, and then extending additively to all of $\mathbb{F}S_n$-mod. Assume $M$ belongs to the block corresponding to the residue contents $c_0, c_1, \ldots, c_{p-1},$ that is, every composition factor of $M$ is of the form $D_p^\lambda$ with $cont_i(\lambda) = c_i$ for all $i \in \mathbb{Z}/p\mathbb{Z}.$ Let $e_i M$ (resp. $f_i M$) denote the largest submodule of $\Res^{S_n}_{S_{n-1}} M$ (resp. $\Ind^{S_{n+1}}_{S_n} M$) all of whose composition factors are of the form $D_p^\mu$ with $cont_i(\mu) = c_i - 1$ (resp. $cont_i(\mu) = c_i + 1$) and $cont_j(\mu) = c_j$ for all $j \neq i \in \mathbb{Z}/p\mathbb{Z}.$ Then the restriction (resp. induction) functor $\Res^{S_n}_{S_{n-1}}$ (resp. $\Ind^{S_{n+1}}_{S_n}$) can be written as $\bigoplus_{i =0}^{p-1} e_i$ (resp. $\bigoplus_{i =0}^{p-1} f_i$).

A removable node $A$ of $\lambda$ is called a normal $i$-node if $rsd(A) = i$ and for every addable node $B$ to the right of $A$ with $rsd(B) = i,$ there exists a removable node $C_B$ strictly between $A$ and $B$ with $rsd(C_B) = i,$ and moreover $B \neq B'$ implies $C_B \neq C_{B'}.$ A removable node is called a good node if it is the leftmost among the normal nodes of a fixed residue. One can similarly define addable conormal $i$-nodes and cogood nodes.

The following results tell us about the behaviour of $e_i$ on Specht and James modules.

\begin{theorem}\label{e_i S^lambda}

    For $\lambda \vdash n, p$ prime, and $i \in \mathbb{Z}/p\mathbb{Z}, e_i S_p^\lambda$ has a filtration with factors occurring given by $\{S_p^\mu \,|\, \mu \text{ is a partition obtained by $\lambda$ by removing a removable $i$-node}\},$ where $S_p^\mu$ occurs above $S_p^\nu$ in the filtration if $\mu \trianglerighteq \nu$ \cite[Theorem 9.2]{James1978} [\Cref{Nakayama's_Conjecture}].
\end{theorem}

We state a partial version of \cite[Theorem E']{brundan2000translation} which tells us how $e_i$ acts on simple modules. 
\begin{theorem}\cite[Theorem E']{brundan2000translation}\label{e_i D^lambda}
    Let $\lambda \vdash n$ be a $p$-regular partition and $i \in \mathbb{Z}/p\mathbb{Z}.$ Then the module $e_i D_p^\lambda$ is
    \begin{enumerate}
        \item zero unless $\lambda$ has at least one normal node of residue $i;$
        \item irreducible if and only if there is a unique normal node of residue $i.$ 
    \end{enumerate}
\end{theorem}

We now prove a result similar to \Cref{2_part_submod_structure} for the case that $p = 2$ and $\lambda$ is $2$-singular below.

\begin{corollary}\label{submod_structure_S(a_a)}
    Let $\lambda = (a,a) \vdash n.$ Then
    \begin{enumerate}
        \item $S^\lambda_2$ is multiplicity-free, and the set of composition factors is $\{D^{\nu_I(\lambda)} \, |\, I \in A_\lambda \setminus \{\emptyset\}\}$.
        \item Write $M_{\nu_I(\lambda)}$ for the smallest submodule of $S^\lambda_2$ that has $D^{\nu_I(\lambda)}$ as a composition factor. Then $M_{\nu_J(\lambda)} \supseteq M_{\nu_I(\lambda)}$ if and only if $J \subseteq I.$
    \end{enumerate}
\end{corollary}
\begin{proof}

By \cite[Theorem 24.15]{James1978}, $S_2^{(a,a)}$ and $S_2^{(a,a-1)}$ have the same number, say $c,$ of composition factors; namely $D_2^{(x,y)}$ is a composition factor of $S_2^{(a,a-1)}$ if and only if $D_2^{(x+1,y)}$ is a composition factor of $S_2^{(a,a)}.$ Take a composition series of $S_2^{(a,a)}, 0 \subsetneq M_1 \subsetneq M_2 \ldots \subsetneq M_c = S_2^{(a,a)},$ and restrict to $S_{2a-1}.$ The partition $(a,a)$ has one removable node, so the restriction of $S_2^{(a,a)}$ to $S_{2a-1}$ is isomorphic to $S_2^{(a,a-1)},$ and our composition series restricts to a chain  of length $c$ of distinct submodules of $S_2^{(a,a-1)};$ this must also be a composition series for $S_2^{(a,a-1)}.$ So if $T_i := M_i/M_{i-1},$ then the restriction of $T_i$ to $S_{2a-1}$ is also simple.

More specifically, let $i$ be the residue of the removable node of $(a,a),$ that is, $i$ is $a \mod 2.$ Let $T$ be a composition factor, and write $\res(T)$ for the restriction of $T$ to $S_{2a-1}.$ Then $\res(T)$ is the direct sum of $e_0T$ and $e_1T.$ $S_2^{(a,a)}$ has no removable nodes with residue $1-i,$ so $e_{1-i}S_2^{(a,a)} = 0$ and $e_{1-i}T = 0$ too, and $e_i T = \res(T)$ is simple. Similarly, $e_i^2S_2^{(a,a)} = 0$ and hence $e_i^2T = 0$ too. So $T$ is isomorphic to some $D_2^\mu$ where $\mu$ is a $2$-regular partition with exactly $1$ normal $i$-node and no normal $(1-i)$-nodes, and therefore $\res(T) = e_i D_2^{(x,y)} = D_2^{(x-1,y)}$ by counting the number of normal $i$-nodes, counting the number of normal $(1-i)$-nodes, and using \Cref{e_i D^lambda}.

$S_2^{(a,a)}$ has $c$ composition factors, so it has at least $c+1$ submodules. Distinct submodules of $S_2^{(a,a)}$ restrict to distinct submodules in $S_2^{(a,a-1)},$ and $S_2^{(a,a-1)}$ is uniserial by \Cref{submod_structure_S(a_a)} so has exactly $c+1$ submodules, hence $S_2^{(a,a)}$ must have at most $c+1$ submodules. Hence the number of submodules is the same, and the restriction functor is an isomorphism between the submodule lattices of $S_2^{(a,a)}$ and $S_2^{(a,a-1)}.$ 

In the case $\lambda = (a,a-1),$ the set of composition factors for $S_2^{(a,a-1)}$ is $\{D_2^{(a+2^i-2, a-2^i+1)} \, | \, 0 \leq 2^i-2 \leq a-1\},$ and the submodules are of the form $M_{(a+2^i-2, a-2^i+1)}$ which have simple head isomorphic to $D^{(a+2^i-2,a-2^i+1)}_2,$ and $M_{(a+2^i-2, a-2^i+1)} \subset M_{(a+2^j-2, a-2^j+1)}$ if and only if $i > j.$

    For $\lambda = (a,a),$ the set of composition factors for $S_2^{(a,a)}$ is $\{D_2^{(a+2^i-1, a-2^i+1)} \, | \, 0 \leq 2^i-2 \leq a-1\},$ and the submodules are of the form $M_{(a+2^i-1, a-2^i+1)}$ which have simple head isomorphic to $D^{(a+2^i-1,a-2^i+1)}_2,$ and $M_{(a+2^i-1, a-2^i+1)} \subset M_{(a+2^j-1, a-2^j+1)}$ if and only if $i > j.$ This is the same as the combinatorial description given by \Cref{2_part_submod_structure} but excluding $\emptyset$ from $A_{(a,a)}$.

\end{proof}

Because $S_p^\lambda$ is multiplicity free when $\lambda$ is a $2$-part partition  \cite[Theorem 24.15]{James1978} with part two of \Cref{2_part_submod_structure} \Cref{submod_structure_S(a_a)} provides the complete submodule lattice.

We provide an alternative reformulation of \Cref{2_part_submod_structure} below. We will first need some notation and a few small results.

\begin{definition}\label{contains}
    Let $a$ and $b$ be non-negative integers. Let $a = \sum_{i\geq 0}a_ip^i$ and $b=\sum_{j\geq 0}b_jp^j$ be their $p$-adic expansions. We write $a \supseteq_p b$ \emph{``$a$ contains $b$''} if and only if $b_i \in \{0, a_i\}$ for every $i$. We may write $a \supseteq b$ if $p$ is clear from context.
\end{definition}

\begin{lemma}
    Let $p$ be a prime and $\lambda = (\lambda_1, \lambda_2) \vdash n.$ Define $\alpha = \lambda_1 -\lambda_2 + 1,$ and let $0 \leq d \leq \lambda_2$ such that $(\lambda_1 + d, \lambda_2 - d)$ is $p$-regular. Then the multiplicity of $D^{(\lambda_1 + d, \lambda_2 -d)}_p$ in $S_p^{\lambda}$ is $1$ if $\alpha + 2d \supseteq_p d;$ and is $0$ otherwise.
\end{lemma}
\begin{proof}
    This follows from \cite[Theorem 24.15]{James1978}, after rewriting some notation. 
\end{proof}

\begin{corollary}
    Let $p$ be a prime and $\lambda = (\lambda_1, \lambda_2) \vdash n.$ Define $\alpha = \lambda_1 -\lambda_2 + 1,$ and define $A_\lambda$ as in \Cref{2_part_submod_structure}. Then for all $I, J \in A_\lambda, I \subseteq J$ if and only if $\alpha + \delta_J + \delta_I \supseteq_p \delta_I.$
\end{corollary}
\begin{proof}

    If $I = \emptyset$ then this is clear. Now assume $I \neq \emptyset.$
    
  We first show that if $I\subseteq J$ then $\alpha + \delta_J + \delta_I \supseteq_p \delta_I.$ Recall that we can write the $p$-adic expansion \[\alpha = \sum_{u \geq 0} \alpha_u p^u,\] $B^{-}_\lambda := \{u \, | \, \alpha_u \neq 0\},$ and $B^+_{\lambda} := \{u \, | \, \alpha_u \neq p-1\}.$  
  
  As $I$ and $J$ are both non-empty, we can write $I = \bigcup_{k=1}^s [i_{2k-1}, i_{2k})$ for $i_{2k-1} \in B^-_\lambda$ and $i_{2k} \in B^+_\lambda,$ and we can write $J = \bigcup_{l=1}^t [j_{2l-1}, j_{2l}),$ for $j_{2l-1} \in B^-_\lambda$ and $j_{2l} \in B^+_\lambda.$ Then we have 
    \[\delta_I = \sum_{i \in I}(p-1-\alpha_i)p^i + \sum_{k=1}^s p^{i_{2k-1}}\] and 
    \[\delta_J = \sum_{j \in J}(p-1-\alpha_j)p^j + \sum_{l=1}^t p^{j_{2l-1}}.\] 
    
    We calculate $\alpha + \delta_J = \alpha + \sum_{j \in J}(p-1-\alpha_j)p^j + \sum_{l=1}^t p^{j_{2l-1}}.$ Consider the $p$-adic expansion of $\alpha + \sum_{j \in J}(p-1-\alpha_j)p^j.$ This has $p-1$ in position $j$ for $j \in J,$ and $\alpha_u$ in position $u$ for $u \not\in J.$ Then the $p$-adic expansion of $\alpha + \sum_{j \in J}(p-1-\alpha_j)p^j + \sum_{l=1}^t p^{j_{2l-1}}$ has $0$ in position $j,$ for $j \in J;$ $\alpha_{j_{2l}}+1$ for $1 \leq l \leq t;$ and has $\alpha_u$ in position $u$ for all other $u.$ Note that $\alpha_{j_{2l}} \neq p-1$ by definition, and so $\alpha_{j_{2l}}+1$ does not affect the $(j_{2l}+1)$st entry. Hence $\alpha + \delta_J + \delta_I \supseteq_p \delta_I.$

    Now to show that if $I \neq \emptyset$ and $\alpha + \delta_J + \delta_I \supseteq_p \delta_I$ then $I \subseteq J.$ We prove the contrapositive: assume $I \not\subseteq J.$ If $J = \emptyset,$ we want to show that $\alpha + \delta_I \not\supseteq_p \delta_I,$ but this follows as the $i_1$th entry in $\delta_I$ is $(p-\alpha_{i_1}) \neq 0$ as $\alpha_{i_1} \neq 0,$ and the $i_1$th entry in $\alpha + \delta_I$ is $0.$ 
    
    Now assume $J \neq \emptyset$ and let $i$ be the smallest non-negative integer such that $i \in I, i \not\in J.$ 
    
    If $i = i_{2k-1}$ for some $k,$ then the $i_{2k-1}$st entry of $\delta_I$ is $p-\alpha_{i_{2k-1}}.$ If also $i_{2k-1} = j_{2l}$ for some $l,$ then the $i_{2k-1}$st entry of $\alpha + \delta_J + \delta_I$ is $1,$ but $p-\alpha_{i_{2k-1}} \neq 1$ as $\alpha_{j_{2l}}\neq p-1.$ If $i_{2k-1} \neq j_{2l}$ for any $l,$ then the $i_{2k-1}$st entry of $\alpha + \delta_J + \delta_I$ is $0,$ but $p-\alpha_{i_{2k-1}} \neq 0$ as $\alpha_{i_{2k-1}} \neq 0.$ Hence, if $I \not\subseteq J,$ then $\alpha + \delta_J + \delta_I \not\supseteq_p \delta_I,$ as required.

    If $i \in I, i \not\in J, i \neq i_{2k-1},$ then the $i$th entry of $\delta_I$ is $p-1-\alpha_i.$ As $i$ is minimal, we must also have that $i = j_{2l}$ for some $l,$ then the $i$th entry of $\alpha + \delta_J + \delta_I$ is $0,$ but $\alpha_{j_{2l}} \neq p-1.$ Hence, if $I \not\subseteq J,$ then $\alpha + \delta_J + \delta_I \not\supseteq_p \delta_I,$ as required. 
    
\end{proof}

We can now reformulate the statement of \Cref{2_part_submod_structure} without referring to elements of $A_\lambda.$

\begin{corollary}\label{2_part_submod_structure_new_proof}
    Let $\lambda = (\lambda_1, \lambda_2) \vdash n.$ Let $S^\lambda_p$ be the Specht module corresponding to $\lambda$ in characteristic $p.$ Then 
    \begin{enumerate}
        \item $S^\lambda_p$ is multiplicity-free, and the set of composition factors is $\{D^{(\lambda_1 + d, \lambda_2 -d)}_p \,|\, 0 \leq d \leq \lambda_2, (\lambda_1+d,\lambda_2-d) \text{ is $p$-regular}, \alpha + 2d \supseteq_p d \}$.
        \item Write $M_d$ for the smallest submodule of $S_p^\lambda$ that has $D^{(\lambda_1+d, \lambda_2-d)}_p$ as a composition factor. Then $M_{d_1} \supseteq M_{d_2}$ if and only if $\alpha + d_1 + d_2 \supseteq_p d_1.$
    \end{enumerate}
\end{corollary}

We may still use the construction of $A_\lambda$ as in \Cref{2_part_submod_structure} to prove further results if it is more convenient to do so. 

We will now prove some results about the submodule structure of $2$-part Specht modules. For the remainder of this section, we will assume that $p=2$ unless otherwise stated, though some results can be generalised to odd primes. We now give a result on when $2$-part Specht modules in characteristic $2$ are uniserial. First we give some preliminary results.

\begin{lemma}\label{one_subset}
    Let $\lambda = (\lambda_1, \lambda_2) \vdash n,$ let $\alpha = \lambda_1 - \lambda_2 + 1$ and let and $A_\lambda$ be the family of subsets as defined earlier in this section. 

    If a subset $[i_1, i_2) \cup [i_3, i_4)\cup\ldots\cup[i_{2r-1},i_{2r}) \in A_\lambda,$ then $[i_1, i_{2r}) \in A_\lambda.$  
\end{lemma}
\begin{proof}
    It suffices to show that $\delta([i_1, i_{2r})) \leq \lambda_2.$ Write $I = [i_1, i_2) \cup [i_3, i_4)\cup\ldots\cup[i_{2r-1},i_{2r}),$ and for $1 \leq k \leq r-1,$ let \[j_k := 2^{i_{2k+1}} - \sum_{i = i_{2k}}^{i_{2k+1}-1}2^i > 0.\]
    
    Then the result follows as
    \begin{align*}
        \delta([i_1, i_{2r})) &\leq \delta(I) - \sum_{k=1}^{r-1} j_k \\
        & \leq \delta(I) \\
        & \leq \lambda_2.
    \end{align*}
\end{proof}

\begin{lemma}\label{subset_extend}
    Let $\lambda = (\lambda_1, \lambda_2) \vdash n,$ let $\alpha = \lambda_1 - \lambda_2 + 1$ and let $A_\lambda$ be the family of subsets as defined earlier in this section. 

    If a subset $[i, j) \in A_\lambda,$ then $[\nu_2(\alpha), j) \in A_\lambda.$ 
\end{lemma}
\begin{proof}
    If $i = \nu_2(\alpha)$ then there is nothing to prove. Else $\nu_2(\alpha) < i$ and 
    \begin{align*}
    \delta([\nu_2(\alpha), j)) &\leq \delta([i, j)) - 2^{i} + \sum_{k = \nu_2(\alpha)}^{i-1} 2^k
    \\& \leq \delta([i, j))
    \\& \leq \lambda_2,
    \end{align*}
    as required.
\end{proof}

\begin{theorem}\label{2_part_uniserial}
    Let $\lambda = (\lambda_1, \lambda_2) \vdash n.$ If $\alpha := \lambda_1 - \lambda_2 + 1$ has at least two non-zero digits in its binary expansion, define $a := \nu_2(\alpha), b := \nu_2(\alpha + 2^{\nu_2(\alpha)}),$ and $ c := \nu_2(\alpha - 2^{\nu_2(\alpha)}).$ That is, $a$ is the first $1$ in the binary expansion of $\alpha,$ $b$ is the first $0$ after $a,$ and $c$ is the first $1$ after $a.$
   
    Then $S^\lambda_2$ is uniserial if and only if 
    \begin{enumerate}
        \item $\alpha$ is a power of two or;
        \item $\alpha$ is not a power of two and $c > b$ and $2^c > \lambda_2$ or;
        \item $\alpha$ is not a power of two and $c < b$ and $2^c + 2^b > \lambda_2.$
    \end{enumerate}
\end{theorem}

\begin{proof}
    $S^\lambda_2$ is uniserial if and only if the poset $A_\lambda$ is in fact a totally ordered set under inclusion. 

    If $\alpha$ is a power of two, then $A_\lambda$ is a totally ordered set under inclusion, as the elements of $A_\lambda$ are either the empty set, or are of the form $[\nu_2(\alpha), i).$  

    Assume $\alpha$ is not a power of two, then it has at least two non-zero digits in its binary expansion, and we can define $a, b,$ and $c$ as in the statement of the theorem. Define $d$ minimal such that $d > \max\{b,c\}$ and $2^d \not\subseteq_2 \alpha;$ one can think of $d$ as the first $0$ after both $b$ and $c.$ 
    
    $A_\lambda$ is not a totally ordered set if and only if there are $I,J \in A_\lambda$ such that $I\not\subseteq J$ and $J \not\subseteq I.$ As $I$ and $J$ are both disjoint unions of intervals, we can find a constituent interval $I' := [i_1, i_2)$ of $I$ and a constituent interval $J' := [j_1, j_2)$ of $J$ which are incomparable. So $A_\lambda$ is not a totally ordered set if and only there are two intervals $I'$ and $J'$ in $A_\lambda$ which are incomparable. Without loss of generality, I can take $i_1 < j_1$ and $i_2 < j_2.$ We also have $a \leq i_1, b \leq i_2$ and $\delta(I') \geq \delta([a,b))$ so $[a,b)$ and $[j_1, j_2)$ are incomparable, and both are in $A_\lambda$ if and only if there are two incomparable intervals $I'$ and $J'$ both in $A_\lambda$. Similarly, $a < c \leq j_1, b < d \leq j_2$ and $\delta(J') \geq \delta([c,d))$ so $[a,b)$ and $[c,d)$ are incomparable, and both are in $A_\lambda$ if and only if $[a,b)$ and an interval $J'$ are incomparable and both in $A_\lambda$. So $A_\lambda$ is not a totally ordered set if and only if $[a,b)$ and $[c,d)$ are both in $A_\lambda.$ We have $\delta([c,d)) > \delta([a,b)),$ so both intervals are in $A_\lambda$ if and only if $\delta([c,d)) \leq \lambda_2.$ 

    In the case $c > b, [c,d)$ is not in $A_\lambda$ if and only if $\delta([c,d)) = 2^c > \lambda_2.$ 

    In the case $c < b, [c,d)$ is not in $A_\lambda$ if and only if $\delta([c,d)) = 2^c + 2^b > \lambda_2.$
\end{proof}

\begin{corollary}
    Let $\lambda = (\lambda_1, \lambda_2) \vdash n.$ If $\lambda_1 \leq 7$ or $\lambda_2 \leq 3$ then $S^\lambda_2$ is uniserial.
\end{corollary}
\begin{proof}
    This follows immediately from \Cref{2_part_uniserial}.
\end{proof}

 We give a result about the socle of $2$-part Specht modules. We will use the following notation regularly, which we give for general primes $p,$ but we will frequently use when $p=2$:

\begin{definition}\cite{murphy1982submodule}
    For a non-negative integer $r$ and a prime $p,$ define $L_p(r)$ to be the smallest non-negative integer such that $r < p^{L_p(r)}.$ If $p=2$ we simply denote this as $L(r).$ 
\end{definition}
One can think of $L_p(r)$ as the number of significant digits when $r$ is written in its $p$-adic expansion.

\begin{definition}
    Let $\alpha$ and $i$ be non-negative integers. We write $\overline{\alpha}_i$ for the remainder when $\alpha$ is divided by $2^i.$ 
\end{definition}

Murphy proved that a $2$-part Specht module has a unique minimal submodule in any characteristic \cite[Theorem 4.6]{murphy1982submodule}. We give a new proof of this when $p=2$ and identify this minimal submodule explicitly.

\begin{theorem}\label{2-part_socle_pre}
    Let $\lambda = (\lambda_1, \lambda_2) \vdash n, \alpha = \lambda_1 - \lambda_2 + 1,$ and let $L = L(\lambda_2).$ Then the socle of $S^\lambda$ in characteristic $2$ is $D^{(\lambda_1+d, \lambda_2 - d)},$ where

    \[
    d =
    \begin{cases}
       0 & \text{ if } \overline{\alpha}_L  = 0 \\  
       2^L - \overline{\alpha}_L & \text{ if } \overline{\alpha}_L \geq 2^L - \lambda_2,  \\
       2^{L-1} - \overline{\alpha}_{L-1} & \text{ if } \overline{\alpha}_L < 2^L - \lambda_2, \overline{\alpha}_L \neq 0.
    \end{cases}
    \]
\end{theorem}
\begin{proof}
    First we show that the socle of $S_2^\lambda$ is simple; that is, there is a unique maximal element in $A_\lambda$ with respect to inclusion (recall that the elements of $A_\lambda$ are subsets of the natural numbers). If $A_\lambda$ contains only the empty set, then this is clear. Otherwise, let $I \in A_\lambda$ be maximal, then by \Cref{one_subset,subset_extend}, it must be of the form $[\nu_2(\alpha), i)$ for some $i.$ Any other maximal element must also be of this form, say $[\nu_2(\alpha), j)$ for some $j,$ but then these two elements can be compared, so there is a unique maximal element in $A_\lambda$.

    In the first case, $\nu_2(\alpha) \geq L > \lambda_2,$ and so the only subset in $A_\lambda$ is the empty set. $\delta(\emptyset) = 0$ and so $S_2^\lambda = D_2^\lambda$ and the socle is $D_2^\lambda$.

    Now assume that $\nu_2(\alpha) < L.$ 

    Assume $\lambda_2 \geq 2^L - \overline{\alpha}_L.$ Let $j \geq L$ be minimal such that $2^j \not\subseteq_2 \alpha.$ Then $\delta([\nu_2(\alpha), j)) = 2^j - \overline{\alpha}_j = 2^L - \overline{\alpha}_L \leq \lambda_2.$ If $k > j$ and $2^k \not\subseteq_2 \alpha,$ then $\delta([\nu_2(\alpha), k)]) \geq 2^k > 2^L > \lambda_2,$ hence $[\nu_2(\alpha), j)$ is the subset of $A_\lambda$ corresponding to the socle of $S_2^\lambda.$

    Now assume $\nu_2(\alpha) < L$ and $\lambda_2 < 2^L - \overline{\alpha}_L.$ Then $2^L -1 \geq 2^L - \overline{\alpha}_L > \lambda_2 \geq 2^{L-1},$ so $2^{L-1} \not\subseteq_2 \alpha.$ Then $[\nu_2(\alpha), L-1)$ is the subset of $A_\lambda$ corresponding to the socle of $S_2^\lambda,$ and $\delta([\nu_2(\alpha), L-1)) = 2^{L-1}- \overline{\alpha}_{L-1}.$
\end{proof}

\begin{corollary}\label{2-part_socle}
    Let $\lambda = (\lambda_1, \lambda_2) \vdash n, \alpha = \lambda_1 - \lambda_2 + 1, L = L(\lambda_2),$ and let $0 \leq d < 2^L$ be such that $\alpha + d \equiv 0 \mod 2^L.$ Then the socle of $S_2^\lambda$ is 
    \begin{itemize}
        \item $D_2^{(\lambda_1 + d, \lambda_2 - d)}$ if $0 \leq d \leq \lambda_2,$ or
        \item $D_2^{(\lambda_1 + d - 2^{L-1}, \lambda_2 - d + 2^{L-1})}$ if $\lambda_2 < d < 2^L$.
    \end{itemize}
\end{corollary}

\begin{corollary}\label{2_part_trivial_socle}
    Let $\lambda = (\lambda_1,\lambda_2) \vdash n.$ Then the socle of $S^\lambda_2$ is trivial if and only if $\lambda_1 \equiv -1 \mod 2^{L(\lambda_2)}.$
\end{corollary}

\begin{remark}
    For general $\lambda \vdash n$ and positive characteristic, it is known when $S_p^\lambda$ has a submodule isomorphic to the trivial module $S_p^{(n)}$ \cite[Theorem 24.4]{James1978}. In particular, if the characteristic is $2,$ this happens if and only if for all $i, \lambda_i \equiv -1 \mod 2^{L(\lambda_{i+1})}.$ Combining this with the result that states that $S_2^\lambda$ has a unique minimal submodule when $\lambda$ is a $2$-part partition \cite[Theorem 4.6]{murphy1982submodule}, we also get the above corollary without needing \Cref{2-part_socle}.
\end{remark}

One can observe that the condition that $I \in A_\lambda$ only if $\delta_I \leq \lambda_2$ means that the submodule structure of $S_p^\lambda$ is dependent only on the first $L_p(\lambda_2)$ digits of the $p$-adic expansion of $\alpha.$ We make this idea more precise in the following lemma.

\begin{lemma}\label{period_2_part}
    Let $\lambda = (\lambda_1, \lambda_2) \vdash n, \mu = (\mu_1, \mu_2) \vdash m$ be $p$-regular partitions, and write $L = L_p(\lambda_2)$. If $\lambda_2 = \mu_2$ and $\lambda_1 \equiv \mu_1 \mod p^L$ then the submodule lattices of $S_p^\lambda$ and $S_p^\mu$ are isomorphic.
\end{lemma}
\begin{proof}
    We prove the lemma by finding a bijection between the set of composition factors of $S_p^\lambda$ and $S_p^\mu$ which respects submodule containment.
    
    Define $\alpha_\lambda := \lambda_1 - \lambda_2 + 1,$ and $\alpha_\mu := \mu_1 - \mu_2 + 1.$ Write $D_\lambda := \{0 \leq d \leq \lambda_2 \,|\, (\lambda_1 + d, \lambda_2 - d) \text{ is $p$-regular and } \alpha_\lambda + 2d \supseteq_p d\}$ and $D_\mu := \{0 \leq d \leq \mu_2 \,|\, (\mu_1 + d, \mu_2 - d) \text{ is $p$-regular and } \alpha_\mu + 2d \supseteq_p d\}.$

    Clearly $\alpha_\lambda \equiv \alpha_\mu \mod p^L,$ and so $\alpha_\lambda + 2d$ and $\alpha_\mu + 2d$ agree on the first $L$ digits. Hence $\alpha_\lambda + 2d \supseteq_p d$ if and only if $\alpha_\mu + 2d \supseteq_p d,$ and $D_\lambda = D_\mu.$

    Similarly, if $d_1, d_2 \in D_\lambda$ then $\alpha_\lambda + d_1 + d_2$ and $\alpha_\mu + d_1 + d_2$ agree on the first $L$ digits, hence $\alpha_\lambda + d_1 + d_2 \supseteq_p d_1$ if and only if $\alpha_\mu + d_1 + d_2 \supseteq_p d_1.$ 

\end{proof}

\begin{remark}
    Note that this bijection maps the irreducible subquotient in $S_p^\lambda$ labelled by $D_p^{(\lambda_1 + d, \lambda_2 - d)}$ to the irreducible subquotient in $S_p^\mu$ labelled by $D_p^{(\mu_1 + d, \mu_2 - d)}.$
\end{remark}

\begin{example}\label{2_part_isolatic}
We illustrate \Cref{period_2_part} by comparing the submodule lattices of $S^{(9,5)}_2, S^{(17,5)}_2$ and $S^{(23,5)}_2.$ In the diagrams below, the vertices represent submodules of $S^\lambda_2,$ and the nodes are labelled by the dimension of the submodule. The directed edge from a vertex $u$ to a vertex $v$ indicates that the submodule associated to $u$ is a maximal submodule of the submodule associated to $v,$ and the edges are labelled by their irreducible subquotients. 

\scalebox{0.8}{
\begin{tabular}{|c|c|c|}
        \hline
        \begin{tikzpicture}[main/.style = {draw, circle}, scale = 3]
            \node at (0,-1/2) {The submodule lattice for $S^{( 9, 5 )}_2$};
            
            \node[main] (0) at (0,0) {$0$};
            \node[main] (64) at (0,1) {$64$};
            \node[main] (65) at (0,2) {$65$};
            \node[main] (77) at (-1/2,3) {$77$};
            \node[main] (429) at (1/2,3) {$429$};
            \node[main] (441) at (0,4) {$441$};
            \node[main] (1001) at (0,5) {$1001$};
            
            \draw[->] (0) edge ["$D_2^{(12,2)}$"', pos = 0.5] (64);
            
            \draw[->] (64) edge ["$D_2^{(14)}$"', pos = 0.5] (65);
            
            \draw[->] (65) edge ["$D_2^{(13,1)}$", pos = 0.5] (77);
            \draw[->] (65) edge ["$D_2^{(10,4)}$"', pos = 0.5] (429);
            
            \draw[->] (77) edge ["$D_2^{(10,4)}$", pos = 0.5] (441);
            
            \draw[->] (429) edge ["$D_2^{(13,1)}$"', pos = 0.5] (441);
            
            \draw[->] (441) edge ["$D_2^{(9,5)}$"', pos = 0.5] (1001);
            \addvmargin{1mm}
        \end{tikzpicture}
         &  
         \begin{tikzpicture}[main/.style = {draw, circle}, scale = 3]
            \node at (0,-1/2) {The submodule lattice for $S^{( 17, 5 )}_2$};
            
            \node[main] (0) at (0,0) {$0$};
            \node[main] (188) at (0,1) {$188$};
            \node[main] (189) at (0,2) {$189$};
            \node[main] (209) at (-1/2,3) {$209$};
            \node[main] (4655) at (1/2,3) {$4655$};
            \node[main] (4675) at (0,4) {$4675$};
            \node[main] (14344) at (0,5) {$14344$};
            
            \draw[->] (0) edge ["$D_2^{(20,2)}$"', pos = 0.5] (188);
            
            \draw[->] (188) edge ["$D_2^{(22)}$"', pos = 0.5] (189);
            
            \draw[->] (189) edge ["$D_2^{(21,1)}$", pos = 0.5] (209);
            \draw[->] (189) edge ["$D_2^{(18,4)}$"', pos = 0.5] (4655);
            
            \draw[->] (209) edge ["$D_2^{(18,4)}$", pos = 0.5] (4675);
            
            \draw[->] (4655) edge ["$D_2^{(21,1)}$"', pos = 0.5] (4675);
            
            \draw[->] (4675) edge ["$D_2^{(17,5)}$"', pos = 0.5] (14344);
            \addvmargin{1mm}
        \end{tikzpicture}
        &
        \begin{tikzpicture}[main/.style = {draw, circle}, scale = 3]
            \node at (0,-1/2) {The submodule lattice for $S^{( 25, 5 )}_2$};
            
            \node[main] (0) at (0,0) {$0$};
            \node[main] (376) at (0,1) {$376$};
            \node[main] (377) at (0,2) {$377$};
            \node[main] (405) at (-1/2,3) {$405$};
            \node[main] (20097) at (1/2,3) {$20097$};
            \node[main] (20125) at (0,4) {$20125$};
            \node[main] (115101) at (0,5) {$115101$};
            
            \draw[->] (0) edge ["$D_2^{(28,2)}$"', pos = 0.5] (376);
            
            \draw[->] (376) edge ["$D_2^{(30)}$"', pos = 0.5] (377);
            
            \draw[->] (377) edge ["$D_2^{(29,1)}$", pos = 0.5] (405);
            \draw[->] (377) edge ["$D_2^{(26,4)}$"', pos = 0.5] (20097);
            
            \draw[->] (405) edge ["$D_2^{(26,4)}$", pos = 0.5] (20125);
            
            \draw[->] (20097) edge ["$D_2^{(29,1)}$"', pos = 0.5] (20125);
            
            \draw[->] (20125) edge ["$D_2^{(25,5)}$"', pos = 0.5] (115101);
            \addvmargin{1mm}
        \end{tikzpicture}
        \\
         \hline
    \end{tabular}
    }

\end{example}

\newpage
\section{An Exact Sequence of 2-Part Specht Modules}\label{An Exact Sequence of 2-Part Specht Modules}

Later in this paper, we will need some results relating to homomorphisms between Specht modules labelled by two part partitions. We state and prove some relevant results here.

\begin{theorem}[Kernel Intersection Theorem]\cite[Corollary 17.18]{James1978}\label{Kernel Intersection Theorem}
    Let $\mu = (\mu_1, \mu_2, \ldots, \mu_r) \vdash n,$ and for $1 \leq d \leq r-1, 1 \leq u \leq \mu_{d+1},$ let $\nu = (\mu_1, \mu_2, \ldots, \mu_{d-1}, \mu_d + u, \mu_{d+1}-u, \mu_{d+2}, \ldots, \mu_r)$ be a composition of $n.$ Let $\psi_{d,u} : M^{\mu} \to M^{\nu}$ be the homomorphism that maps a tabloid $\{s\}$ to the sum of all tabloids $\{t\}$ such that $\{t\}$ agrees with $\{s\}$ in every row except the $d$ and $(d+1)$st rows, and the $(d+1)$st row of $\{t\}$ is a subset of size $(\mu_{d+1} - u)$ of the $(d+1)$st row of $\{s\}.$ Then \[S^\mu = \bigcap_{d=1}^{r-1} \bigcap_{u=1}^{\mu_{d+1}} \Ker(\psi_{d,u}).\]    
\end{theorem}

Note that the maps we call $\psi_{d,u}$ are called $\psi_{d,\mu_{d+1}-u}$ in \cite[Definition 17.10]{James1978}.

\begin{definition} \cite[Definitions 13.1, 13.6]{James1978}
    Let $\lambda \vdash n$ and let $\mu$ be a composition of $n.$ A $\lambda$-tableau of type $\mu$ is a mapping $T$ from the Young diagram of $\lambda$ to $\mathbb{N}$ such that for all $i,$ the number of nodes that map to $i$ is exactly $\mu_i.$ We denote the set of $\lambda$-tableaux of type $\mu$ by $\mathcal{T}(\lambda,\mu).$

    If $T$ is a $\lambda$-tableau of type $\mu,$ we say that $T$ is \emph{semistandard} if the numbers are weakly increasing along the rows of $T$ (from left to right) and strictly increasing down the columns of $T$ (from top to bottom). We denote the set of semistandard $\lambda$-tableaux of type $\mu$ by $\mathcal{T}_0(\lambda, \mu).$ 
\end{definition}
    Note that what we previously called a $\lambda$-tableau is a $\lambda$-tableau of type $(1^n)$.

\begin{theorem}\label{dim_num_semistand_homs}\cite[Corollary 13.14]{James1978}
     Unless $p = 2$ and $\lambda$ is $2$-singular, $\dim \Hom_{\mathbb{F}S_n}(S_p^\lambda, M_p^\mu) = |\mathcal{T}_0(\lambda, \mu)|.$
\end{theorem}

\begin{lemma}\label{homs_2_parts}
    Suppose $n=2k,$ and assume $0 \leq i < k.$ Then there is a unique, up to scaling, non-zero homomorphism $\hat{\Theta}_i : S^{(n-i,i)}_2 \to S^{(n-i-1,i+1)}_2.$
\end{lemma}
\begin{proof}
    To construct a homomorphism from $S^{(n-i,i)}_2 \to S^{(n-i-1,i+1)}_2,$ we can use \cite[Definition 13.3]{James1978} to construct, given a semistandard $(n-i,i)$-tableau $S$ of type $(n-i-1,i+1),$ a homomorphism $\Theta_S : M^{(n-i,i)}_2\to M^{(n-i-1,i+1)}_2,$ and then restrict to the function $\hat{\Theta}_S : S^{(n-i,i)}_2 \to M^{(n-i-1,i+1)}_2.$ By the Kernel Intersection Theorem (\Cref{Kernel Intersection Theorem}), we then just need to check that for $d=1$ and $1 \leq u \leq i+1,$ we have $\psi_{d,u}\circ \hat{\Theta}_S = 0.$

For partitions $\lambda, \mu \vdash n,$ let $\mathcal{T}_0(\lambda, \mu)$ be the set of semistandard $\lambda$-tableaux of type $\mu.$ By \Cref{dim_num_semistand_homs} we have, unless the characteristic of the field is $2$ and $\lambda$ is $2$-singular, that $\Dim(\Hom(S_p^{\lambda}, M_p^{\mu})) = |\mathcal{T}_0(\lambda, \mu)|.$ In the case that $0 \leq i < k, \lambda = (n-i,i)$ and $\mu = (n-i-1,i+1),$ there is exactly one semistandard $(n-i,i)$-tableau $S$ of type $(n-i-1,i+1).$ More explicitly, $S$ is the $(n-i,i)$-tableau which has all $1$s in the first row except the last entry, and all $2$s in the second row. Hence there is, up to scaling, one homomorphism $\hat{\Theta}_S : S^{(n-i,i)}_2 \to M^{(n-i-1,i+1)}_2.$ 

The composition of maps $\psi_{1,u}\circ\hat\Theta_S$ is a homomorphism from $S^{(n-i,i)}_2$ to $M^{(n-i+u-1,i-u+1)}_2$. If $1 < u \leq i+1,$ then there are no semistandard $(n-i,i)$-tableaux of type $(n-i+u-1,i-u+1),$ and therefore $\Hom(S^{(n-i,i)}_2,M^{(n-i+u-1,i-u+1)}_2)=0,$ and $\psi_{1,u}\circ\hat\Theta_S$ is the zero map. In the case $u=1,$ let $T$ be the semistandard $(n-i,i)$-tableau of type $(n-i,i),$ and let $T'$ be the row standard $(n-i,i)$-tableau of type $(n-i,i)$ which has one $2$ in the first row and one $1$ in the second row. Then, by \cite[Lemma 4.3]{dodge2012some} \[\psi_{1,1} \circ \Theta_S = (n-i)\Theta_T + \Theta_{T'}.\] By \cite[Lemma 4.4]{dodge2012some} we also get \[\hat{\Theta}_{T'} = i\hat{\Theta}_T,\] and so \[\psi_{1,1} \circ \hat{\Theta}_S = n \hat{\Theta}_T = 0\] as $n$ is even. Hence by the Kernel Intersection Theorem, $\Im(\hat{\Theta}_S) \subseteq S^{(n-i-1,i+1)}_2.$
\end{proof}

We call this map $\hat{\Theta}_i.$

We now state the main result of this section.

\begin{proposition}\label{2_part_exact_sequence}
    Suppose $n=2k,$ then there is an exact sequence \[0 \xrightarrow{\hat{\Theta}_{-1}} S^{(n)}_2 \xrightarrow{\hat{\Theta}_0} S^{(n-1,1)}_2 \xrightarrow{\hat{\Theta}_1} \ldots \xrightarrow{\hat{\Theta}_{k-2}} S^{(k+1,k-1)}_2 \xrightarrow{\hat{\Theta}_{k-1}} S^{(k,k)}_2 \xrightarrow{\hat{\Theta}_k} 0.\]
\end{proposition}

\begin{proof}
    By \Cref{homs_2_parts}, $\hat{\Theta}_i = \hat{\Theta}_{S} : S^{(n-i,i)}_2 \to S^{(n-i-1,i+1)}_2,$ where $S$ is the semistandard $(n-i,i)$-tableau of type $(n-i-1,i+1).$ That is, $S$ is the semistandard $(n-i,i)$-tableau which has all $1$s in the first row except the last entry which is $2,$ and all $2$s in the second row.
    
    First, we show that for $-1 \leq i \leq k-1, \hat{\Theta}_{i+1} \circ \hat{\Theta}_i = 0.$ If $i = -1$ or $i = k-1,$ this is clear, so assume $0 \leq i \leq k-2.$ Let $T$ be the semistandard $(n-i-1,i+1)$-tableau of type $(n-i-2,i+2).$ That is, $T$ is the semistandard $(n-i-1,i+1)$-tableau which has all $1$s in the first row except the last entry, and all $2$s in the second row. Then $\hat{\Theta}_{i+1} = \hat{\Theta}_T.$ Let $U$ be the semistandard $(n-i,i)$-tableau of type $(n-i-2,i+2),$ that is, $U$ is the semistandard $(n-i,i)$-tableau which has all $1$s in the first row except two $2$s, and all $2$s in the second row. By \cite[Proposition 4.7]{dodge2012some}, we have \[\Theta_T \circ \Theta_S = \binom{n-i-2+0}{n-i-2}\binom{1+1}{1}\binom{0+0}{0}\binom{0+i}{0}\Theta_U = 0.\] Hence $\hat{\Theta}_{i+1}\circ \hat{\Theta}_i = 0,$ and $\Im(\hat{\Theta}_i) \subseteq \Ker(\hat{\Theta}_{i+1}).$ 

    We need to show that $\Im(\hat{\Theta}_i) \supseteq \Ker(\hat{\Theta}_{i+1}).$ If $i = -1,$ this follow as the map $S^{(n)}_2 \xrightarrow{\hat{\Theta}_0} S^{(n-1,1)}_2$ is a non-zero map from a simple module, hence is an injection, so $\Im(\hat{\Theta}_i) = 0 = \Ker(\hat{\Theta}_{i+1}).$

    To show that $\Im(\hat{\Theta}_i) \supseteq \Ker(\hat{\Theta}_{i+1})$ when $0 \leq i \leq k-2,$ we consider the restriction of the exact sequence to $S_{n-1},$ noting that restriction is an exact functor and so the sequence is exact over $KS_n$ if and only if it is exact over $KS_{n-1}.$ We will denote the restriction of $S_2^\lambda$ to $S_{n-1}$ by $\Res(S_2^\lambda),$ and the induced restriction of $\hat{\Theta}_i$ by $\hat{\Theta}'_i.$

    If $i=0$ then $\Res(S^{(n)}_2) \quad \isom S^{(n-1)}_2,$ and if $i=k$ then $\Res(S^{(k,k)}_2) \quad \isom S^{(k,k-1)}_2.$ If $0<i<k$ then $\Res(S^{(n-i,i)}_2) $ has a filtration (from bottom to top) via $S^{(n-i-1,i)}_2$ and $S^{(n-i,i-1)}_2$ \cite[Theorem 9.3]{James1978}. As $n$ is even, these lie in different blocks \cite[Theorem 21.11, Lemma 21.12]{James1978}, and so we actually have that $\Res(S^{(n-i,i)}_2) \isom S^{(n-i-1,i)}_2 \oplus S^{(n-i,i-1)}_2.$ We denote $S^{(n-i-1,i)+}_2$ for the copy of $S^{(n-i-1,i)}_2$ inside of $\Res(S^{(n-i-1,i+1)}_2)$ and $S^{(n-i-1,i)-}_2$ for the copy of $S^{(n-i-1,i)}_2$ inside of $\Res(S^{(n-i,i)}_2).$ 

    Consider each $\hat{\Theta}'_i |_{S^{(n-i-1,i)-}_2}$ whose codomain is $\Res(S^{(n-i-1,i+1)}_2) \quad \isom S^{(n-i-1,i)+}_2 \oplus S^{(n-i-2,i+1)-}_2.$ $S^{(n-i-1,i)}_2$ and $S^{(n-i-2,i+1)}_2$ lie in different blocks, and hence $\Im(\hat{\Theta}'_i) \cap S^{(n-i-2,i+1)-}_2 = 0,$ so $\Im(\hat{\Theta}'_i|_{S^{(n-i-1,i)-}_2}) \subseteq S^{(n-i-1,i)+}_2.$ 
    
    If $i=0, \hat{\Theta}_0$ is a non-zero map from $S^{(n)}_2$ to $S^{(n-1,1)}_2,$ so $\hat{\Theta}'_0$ is a non-zero map from $S^{(n-1)-}_2$ to $S^{(n-1)+}_2,$ and $S^{(n-1)}_2$ is irreducible, hence $\hat{\Theta}'_0$ is an isomorphism. Now assume for induction that $\hat{\Theta}'_{i-1}$ restricts to an isomorphism from $S^{(n-i,i-1)-}_2$ to $S^{(n-i,i-1)+}_2$. We want to show that $\hat{\Theta}'_i$ restricts to an isomorphism from $S^{(n-i-1,i)-}_2$ to $S^{(n-i-1,i)+}_2$. Indeed, we have that $\Im(\hat{\Theta}'_{i-1}) \supseteq S^{(n-i,i-1)+}_2$ by assumption and $\hat{\Theta}_i \circ \hat{\Theta}_{i-1} = 0,$ so if $\hat{\Theta}'_i$ restricts to zero on $S^{(n-i-1,i)-}_2$ then $\hat{\Theta}_i$ is the zero map which is a contradiction, so $\hat{\Theta}'_i$ is non-zero on $S^{(n-i-1,i)-}_2.$ Furthermore, as $S^{(n-i-1,i)}_2$ has simple head, namely $D^{(n-i-1,i)}_2,$ which only appears once as a composition factor and $\hat{\Theta}'_i|_{S^{(n-i-1,i)-}_2}$ is a non-zero map from a copy of $S^{(n-i-1,i)}_2$ to itself, we must have that $\hat{\Theta}'_i|_{S^{(n-i-1,i)-}_2}$ is an isomorphism, otherwise $S^{(n-i-1,i)}_2$ would have $D^{(n-i-1,i)}_2$ as the head of a proper submodule and $S^{(n-i-1,i)}_2$ would no longer be multiplicity-free. 
    
    Hence $\Ker(\hat{\Theta}'_i) \subseteq S^{(n-i,i-1)+}_2$ for all $1 \leq i \leq k$ and $\Im(\hat{\Theta}'_i) \supseteq S^{(n-i-1,i)+}_2$ for all $0 \leq i \leq k-1.$ As $\hat{\Theta}_i$ and $\hat{\Theta}'_i$ are identical as linear maps, we have that when $0 \leq i < k-1, \Im(\hat{\Theta}_i) \supseteq \Ker(\hat{\Theta}_{i+1})$ as required, and the sequence is exact.

\end{proof}

\begin{remark}\label{2_part_explicit_hom}
    We explicitly calculate $\hat{\Theta}_i(e_s)$ on a polytabloid $e_s$ in $S_2^{(n-i,i)}.$ For $0 \leq i \leq k-1,$ write $n-2i = 2l$ and let $s$ be an $(n-i,i)$-tableau, with entries $a_1, \ldots, a_i, c_1, \ldots, c_{2l}$ in the first row (from left to right), and $b_1, \ldots, b_i$ in the second row (from left to right). Then \[\hat{\Theta}_i(e_s) = \prod_{k=1}^i(1+(a_k,b_k))\Theta_i(\{s\}).\] 

Where $\Theta_i(\{s\})$ is a linear combination of tabloids of the following types:
\begin{itemize}
    \item the second row contains $b_j$ for every $j,$ and exactly one $a_k$ for some $k$ between $1$ and $i,$ or
    \item the second row contains $b_j$ for every $j,$ and exactly one $c_k$ for some $k$ between $1$ and $2l.$
\end{itemize}

But any tabloid of the first type is annihilated by $(1+(a_k, b_k)),$ so $\hat{\Theta}_i(e_s)$ is the sum of all tabloids which have in the second row exactly one of $a_j$ or $b_j$ for every $j,$ and exactly one $c_k$ for some $k$. We can write this as a sum of polytabloids in the following way:

Let $1 \leq j \leq l,$ and write $t_j$ for the $(n-i-1,i+1)$-tableau which has $a_1, \ldots, a_i, c_{2j-1},$ in the first row (from left to right), $b_1, \ldots, b_i, c_{2j}$ in the second row (from left to right), and $c_1, \ldots, \hat{c}_{2j-1}, \hat{c}_{2j}, \ldots, c_{2l}$ in the first row (from left to right) after $c_{2j-1}$, where $\hat{c}_{2j-1}$ and $\hat{c}_{2j}$ means we skip over $c_{2j-1}$ and $c_{2j}$. 
    
    Then it is easy to see that \[\hat{\Theta}_i (e_s) := \sum_{1 \leq j \leq l} e_{t_j}.\] 
\end{remark}

\begin{remark}
Note that as $n$ is even, \Cref{2_part_submod_structure_new_proof} states that for $i \geq 1, \lambda := (n-i,i), \alpha := \lambda_1 - \lambda_2 + 1 = n-2i+1$ is odd, so $\alpha + 2 \supseteq_2 1$ and there is a submodule $M_1 \subseteq S_2^{(n-i,i)}$ with simple head isomorphic to $D_2^{(n-i+1,i-1)}.$ Additionally, for any other odd $d$ such that $d \leq i$ and $ \alpha + 2d \supseteq_2 d,$ we must have that $\alpha + 1 + d \supseteq_2 1,$ and so $M_1$ contains all of the composition factors $D_2^{(n-k,k)}$ of $S_2^{(n-i,i)}$ where $k$ and $i$ have different parity. $S_2^{(n-i,i)}$ is multiplicity-free, so there is a unique submodule with simple head isomorphic to $D_2^{(n-i+1,i-1)},$ so we must have that $M_1 = \Im(\hat{\Theta}_{i-1}).$ \Cref{2_part_exact_sequence} also shows that the restriction of $M_1$ to $S_{n-1}$ is $S_2^{(n-i,i-1)-}$ and so these modules have the same submodule structure.
\end{remark}

We give this submodule a special name:
\begin{definition}\label{S*(n-r_r)}
    For $n$ even and $i\geq 1, S_2^{*(n-i,i)}$ is the image of $\hat{\Theta}_{i-1},$ a submodule of $S_2^{(n-i,i)}$ isomorphic to a quotient of $S_2^{(n-i+1,i-1)},$ also known as $M_1.$
\end{definition}

We illustrate the exact sequence above and the role of $S_2^{*(n-i,i)}$ when $n=14.$ In the diagrams below, the vertices represent submodules of $S_2^\lambda,$ and the nodes are labelled by the dimension of the submodule. The directed edge from a vertex $u$ to a vertex $v$ indicates that the submodule associated to $u$ is a maximal submodule of the submodule associated to $v,$ and the edges are labelled by their irreducible subquotients.

We show the submodule lattices for $S_2^{(14)}, S_2^{(13,1)}, \ldots, S_2^{(7,7)}.$ In the submodule lattice for $S_2^{(n-i,i)}$ we highlight in blue the submodule $S_2^{*(n-i,i)}$ and in red the quotient $S_2^{(n-i,i)}/S_2^{*(n-i,i)} \isom S_2^{*(n-i-1,i+1)}.$

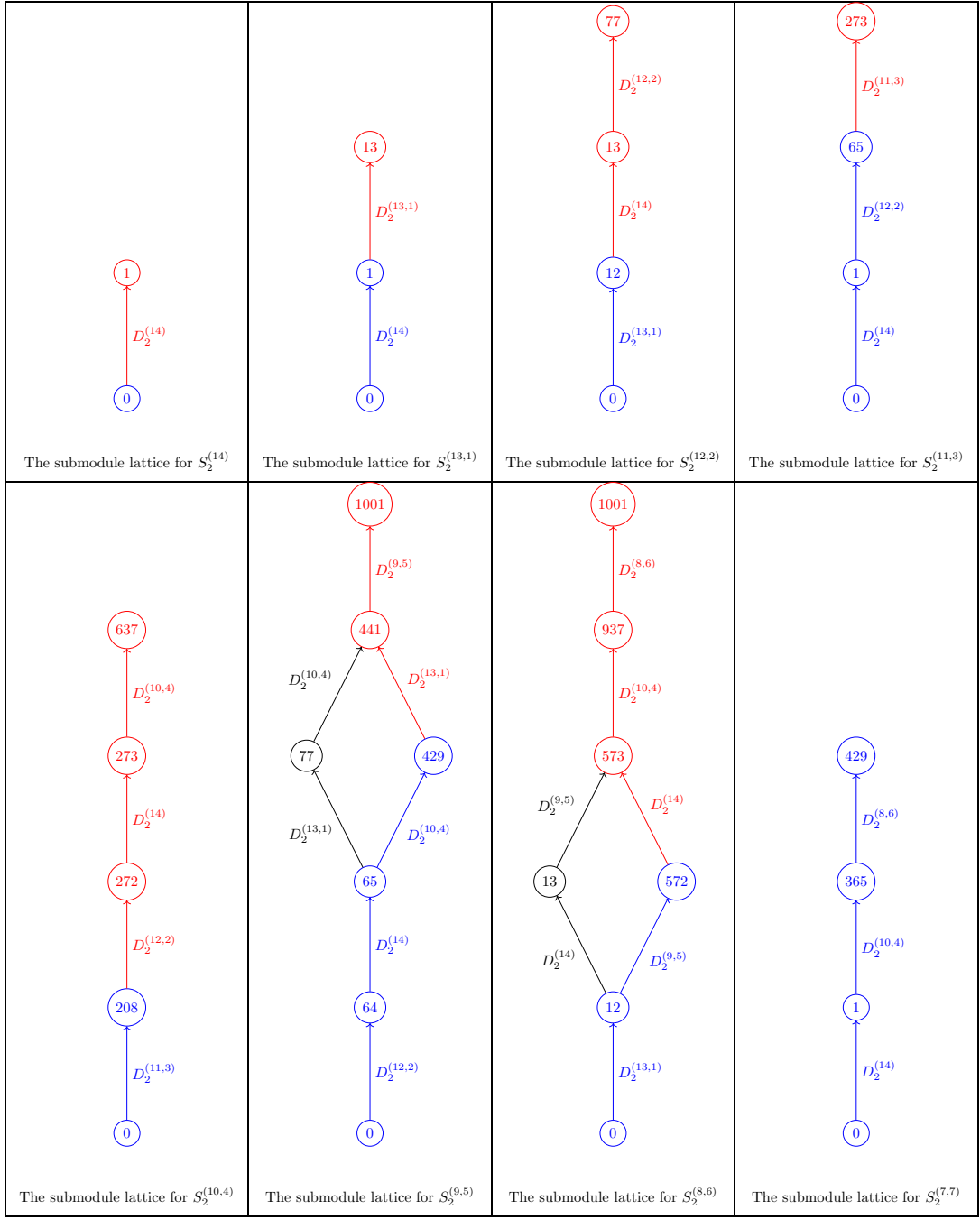
\resizebox{0.9\textwidth}{!}{
\begin{tabular}{|c|c|c|c|}
     \hline
\begin{tikzpicture}[main/.style = {draw, circle}, scale = 3]
\node at (0,-1/2) {The submodule lattice for $S^{( 14 )}_2$};

\node[main, blue] (0) at (0,0) {$0$};
\node[main, red] (1) at (0,1) {$1$};

\draw[->] (0) edge [red, "$D_2^{(14)}$"', pos = 0.5] (1);

\end{tikzpicture}
     &
\begin{tikzpicture}[main/.style = {draw, circle}, scale = 3]
\node at (0,-1/2) {The submodule lattice for $S^{( 13, 1 )}_2$};

\node[main, blue] (0) at (0,0) {$0$};
\node[main, blue] (1) at (0,1) {$1$};
\node[main, red] (13) at (0,2) {$13$};

\draw[->] (0) edge [blue, "$D_2^{(14)}$"', pos = 0.5] (1);

\draw[->] (1) edge [red, "$D_2^{(13,1)}$"', pos = 0.5] (13);

\end{tikzpicture}
     &
\begin{tikzpicture}[main/.style = {draw, circle}, scale = 3]
\node at (0,-1/2) {The submodule lattice for $S^{( 12, 2 )}_2$};

\node[main, blue] (0) at (0,0) {$0$};
\node[main, blue] (12) at (0,1) {$12$};
\node[main, red] (13) at (0,2) {$13$};
\node[main, red] (77) at (0,3) {$77$};

\draw[->] (0) edge [blue, "$D_2^{(13,1)}$"', pos = 0.5] (12);

\draw[->] (12) edge [red, "$D_2^{(14)}$"', pos = 0.5] (13);

\draw[->] (13) edge [red, "$D_2^{(12,2)}$"', pos = 0.5] (77);

\end{tikzpicture}
     &
\begin{tikzpicture}[main/.style = {draw, circle}, scale = 3]
\node at (0,-1/2) {The submodule lattice for $S^{( 11, 3 )}_2$};

\node[main, blue] (0) at (0,0) {$0$};
\node[main, blue] (1) at (0,1) {$1$};
\node[main, blue] (65) at (0,2) {$65$};
\node[main, red] (273) at (0,3) {$273$};

\draw[->] (0) edge [blue, "$D_2^{(14)}$"', pos = 0.5] (1);

\draw[->] (1) edge [blue, "$D_2^{(12,2)}$"', pos = 0.5] (65);

\draw[->] (65) edge [red, "$D_2^{(11,3)}$"', pos = 0.5] (273);

\end{tikzpicture}
     \\ 
     \hline
\begin{tikzpicture}[main/.style = {draw, circle}, scale = 3]
\node at (0,-1/2) {The submodule lattice for $S^{( 10, 4 )}_2$};

\node[main, blue] (0) at (0,0) {$0$};
\node[main, blue] (208) at (0,1) {$208$};
\node[main, red] (272) at (0,2) {$272$};
\node[main, red] (273) at (0,3) {$273$};
\node[main, red] (637) at (0,4) {$637$};

\draw[->] (0) edge [blue, "$D_2^{(11,3)}$"', pos = 0.5] (208);

\draw[->] (208) edge [red, "$D_2^{(12,2)}$"', pos = 0.5] (272);

\draw[->] (272) edge [red, "$D_2^{(14)}$"', pos = 0.5] (273);

\draw[->] (273) edge [red, "$D_2^{(10,4)}$"', pos = 0.5] (637);

\end{tikzpicture}
     &
\begin{tikzpicture}[main/.style = {draw, circle}, scale = 3]
\node at (0,-1/2) {The submodule lattice for $S^{( 9, 5 )}_2$};

\node[main, blue] (0) at (0,0) {$0$};
\node[main, blue] (64) at (0,1) {$64$};
\node[main, blue] (65) at (0,2) {$65$};
\node[main] (77) at (-1/2,3) {$77$};
\node[main, blue] (429) at (1/2,3) {$429$};
\node[main, red] (441) at (0,4) {$441$};
\node[main, red] (1001) at (0,5) {$1001$};

\draw[->] (0) edge [blue, "$D_2^{(12,2)}$"', pos = 0.5] (64);

\draw[->] (64) edge [blue, "$D_2^{(14)}$"', pos = 0.5] (65);

\draw[->] (65) edge ["$D_2^{(13,1)}$", pos = 0.5] (77);
\draw[->] (65) edge [blue, "$D_2^{(10,4)}$"', pos = 0.51] (429);

\draw[->] (77) edge ["$D_2^{(10,4)}$", pos = 0.5] (441);

\draw[->] (429) edge [red, "$D_2^{(13,1)}$"', pos = 0.5] (441);

\draw[->] (441) edge [red, "$D_2^{(9,5)}$"', pos = 0.5] (1001);

\end{tikzpicture}
     &
\begin{tikzpicture}[main/.style = {draw, circle}, scale = 3]
\node at (0,-1/2) {The submodule lattice for $S^{( 8, 6 )}_2$};

\node[main, blue] (0) at (0,0) {$0$};
\node[main, blue] (12) at (0,1) {$12$};
\node[main] (13) at (-1/2,2) {$13$};
\node[main, blue] (572) at (1/2,2) {$572$};
\node[main, red] (573) at (0,3) {$573$};
\node[main, red] (937) at (0,4) {$937$};
\node[main, red] (1001) at (0,5) {$1001$};

\draw[->] (0) edge [blue, "$D_2^{(13,1)}$"', pos = 0.5] (12);

\draw[->] (12) edge ["$D_2^{(14)}$", pos = 0.5] (13);
\draw[->] (12) edge [blue, "$D_2^{(9,5)}$"', pos = 0.5] (572);

\draw[->] (13) edge ["$D_2^{(9,5)}$", pos = 0.5] (573);

\draw[->] (572) edge [red, "$D_2^{(14)}$"', pos = 0.5] (573);

\draw[->] (573) edge [red, "$D_2^{(10,4)}$"', pos = 0.5] (937);

\draw[->] (937) edge [red, "$D_2^{(8,6)}$"', pos = 0.5] (1001);

\end{tikzpicture}
     &
\begin{tikzpicture}[main/.style = {draw, circle}, scale = 3]
\node at (0,-1/2) {The submodule lattice for $S^{( 7, 7 )}_2$};

\node[main, blue] (0) at (0,0) {$0$};
\node[main, blue] (1) at (0,1) {$1$};
\node[main, blue] (365) at (0,2) {$365$};
\node[main, blue] (429) at (0,3) {$429$};

\draw[->] (0) edge [blue, "$D_2^{(14)}$"', pos = 0.5] (1);

\draw[->] (1) edge [blue, "$D_2^{(10,4)}$"', pos = 0.5] (365);

\draw[->] (365) edge [blue, "$D_2^{(8,6)}$"', pos = 0.5] (429);

\end{tikzpicture}
     \\
     \hline
\end{tabular}
}

\newpage
\section{A Filtration of Hook Specht Modules by 2-Part Specht Modules}\label{A Filtration of Hook Specht Modules by 2-Part Specht Modules}

Peel proved that if $n-r \geq r,$ then there is a submodule of $S_2^{(n-r,1^r)}$ isomorphic to $S_2^{(n-r,r)}$ \cite[Theorem 4]{peel1971hook}. Sutton extended this result and proved, using Khovanov-Lauda-Rouqier (KLR) algebras, that $S_2^{(n-r,1^r)}$ has a \emph{filtration} via $2$-part Specht modules \cite[Theorems 3.7, 3.10]{sutton2017graded}. We provide an elementary proof of the filtration below, that is, we construct the filtration explicitly.

\begin{theorem}\label{FPL}
    Let $\lambda = (n-r,1^r) \vdash n$ be a hook partition of $n,$ with $0 \leq r \leq n-r.$ Then $S^\lambda_2$ has a filtration via Specht modules labelled by the $2$-row partitions (from bottom to top) $S_2^{(n-r,r)}, S_2^{(n-r+2, r-2)}, \ldots, S_2^{(n-r+2k, r-2k)}, \ldots.$
\end{theorem}
\begin{proof}

    The main part of the proof is to build the submodules $0 = M_{-1} \leq M_0 \leq \ldots \leq M_{\floor{r/2}} = S_2^{(n-r,1^r)},$ where each quotient $M_k/M_{k-1}$ is isomorphic to $S_2^{(n-r+2k,r-2k)}.$

    We first define $M_k$ as a vector subspace of $S^{(n-r,1^r)}$ in the following way.

    If $r > 2k,$ let $\bar{X}_k$ be the set of ordered $(2r-k)$-tuples \[X = (a_1, b_1, \ldots, a_{r-2k}, b_{r-2k}, c_1, d_1, e_1, \ldots, c_{k}, d_k, e_k),\] consisting of different elements of $\{1, \ldots, n\}.$ Then for each $X \in \Bar{X_k},$ let $T_X$ be the set of column-standard $(n-r,1^r)$-tableaux $t$ such that the first column of $t$ contains
    \begin{itemize}
        \item $a_1$ and $b_1,$
        \item exactly one of $a_i$ or $b_i$ for each $i = 2, \ldots, r-2k,$
        \item exactly two of $c_j, d_j, e_j,$ for each $j = 1, \ldots, k.$
    \end{itemize}
    
    Then $M_k$ is the vector subspace of $S_2^{(n-r,1^r)}$ generated by $\{\sum_{t \in T_X} e_t \, | \, X \in \Bar{X}_k\}.$

    We need to show that $M_{k-1} \subseteq M_k.$ First, fix an \[X = (a_1, b_1, \ldots, a_{r-2k+2}, b_{r-2k+2}, c_1, d_1, e_1, \ldots, c_{k-1}, d_{k-1}, e_{k-1}) \in \Bar{X}_{k-1},\] we want to show that $\sum_{t \in T_X} e_t \in M_k.$ Let
    
    \begin{itemize}
        \item $Y = (a_1, b_1, \ldots, a_{r-2k}, b_{r-2k}, c_1, d_1, e_1, \ldots, c_{k-1}, d_{k-1}, e_{k-1}, a_{r-2k+1}, b_{r-2k+1}, a_{r-2k+2})$
        \item $Y' = (a_1, b_1, \ldots, a_{r-2k}, b_{r-2k}, c_1, d_1, e_1, \ldots, c_{k-1}, d_{k-1}, e_{k-1}, a_{r-2k+1}, b_{r-2k+1}, b_{r-2k+2}),$
    \end{itemize}
    
     then it is a straightforward calculation to show that \[\sum_{t \in T_X}e_t = \sum_{t \in T_Y}e_t + \sum_{t\in T_{Y'}}e_t,\] and hence $M_{k-1}$ is a vector subspace of $M_k.$

    If $r = 2k,$ we define $M_k$ in a way similar to above. Let $\Bar{X}_k$ be the set of ordered $(3k+1)$-tuples \[X = (a_1, c_1, d_1, e_1, \dots, c_k, d_k, e_k),\] consisting of different elements of $\{1, \ldots, n\}.$ Then for each $X \in \Bar{X}_k,$ let $T_X$ be the set of column-standard $(n-r,1^r)$-tableaux $t$ such that the first column of $t$ contains 
    \begin{itemize}
        \item $a_1,$
        \item exactly two of $c_j, d_j, e_j$ for each $j = 1, \ldots, k.$
    \end{itemize}

    Then $M_k$ is the vector subspace of $S^{(n-r,1^r)}$ generated by $\{\sum_{t \in T_X} e_t | X \in \Bar{X}_k\}.$ As before, we need to show that $M_{k-1} \subseteq M_k.$ If \[X = (a_1, b_1, a_2, b_2, c_1, d_1, e_1, \ldots, c_{k-1}, d_{k-1}, e_{k-1}) \in \Bar{X}_k,\] we want to show that $\sum_{t \in T_X} e_t \in M_k.$ Let
    \begin{itemize}
        \item $Y = (a_2, c_1, d_1, e_1, \ldots, c_{k-1}, d_{k-1}, e_{k-1}, a_1, b_1, b_2),$
        \item $Z = (b_2, c_1, d_1, e_1, \ldots, c_{k-1}, d_{k-1}, e_{k-1}, a_1, b_1, a_2),$
    \end{itemize}

and then it is a straightforward calculation to show that \[\sum_{t \in T_X} e_t = \sum_{t \in T_{Y}}e_t + \sum_{t \in T_{Z}}e_t\,\] and hence $M_{k-1} \subseteq M_k.$

It is clear that each $M_k$ is actually a submodule of $S_2^{(n-r,1^r)},$ as $\Bar{X}_k$ is fixed under the natural action of $S_n$ on tuples of elements $\{1, \ldots, n\},$ and for all $\pi \in S_n,$ \[\pi \sum_{t \in T_X} e_t = \sum_{t \in T_{\pi X}}e_t,\]  so the given spanning set for $M_k$ is closed under the action of $S_n,$ and hence so is $M_k.$
    
Now we construct homomorphisms $\theta_k, 0 \leq k \leq r/2,$ where for each $k,$ \[\theta_k: S_2^{(n-r+2k, r-2k)}\to \frac{M_k}{M_{k-1}},\] and we set $M_{-1} =0.$

Then, as long as $M_{\floor{r/2}} = S_2^{(n-r,1^r)},$ we'll have a filtration of $S_2^{(n-r,1^r)}$ in which the factors are quotients of $\{S_2^{(n-r+2k,r-2k)} \,|\, 0 \leq k \leq 2r \}.$ As the dimensions add up \cite[Corollary 2.7]{murphy1982submodule}, each $\theta_k$ must be surjective and hence we get the desired result.

We define $\theta_k$ in the following way. For an $(n-r+2k, r-2k)$-tableau $s,$ let $a_i$ be the entry in the $i$th column first row, and $b_i$ be the entry in the $i$th column second row. Then if $r > 2k$ set $X = (a_1, b_1, \ldots, a_{r-2k}, b_{r-2k}, c_1, d_1, e_1, \ldots, c_k, d_k, e_k)$ for some choice of $c_j, d_j, e_j$ not equal to $a_1, b_1, \ldots, a_{r-2k}, b_{r-2k}.$ If $r = 2k,$ set $X = (c_1, d_1, e_1, \ldots, c_k, d_k, e_k)$ for some choice of $c_j, d_j, e_j.$ Then define the maps \[\theta_k(e_s) := \sum_{t \in T_X} e_t + M_{k-1},\] and extend linearly.

First we will show that this is independent of the choice of $c_j, d_j, e_j,$ then we will show that $\theta_k$ is a $KS_n$-module homomorphism, and that $M_{\floor{r/2}} = S_2^{(n-r,1^r)}.$

To show that $\theta_k$ is independent of the choice of $c_j, d_j, e_j,$ we fix the $(n-r+2k,r-2k)$-tableau $s,$ and fix $X = (a_1, b_1, \ldots, a_{r-2k}, b_{r-2k}, c_1, d_1, e_1, \ldots, c_k, d_k, e_k)$ for some choice of $c_j, d_j, e_j$ not equal to $a_1, b_1, \ldots, a_{r-2k}, b_{r-2k}.$ It is clear that relabelling $c_j, d_j$ and $e_j$ for some fixed $j$ has no impact on the image of $\theta_k,$ and neither does relabelling some $\{c_j, d_j, e_j\}$ with $\{c_{j'}, d_{j'}, e_{j'}\}$

It suffices to show that a different choice of say $e_1$ to be some $l\leq n, l \not \in X$ is well-defined (noting that such an $l$ exists as $k > 0$ and $2r-k <n$).  

For $r < 2k,$ we do the following:

Let $\theta_k^1(e_s)$ be the sum as defined above, and let $\theta_k^2(e_s)$ be the same sum but with $l$ in place of $e_1.$ Now define $s'$ to be the $(n-r+2(k-1), r-2(k-1))$-tableau which has $a_i, b_i$ in the first $r-2k$ columns, $c_1, d_1$ in the $r-2k+1$st column, and $e_1, l$ in the $r-2k+2$nd column. Then $\theta_k^1(e_s) + \theta_k^2(e_s) = \theta_{k-1}(e_{s'}).$

When $r =2k,$ the calculation is slightly different:

Let $\theta_k^1(e_s)$ be the sum as defined above, and let $\theta_k^2(e_s)$ be the same sum but with $l$ in place of $e_1.$ Now define $s'$ to be the $(n-r+2(k-1), r-2(k-1))$-tableau which keeps $a_1$ and has $b_1 = c_1, a_2 = e_1, b_2 = l,$ and let $s''$ be the $(n-r+2(k-1), r-2(k-1))$-tableau which keeps $a_1$ and has $b_1 = d_1, a_2 = e_1, b_2 = l.$ Then $\theta_k^1(e_s) + \theta_k^2(e_s) = \theta_{k-1}(e_{s'}) + \theta_{k-1}(e_{s''}).$ 

Hence each $\theta_k(e_s)$ is well-defined, and we extend $\theta_k$ linearly. Next, we shall show that $\theta_k$ is well-defined as a $KS_n$-homomorphism. First we do this in the case that $r \neq 2k.$ Clearly if $\sigma$ is a column stabiliser of $s,$ then $(1 + \sigma)\theta_k(e_s) = 0.$ Now, as $S_2^\lambda$ is cyclic, we can choose $c_1 = a_{r-2k + 1},$ and fix $\Tilde{t}$ to be the column-standard tableau that has in the first column $b_1, a_i$ for $i = 1, \ldots, r-2k,$ and $c_j, d_j$ for $j= 1, \ldots, k.$ For $i = 2, \ldots, r-2k$ define $\hat{h}_i = (1 + (a_i, b_i))$ and for $j = 1, \ldots, k$ define $\check{h}_j = (1 + (c_j, d_j, e_j) + (c_j, e_j, d_j)).$ Then,

\begin{align*}
    \theta_k(e_s) &= \sum_{t\in T_X}e_t + M_{k-1}\\
    &= \prod_{i=2}^{r-2k}\hat{h}_i\prod_{j=1}^k\check{h}_j e_{\Tilde{t}} + M_{k-1}.
\end{align*}

Note that the collection of $\hat{h}_i$ and $\check{h}_j$ all pairwise commute. 

There are $6$ types of Garnir relations:
\begin{enumerate}
    \item The Garnir relation on the set $\{a_1, b_1, a_2\}$ which has the Garnir element $g_{1,a} := 1 + (b_1, a_2) + (a_1, b_1, a_2).$
    \item The Garnir relation on the set $\{b_1, a_2, b_2\}$ which has the Garnir element $g_{1,b} := 1 + (b_1, b_2) + (b_1, a_2, b_2).$
    \item For $2 \leq i \leq r-2k-1,$ the Garnir relation on the set $\{a_i, b_i, a_{i+1}\}$ which has the Garnir element $g_{i, a} := 1 + (b_i, a_{i+1}) + (a_i, b_i, a_{i+1}).$
    \item For $2 \leq i \leq r-2k-1,$ the Garnir relation on the set $\{b_i, a_{i+1}, b_{i+1}\}$ which has the Garnir element $g_{i,b} := 1 + (b_i, b_{i+1}) + (b_i, a_{i+1}, b_{i+1}).$ 
    \item If $k\neq 0$ or $n \neq 2r,$ the Garnir relation on the set $\{a_{r-2k}, b_{r-2k}, a_{r-2k+1}\}$ which has the Garnir element $g_{r-2k, a} := 1 + (b_{r-2k}, a_{r-2k+1}) + (a_{r-2k}, b_{r-2k}, a_{r-2k+1}).$
    \item For $i > r-2k,$ the Garnir relation on the set $\{a_i,a_{i+1}\}$ which has the Garnir element $1 + (a_i, a_{i+1}).$
\end{enumerate}

We will show that each of these Garnir elements also eliminates $\theta_k(e_s)$ and hence the homomorphism is well defined.

\begin{enumerate}
    \item Note that $g_{1,a}$ commutes with every $\hat{h}_i$ and $\check{h}_j$ except for $\hat{h}_2 = (1 + (a_2, b_2)),$ and that $g_{1,a}\hat{h}_2$ commutes with every other $\hat{h}_i$ and $\check{h}_j.$ Then

    \begin{align*}
        g_{1,a}\hat{h}_2 &= (1 + (b_1, a_2) + (a_1,b_1,a_2))(1 + (a_2,b_2))\\
        &= (1 + (b_1, a_2) + (a_1, b_1, a_2) + (a_2, b_2) + (b_1, a_2, b_2) + (a_1, b_1, a_2, b_2)),
    \end{align*}

    Let $e_{t/x}$ be the polytabloid for the tableau which has in the first column \[\{a_1, b_1, b_2, a_2, \ldots, a_{r-2k}, c_1, d_1, \ldots, c_k, d_k\}\setminus\{x\}.\]

    Then \[g_{1,a}\hat{h}_2 e_{\Tilde{t}} = \sum_{x \not\in \{a_1, b_1, b_2, a_2\}} e_{t/x},\]

    and for each $x,$ if $x = a_i,$ then $\hat{h}_i g_{1,a}\hat{h}_2 e_{t/x}  = 0,$ and if $x = c_j$ or $d_j,$ then $\check{h}_j g_{1,a}\hat{h}_2 e_{t/x}  = 0.$

    \item Similar to above.

    \item Note that for $2 \leq i \leq r-2k,$ $g_{i,a}$ commutes with every $\hat{h}_j$ except for $\hat{h}_i$ and $\hat{h}_{i+1}.$ Then

    \begin{align*}
        g_{i, a}\hat{h}_i\hat{h}_{i+1} &=\\
        &= (1 + (b_i, a_{i+1}) + (a_i, b_i, a_{i+1}))(1 + (a_i, b_i))(1 + (a_{i+1}, b_{i+1}))\\
        &= 1 + (b_i, a_{i+1}) + (a_i, b_i, a_{i+1}) + (a_i, b_i) + (a_i, a_{i+1}, b_i) + (a_i, a_{i+1}) \\
        &+ (a_{i+1}, b_{i+1}) + (b_i, a_{i+1}, b_{i+1}) + (a_i, b_i, a_{i+1}, b_{i+1}) + (a_i, b_i)(a_{i+1}, b_{i+1}) \\
        &+ (a_i, a_{i+1}, b_{i+1}, b_i) + (a_i, a_{i+1}, b_{i+1}),
    \end{align*}    
    and
    \begin{align*}
        g_{i, a}\hat{h}_i\hat{h}_{i+1}e_{\Tilde{t}} = 0.
    \end{align*}
    
    \item Similar to above.

    \item Note that $g_{r-2k,a}$ commutes with every $\hat{h}_{i}$ except for $\hat{h}_{r-2k}$ and with every $\check{h}_j$ except for $\check{h}_1,$ as $a_{r-2k+1} = c_1.$ Then

    \begin{align*}
        &g_{r-2k,a}\hat{h}_{r-2k}\check{h}_1 = \\
        &= (1 + (b_{r-2k}, c_1) + (a_{r-2k}, b_{r-2k}, c_1))(1 + (a_{r-2k}, b_{r-2k})(1 + (c_1, d_1, e_1) + (c_1, e_1, d_1)))\\
        &= 1 + (b_{r-2k}, c_1) + (a_{r-2k, b_{r-2k}, c_1}) + (a_{r-2k}, b_{r-2k}) + (a_{r-2k}, c_1, b{r-2k}) + (a_{r-2k}, c_1) \\&+ (c_1, d_1, e_1) + (b_{r-2k}, c_1, e_1, d_1) + (a_{r-2k}, b_{r-2k}, c_1, d_1, e_1) + (a_{r-2k}, b_{r-2k})(c_1, d_1, e_1) \\& + (a_{r-2k}, c_1, d_1, e_1, b_{r-2k}) + (a_{r-2k}, c_1, d_1, e_1) + (c_1, e_1, d_1) + (b_{r-2k}, c_1, e_1, d_1) \\&+ (a_{r-2k}, b_{r-2k}, c_1, e_1, d_1) + (a_{r-2k}, b_{r-2k})(c_1, e_1, d_1) + (a_{r-2k}, c_1, e_1, d_1, b_{r-2k}) \\&+ (a_{r-2k}, c_1, e_1, d_1).
    \end{align*}
    and
    \begin{align*}
        g_{r-2k, a}\hat{h}_{r-2k}\check{h}_1 e_t = 0.
    \end{align*}

    \item The cases where $\{a_i, a_{i+1}\} \cap \{c_j, d_j, e_j\}_{j=1}^k = 1$ or $2$ have already been dealt with. If $(a_i, a_{i+1})$ permutes two entries which are not in $X,$ then $(1+(a_i, a_{i+1}))$ commutes with every $\hat{h}_i$ and every $\check{h}_j,$ and $(1+(a_i, a_{i+1}))e_{\tilde{t}} = 0$ so $(1+(a_i, a_{i+1}))\theta_k(e_s) = 0$ as required.
    
\end{enumerate}

In the case $r = 2k,$ it is sufficient to check the Garnir relation $1 +(a_1, c_1)$ as we have already shown that $(1 + (e_1, l))\theta_k(e_s) \in M_{k-1}.$ One can verify that $(1 + (a_1, c_1)) \theta_k(e_s) \in M_{k-1}.$

Hence $\theta_k$ is a $KS_n$ homomorphism, because the construction is invariant under the action of $S_n.$ That is, for all $\pi \in S_n, \theta_k(\pi e_s) = \theta_k(e_{\pi s}) = \pi \theta_k(e_s).$

All that remains is to show that $M_{\floor{r/2}} = S_2^{(n-r, 1^r)}.$ We split this into the case that $r = 2k$ and $r = 2k+1$.

Fix $k$ such that $r = 2k+1.$ We want to show that $M_k = S_2^{(n-r, 1^r)}.$ It suffices to show that $e_t \in M_k,$ for some $(n-r, 1^r)$-tableau $t$.

Fix a subset $X = \{a_1, b_1, c_1, d_1, e_1, \ldots, c_k, d_k, e_k\}.$ Then $\sum_{t \in T_X} e_t \in M_k$ by definition. Fix $\tilde{t}$ to be a tableau which has in its first column $\{a_1, b_1, c_1, d_1, \ldots, c_k, d_k\},$ then \[\sum_{t\in T_X} e_t = \prod_{j=1}^k \check{h}_j e_{\Bar{t}}.\]

Now write $\Tilde{h}_j = (1 + (b_1, c_j) + (b_1, d_j))$ and consider the expression \[\Tilde{h}_k \cdots \Tilde{h}_1 \prod_{j=1}^k \check{h}_j e_{\Bar{t}}.\] Before simplifying the expression, we shall observe that $\Tilde{h}_j$ commutes with all $\check{h}_{j'}$ except $j' = j,$ and the fact that \[\Tilde{h}_j\check{h}_j e_{\Bar{t}} = e_{\Bar{t}}.\]

Hence \begin{align*}
    &\Tilde{h}_k \cdots \Tilde{h}_1 \prod_{j=1}^k \check{h}_j e_{\Bar{t}} \\
    =& \Tilde{h}_k \cdots \Tilde{h}_2 (\Tilde{h}_1\check{h}_1) \prod_{j=2}^k \check{h}_j e_{\Bar{t}}\\
    =& \Tilde{h}_k \cdots \Tilde{h}_2 \prod_{j=2}^k \check{h}_j (\Tilde{h}_1\check{h}_1)e_{\Bar{t}}\\
    =& \Tilde{h}_k \cdots \Tilde{h}_2 \prod_{j=2}^k \check{h}_j e_{\Bar{t}}\\
    &\cdots\\
    =& e_{\Bar{t}}.
\end{align*}

So if $r = 2k+1,$ then $M_k = S_2^{(n-r,1^r)}.$

Now fix $k$ such that $r = 2k.$ As before, we want to show that $M_k = S_2^{(n-r, 1^r)}.$ It suffices to show that $e_t \in M_k,$ for some $t$ an $(n-r, 1^r)$-tableau. 

Fix $X = \{a, c_1, d_1, e_1, \ldots, c_k, d_k, e_k\}.$ Let $Y = Y_{k-1} \cup \{c_1, d_1, e_1, \ldots, c_{k-1}, d_{k-1}, e_{k-1}\}$ (so we consider $a_1 = a, b_1 = e_k, a_2 = c_k, b_2 = d_k$). Write $\Bar{t}$ for the tableau which has in the first column $a, c_1, d_1, \ldots, c_{k-1}, d_{k-1}, c_k, e_k.$ Then \[\sum_{t \in T_X} e_t + \sum_{t \in T_Y}e_t = ((c_k, d_k) + (c_k,d_k,e_k) + (c_k,e_k,d_k)) \prod_{j = 1}^{k-1} e_{\Bar{t}}.\]

Noting that $((c_k, d_k) + (c_k,d_k,e_k) + (c_k,e_k,d_k))$ commutes with $\check{h}_j$ for all $j$ from $1$ to $k-1,$ and writing $t^*$ for the tableau which has in its first column $\{a, c_j, d_j\}$ for $j = 1, \ldots, k,$ one sees that \[((c_k, d_k) + (c_k,d_k,e_k) + (c_k,e_k,d_k)) \prod_{j = 1}^{k-1} e_{\Bar{t}} = \prod_{j = 1}^{k-1} e_{t^*}.\]

Writing $\Tilde{h}_j = 1 + (d_k, c_j) + (d_k, d_j)$ for $j = 1, \ldots, k-1,$ and simplifying in a way similar to the previous case, we get \[\Tilde{h}_{k-1} \cdots \Tilde{h}_{1} \prod_{j = 1}^{k-1} e_{t^*} = e_{t^*}.\]

So if $r = 2k,$ then $M_k = S_2^{(n-r,1^r)}.$

So we have that \[\dim(S_2^{(n-r,1^r)}) = \sum_{k=0}^{r/2}\dim(M_k / M_{k-1}) \leq \sum_{k = 0}^{r/2} \dim(S_2^{(n-r+2k, r-2k)}) = \dim(S_2^{(n-r,1^r)}),\] where the first equality comes from the filtration, the inequality comes from the homomorphisms, and the final equality comes from \cite[Corollary 2.7]{murphy1982submodule}.

Hence $S_2^{(n-r, 1^r)}$ has a filtration via hook Specht modules labeled by 2-part partitions.
\end{proof}

\begin{example}\label{example_of_filtration}
    $S_2^{(6,1^4)}$ has a filtration, from bottom to top, via $S_2^{(6,4)}, S_2^{(8,2)}$ and $S_2^{(8)}.$ More explicitly, we have the submodules $0 = M_{-1} \leq M_0 \leq M_1 \leq M_2 = S_2^{(6,1^4)},$ where 
    \begin{enumerate}
        \item $M_0$ is generated by \[x = (1+(34))(1+(56))(1+78)e\, \young(14689<10>,2,3,5,7),\]
        \item $M_1$ is generated by \[y = (1+(34))(1+(567)+(576))e\, \young(14789<10>,2,3,5,6),\]
        \item $M_2$ is generated by \[z = (1+(234)+(243))(1+(567)+(576))e\, \young(14789<10>,2,3,5,6)\]
    \end{enumerate}

    and we have the following homomorphisms:
    \begin{enumerate}
        \item $\theta_0 : S_2^{(6,4)} \to M_0,$ \[\theta_0(e\, \young(13579<10>,2468)) = x,\]
        \item $\theta_1 : S_2^{(8,2)} \to M_1/M_0,$ \[\theta_1(e\, \young(135789<10>,24)) = y + M_0,\]
        \item $\theta_2 : S_2^{(10)} \to M_2/M_1,$ \[\theta_2(e\, \young(123456789<10>)) = z + M_1.\]
    \end{enumerate}

    We also have the following submodule lattice. In the diagram below, the vertices represent submodules of $S_2^{(6,1^4)},$ and the nodes are labelled by the dimension of the submodule, with an additional subscript if there are multiple submodules of the same dimension. The directed edge from a vertex $u$ to a vertex $v$ indicates that the submodule associated to $u$ is a maximal submodule of the submodule associated to $v.$ The submodules $M_{-1}, M_0, M_1$ and $M_2$ are all labelled, and the subquotients $M_0/M_{-1}, M_1/M_0$ and $M_2/M_1$ have nodes and edges coloured purple, blue, and red respectively. For these subquotients, we also label the edges by which irreducible representation they are isomorphic to.

    \scalebox{0.8}{
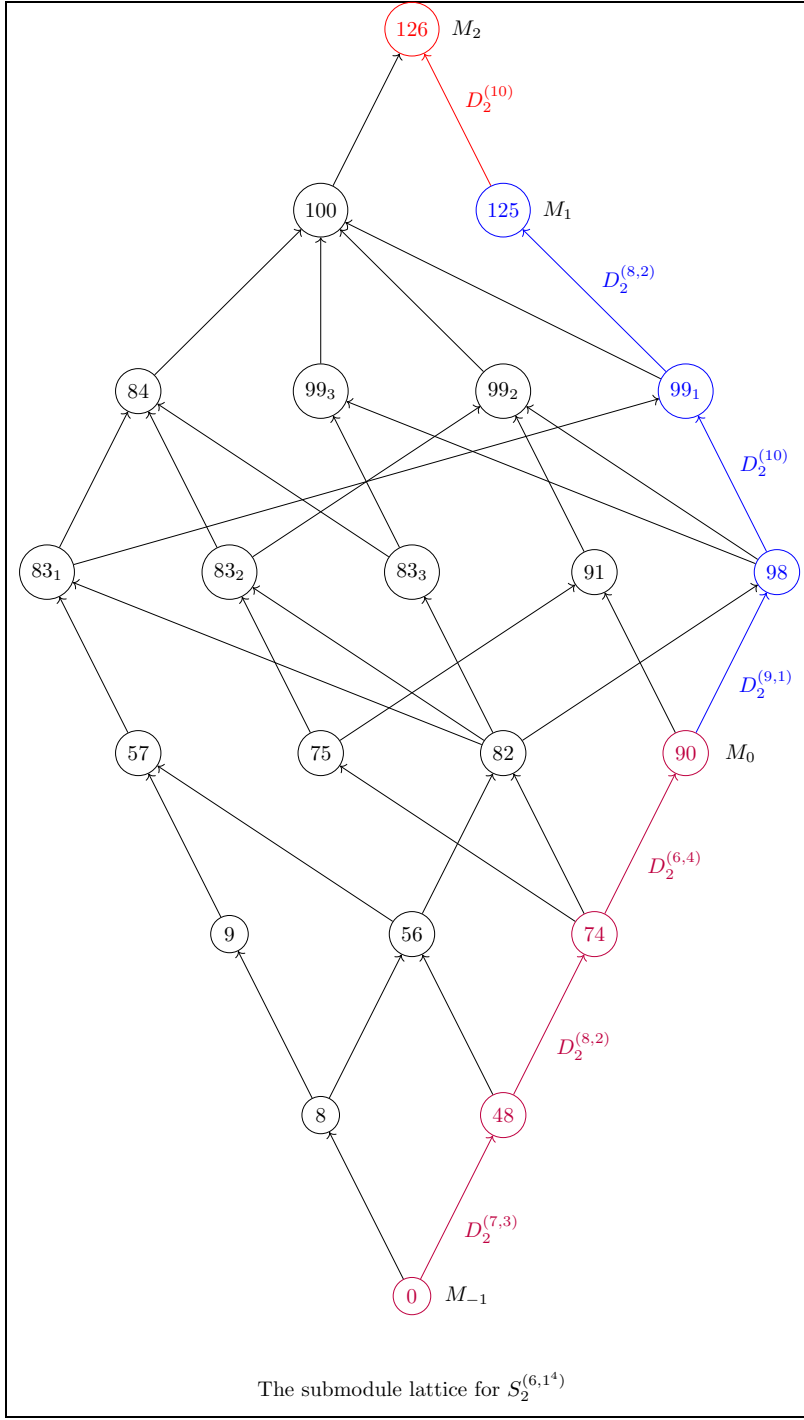
    \begin{tabular}{|c|}
        \hline
        \begin{tikzpicture}[main/.style = {draw, circle}, scale = 3]
        \node at (0,-1/2) {The submodule lattice for $S^{( 6, 1^4)}_2$};
        
        \node[main, purple] (0) at (0,0) {$0$};
        \node[main] (8) at (-1/2,1) {$8$};
        \node[main] (9) at (-1,2) {$9$};
        \node[main, purple] (48) at (1/2,1) {$48$};
        \node[main] (56) at (0,2) {$56$};
        \node[main] (57) at (-3/2,3) {$57$};
        \node[main, purple] (74) at (1,2) {$74$};
        \node[main] (75) at (-1/2,3) {$75$};
        \node[main] (82) at (1/2,3) {$82$};
        \node[main] (83_1) at (-2,4) {$83_1$};
        \node[main] (83_2) at (-1,4) {$83_2$};
        \node[main] (83_3) at (0,4) {$83_3$};
        \node[main] (84) at (-3/2,5) {$84$};
        \node[main, purple] (90) at (3/2,3) {$90$};
        \node[main] (91) at (1,4) {$91$};
        \node[main, blue] (98) at (2,4) {$98$};
        \node[main] (99_3) at (-1/2,5) {$99_3$};
        \node[main] (99_2) at (1/2,5) {$99_2$};
        \node[main, blue] (99_1) at (3/2,5) {$99_1$};
        \node[main] (100) at (-1/2,6) {$100$};
        \node[main, blue] (125) at (1/2,6) {$125$};
        \node[main, red] (126) at (0,7) {$126$};
        
        \draw[->] (0) edge (8);
        \draw[->] (0) edge [purple, "$D_2^{(7,3)}$"', pos = 0.5] (48);
        
        \draw[->] (8) edge (9);
        \draw[->] (8) edge (56);
        
        \draw[->] (9) edge (57);
        
        \draw[->] (48) edge (56);
        \draw[->] (48) edge [purple, "$D_2^{(8,2)}$"', pos = 0.5] (74);
        
        \draw[->] (56) edge (57);
        \draw[->] (56) edge (82);
        
        \draw[->] (57) edge (83_1);
        
        \draw[->] (74) edge (75);
        \draw[->] (74) edge (82);
        \draw[->] (74) edge [purple, "$D_2^{(6,4)}$"', pos = 0.5] (90);
        
        \draw[->] (75) edge (83_2);
        \draw[->] (75) edge (91);
        
        \draw[->] (82) edge (83_3);
        \draw[->] (82) edge (83_1);
        \draw[->] (82) edge (83_2);
        \draw[->] (82) edge (98);
        
        \draw[->] (83_1) edge (84);
        \draw[->] (83_1) edge (99_1);
        
        \draw[->] (83_2) edge (84);
        \draw[->] (83_2) edge (99_2);
        
        \draw[->] (83_3) edge (84);
        \draw[->] (83_3) edge (99_3);
        
        \draw[->] (84) edge (100);
        
        \draw[->] (90) edge (91);
        \draw[->] (90) edge [blue, "$D_2^{(9,1)}$"', pos = 0.5] (98);
        
        \draw[->] (91) edge (99_2);
        
        \draw[->] (98) edge (99_3);
        \draw[->] (98) edge [blue, "$D_2^{(10)}$"', pos = 0.5] (99_1);
        \draw[->] (98) edge (99_2);
        
        \draw[->] (99_1) edge (100);
        \draw[->] (99_1) edge [blue, "$D_2^{(8,2)}$"', pos = 0.5] (125);
        
        \draw[->] (99_2) edge (100);
        
        \draw[->] (99_3) edge (100);
        
        \draw[->] (100) edge (126);
        
        \draw[->] (125) edge [red, "$D_2^{(10)}$"', pos = 0.5] (126);

        \node at (0.3,7) {$M_2$};
        \node at (0.8,6) {$M_1$};
        \node at (1.8,3) {$M_0$};
        \node at (0.3, 0) {$M_{-1}$};
        \end{tikzpicture}\\
        \hline
    \end{tabular}
    }
    
\end{example}

The filtration above allows us to write out explicitly the decomposition numbers for hook Specht modules in characteristic $2.$

\begin{lemma}
    For $a, b$ integers, define 
    \[f(a,b) = 
    \begin{cases}
       1 & \text{ if } a \geq b \geq 0 \text{ and } a \supseteq_2 b \\  
       0 & \text{ otherwise. } 
    \end{cases}
    \]
    Then \[ [S_2^{(n-r, 1^r)} : D_2^{(n-j,j)}]  = \sum_{k \geq 0} f(n+1-2j, r-2k-j).\]
\end{lemma}
\begin{proof}
    This follows immediately from \Cref{FPL} and from \cite[Theorem 24.15]{James1978} (noting that the definition of $f$ in this paper differs slightly to that in the sited result) which states that $[S_2^{(n-r,r)} : D_2^{(n-j,j)}] = f(n+1-2j,r-j).$
\end{proof}

\begin{corollary}
    If $n$ is odd and $r-j$ is odd, then \[ [S_2^{(n-r, 1^r)} : D_2^{(n-j,j)}]  = 0.\]
\end{corollary}
\begin{proof}
    This follows immediately from the above as $n+1-2j$ is even and $r-2k-j$ is odd for all $k.$ It also follows by the fact that $(n-r,1^r)$ and $(n-j,j)$ have different $2$-cores \cite[Theorem 24.11]{James1978}.
\end{proof}

\newpage
\section{Another Filtration of Hook Specht Modules}\label{Another Filtration of Hook Specht Modules}

In the previous section, we gave a filtration of hook Specht modules $S_2^\lambda$ when $\lambda = (n-r,1^r) \vdash n$ for $n-r \geq r.$ 

In the case $n$ is odd, we get the following nice result about the dual of a hook Specht module.

\begin{theorem}\label{odd_self_dual}
    Let $n$ be odd and $1 \leq r < n.$ Then $S_2^{(n-r,1^r)}$ is self-dual. That is, $S_2^{(n-r,1^r)} \isom S_2^{(r+1, 1^{n-r-1})}.$
\end{theorem}
\begin{proof}
    We will prove this theorem by explicitly providing an isomorphism between these two Specht modules.
    
    Let $s$ be the standard tableau of shape $(n-r,1^r)$ which has $1, \ldots, r+1$ in the first column and $1, r+2, \ldots, n$ in the first row. Let $t$ be the conjugate of $s,$ that is, the $(r+1, 1^{n-r-1})$-tableau which has $1, \ldots, r+1$ in the first row and $1, r+2, \ldots, n$ in the first column.

    Let $h_1$ be the identity of $S_n$, $h_2 = (12)$ and for $l \in \{3, \ldots, r+1\},$ let $ h_l = (1 l 2) \in S_n.$ Let $h = \sum_{l = 1}^{r+1}h_l.$

    Let $g_{r+2}$ be the identity of $S_n$, and for $k \in \{1, \ldots, r+1\}$ let $g_k = (k, k+1, \ldots, r+2).$ Let $g = \sum_{k=1}^{r+2}g_k.$

    Let $e_s$ be the polytabloid for the tableau $s$. Define a map $f : S_2^{(n-r,1^r)} \to S_2^{(r+1, 1^{n-r-1})}$ by $f(e_s) = he_t$ and extend linearly under the action of $S_n.$

    To show that this is a homomorphism, we need to show that the map respects column permutations and the Garnir relations of $e_s.$ The column relations follow clearly. To show that the Garnir relations are satisfied, we need to show that $ghe_t = 0,$ and for $i \geq r+2, (1+(i,i+1))he_t = 0.$    
    
    We have that $ghe_t$ is a sum of $(r+2)(r+1)$ terms, which is even. If $k \leq l$ then $(g_kh_l + g_{l+1}h_k)e_t = 0,$  and if $k > l$ then $(g_kh_l + g_lh_{k-1})e_t = 0,$ so the terms of $ghe_t$ cancel pairwise. Now for $i \geq r+2,$ consider $(1+(i,i+1))he_t.$ Clearly $(1+(i,i+1))$ commutes with $h,$ and $(i,i+1)$ is a column stabiliser for $e_t,$ so $(1+(i,i+1))he_t = 0$ as required. Hence the map $f$ is a homomorphism.

    To show that $f$ is an isomorphism, it suffices to show that $f$ is a surjection, as the Specht modules have equal dimension.

    We split this into the case where $r$ is even and where $r$ is odd.
    If $r$ is even, consider the sum of polytabloids \[ u = (1 + \sum_{x = 2}^{r+1} \sum_{y=r+2}^n (xy))e_s.\] We will show that $f(u) = e_t,$ and so $f$ is surjective as $e_t$ is a generator for $S_2^{(r+1, 1^{n-r-1})}$. 

    Write $z = \sum_{x = 2}^{r+1} \sum_{y=r+2}^n (xy)$ and $h = 1 + h'.$ To show that $f(u) = e_t,$ we will show that $(1+z)(1+h')e_t = e_t.$ 
        
    For $X,Y \subseteq \{1, \ldots, n\}$ such that $X \subseteq \{2, \ldots, r+1\} \cup Y$ and $|X| = |Y|,$ let $e_X^Y$ be the tabloid for the tableau which has in the first row but not first column the numbers $\{2, \ldots, r+1\}\cup Y \setminus X.$ Then it is clear that for $x \in \{2, \ldots, r+1\}, y \in \{r+2, \ldots, n\}, (xy)e_t = e_{\{x\}}^{\{y\}}$ and for $l \in \{2, \ldots, r+1\}, h_le_t = e_{\{l\}}^{\{1\}}.$ For each $l,$ we have $e_{\{l\}}^{\{1\}} = e_{t'}$ where $t'$ is the tableau which has $1$ in the second column, and $2, \ldots, l-1, l+1, \ldots, r+1$ in the first row from the third column onwards. Using the Garnir relation on the first two columns, we have that $e_{\{l\}}^{\{1\}} = \sum_{k = r+2}^n e_{\{l\}}^{\{k\}} + e_t.$

    So \[(z + h')e_t = \sum_{x = 2}^{r+1}\sum_{y=r+2}^n e_{\{x\}}^{\{y\}} + \sum_{l=2}^{r+1}\sum_{k=r+2}^n (e_{\{l\}}^{\{k\}} + e_t) = 0\] where the final equality comes from the cancellation of the two summations, and the fact that $r(n-r-1)$ is even.
                
    Now consider $zh'e_t = z(\sum_{l=2}^{r+1}e_{\{l\}}^{\{1\}}).$ If $x = l$ then $x$ and $y$ are both in the first column of $e_{\{l\}}^{\{1\}},$ hence the permutation $(xy)$ is a column stabiliser of $e_{\{l\}}^{\{1\}}$ and $(xy)e_{\{l\}}^{\{1\}} = e_{\{l\}}^{\{1\}}.$ If $x \neq l$ then $(xy)e_{\{l\}}^{\{1\}} = e_{\{l,x\}}^{\{1,y\}}.$ 

    So \[ zh'e_t = \sum_{\mathclap{\substack{l,x = 2,\\l\neq x\\}}}^{r+1}\sum_{y=r+2}^n e_{\{l,x\}}^{\{1,y\}} + \sum_{x=2}^{r+1}\sum_{y=r+2}^{n}e_{\{x\}}^{\{1\}}.\]

    Note that \[ \sum_{\mathclap{\substack{l,x = 2,\\l\neq x\\}}}^{r+1}\sum_{y=r+2}^n e_{\{l,x\}}^{\{1,y\}} = 0\] as $e_{\{l,x\}}^{\{1,y\}} = e_{\{x,l\}}^{\{1,y\}}$ and $x\neq l$ so these terms cancel out. 

    Similarly, \[ \sum_{x=2}^{r+1}\sum_{y=r+2}^{n}e_{\{x\}}^{\{1\}} = (n-r-1)\sum_{x=2}^{r+1}e_{\{x\}}^{\{1\}} = 0\] as $(n-r-1)$ is even.

    Hence if $r$ is even, we have $f(u) = (1+z)(1+h')e_t = e_t.$
        
    If $r$ is odd, consider the sum of polytabloids \[ v = (\sum_{x = 2}^{r+1} \sum_{y=r+2}^n (xy))e_s.\] Then by a similar calculation to the previous case, one can verify that $f(v) = e_t,$ and so $f$ is surjective.

    Hence when $n$ is odd and $1 \leq r < n,$ we have $S_2^{(n-r,1^r)} \isom S_2^{(r+1, 1^{n-r-1})}.$
\end{proof}

\begin{example}\label{Example_S_2^(7_1^4)}
Consider $S_2^{(7,1^4)}.$ We have $s = \young(16789<10><11>,2,3,4,5)$ and $t = \young(12345,6,7,8,9,<10>,<11>).$ Then \[f(e_s) = e\,\young(12345,6,7,8,9,<10>,<11>) + e\,\young(21345,6,7,8,9,<10>,<11>) + e\,\young(31245,6,7,8,9,<10>,<11>) + e\,\young(41325,6,7,8,9,<10>,<11>) + e\,\young(51342,6,7,8,9,<10>,<11>) \,.\] We provide the submodule lattice for $S_2^{(7,1^4)}$ below, and by \Cref{odd_self_dual} we know that this is the same as the submodule lattice for $S_2^{(5,1^6)}.$ In the diagram below, the vertices represent submodules of $S^\lambda_2,$ and the nodes are labelled by the dimension of the submodule. The directed edge from a vertex $u$ to a vertex $v$ indicates that the submodule associated to $u$ is a maximal submodule of the submodule associated to $v,$ and the edges are labelled by their irreducible subquotients.

    \scalebox{0.8}{
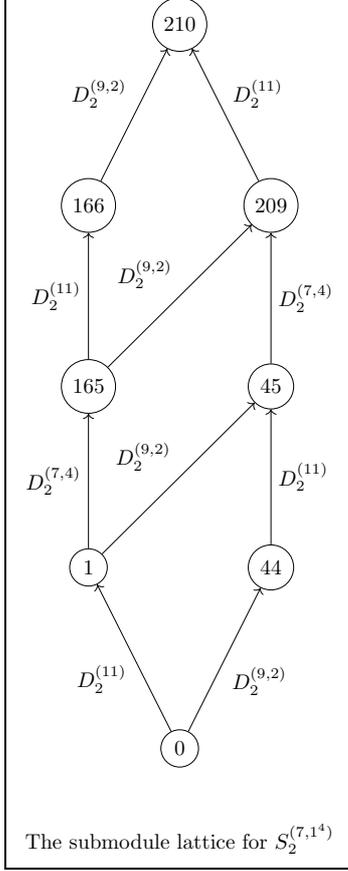
\begin{tabular}{|c|}
         \hline
        \begin{tikzpicture}[main/.style = {draw, circle}, scale = 3]
        \node at (0,-1/2) {The submodule lattice for $S^{( 7, 1^4 )}_2$};
        
        \node[main] (0) at (0,0) {$0$};
        \node[main] (1) at (-1/2,1) {$1$};
        \node[main] (44) at (1/2,1) {$44$};
        \node[main] (45) at (1/2,2) {$45$};
        \node[main] (165) at (-1/2,2) {$165$};
        \node[main] (166) at (-1/2,3) {$166$};
        \node[main] (209) at (1/2,3) {$209$};
        \node[main] (210) at (0,4) {$210$};
        
        \draw[->] (0) edge ["$D_2^{(11)}$", pos = 0.5] (1);
        \draw[->] (0) edge ["$D_2^{(9,2)}$"', pos = 0.5] (44);
        
        \draw[->] (1) edge ["$D_2^{(9,2)}$", pos = 0.5] (45);
        \draw[->] (1) edge ["$D_2^{(7,4)}$", pos = 0.5] (165);
        
        \draw[->] (44) edge ["$D_2^{(11)}$"', pos = 0.5] (45);
        
        \draw[->] (45) edge ["$D_2^{(7,4)}$"', pos = 0.5] (209);
        
        \draw[->] (165) edge ["$D_2^{(11)}$", pos = 0.5] (166);
        \draw[->] (165) edge ["$D_2^{(9,2)}$", pos = 0.5] (209);
        
        \draw[->] (166) edge ["$D_2^{(9,2)}$", pos = 0.5] (210);
        
        \draw[->] (209) edge ["$D_2^{(11)}$"', pos = 0.5] (210);
        
        \end{tikzpicture}
        \\
         \hline
    \end{tabular}
    }
\end{example}

\Cref{FPL,odd_self_dual} allow us to calculate the full submodule structure of hook Specht modules, when $n$ is odd and somewhat small.

\begin{example}
For $\lambda = (9, 1^4), S_2^\lambda$ has composition factors $D_2^{(9,4)}, D_2^{(11,2)}$ and $D_2^{(13)}$ with multiplicities $1, 2, 3$ respectively. By \cite[Theorem 5.5]{murphy1980decomposability}, $S_2^\lambda$ has a direct summand isomorphic to the trivial module, so $S_2^\lambda \isom D_2^{(13)} \oplus M$ for a submodule $M$ which is self dual, has a submodule isomorphic to $S_2^{(9,4)},$ has a quotient isomorphic to the dual of $S_2^{(9,4)},$ and has composition factors $D_2^{(9,4)}, D_2^{(11,2)}$ and $D_2^{(13)}$ with multiplicities $1, 2, 2$ respectively. $S_2^{(9,4)}$ itself is uniserial with composition factors, from bottom to top, $D_2^{(11,2)}, D_2^{(13)}, D_2^{(9,4)},$ and hence $M$ must be uniserial with composition factors, from bottom to top, $D_2^{(11,2)}, D_2^{(13)}, D_2^{(9,4)}, D_2^{(13)}, D_2^{(11,2)}$.
\end{example}

\Cref{odd_self_dual} does not work for $n$ even, but we can instead construct another filtration in the case that $n$ is even and $n-r \leq r,$ if we relax the condition that all of the quotients must be isomorphic to Specht modules.

Recalling \Cref{S*(n-r_r)}, we now state the main result of this section:

\begin{theorem}\label{second_filtration}
    Let $\lambda = (n-r,1^r) \vdash n$ be a hook partition of $n,$ with $n$ even and $n-r \leq r.$ Then $S_2^\lambda$ has a filtration (from bottom to top) by the modules $S_2^{*(r,n-r)}, S_2^{(r+2,n-r-2)}, \ldots, S_2^{(r+2l, n-r-2l)}, \ldots.$
\end{theorem}

\begin{proof}
    As before, the main part of the proof is to build the submodules $0 = N_{-1} \leq N_0 \leq \ldots \leq N_{\floor{(n-r)/2}} = S_2^{(n-r,1^r)},$ where the quotients $N_l/N_{l-1}$ are isomorphic to $S_2^{(r+2l, n-r-2l)}$ for $l \geq 1$ and $N_0 \isom S_2^{*(r,n-r)}.$

    We first define $N_l$ as a vector subspace of $S_2^{(n-r,1^r)}$ in the following way.
    
    If $l=0,$ let $\Bar{X}_0$ be the set of ordered $n$-tuples \[X = (a_1, b_1, \ldots, a_{n-r-1},b_{n-r-1}, c_{1}, \ldots, c_{2r-n+2}),\] consisting of different elements of $\{1, \ldots, n\}.$ Then for each $X \in \Bar{X}_0,$ let $T_X$ be the set of column-standard $(n-r,1^r)$-tableaux $t$ such that the first column of $t$ contains
    \begin{itemize}
        \item $a_1$ and $b_1,$
        \item exactly one of $a_i$ or $b_i$ for each $i = 2, \ldots, n-r-1,$
        \item all but one of $c_1, \ldots, c_{2r+2-n}.$
    \end{itemize}

    If $0 < 2l < n-r,$ let $\Bar{X}_l$ be the set of ordered $(n-2l+1)$-tuples \[X = (a_1, b_1, \ldots, a_{n-r-2l},b_{n-r-2l}, c_1, \ldots, c_{2r+2l+1-n}),\]
    consisting of different elements of $\{1, \ldots, n\}.$ Then for each $X \in \Bar{X}_l,$ let $T_X$ be the set of column-standard $(n-r,1^r)$-tableaux $t$ such that the first column of $t$ contains
    \begin{itemize}
        \item $a_1$ and $b_1,$
        \item exactly one of $a_i$ or $b_i$ for each $i = 2, \ldots, n-r-2l,$
        \item all but one of $c_1, \ldots, c_{2r+2l+1-n}.$
    \end{itemize}

    If $2l = n-r,$ let $\Bar{X}_l$ be the set of ordered $(r+3)$-tuples \[X = (a, b_1, b_2, b_3, c_1, \ldots, c_{r-1}),\] consisting of different elements of $\{1, \ldots, n\}.$ Then for each $X \in \Bar{X}_l,$ let $T_X$ be the set of column-standard $(n-r,1^r)$-tableaux $t$ such that the first column of $t$ contains
    \begin{itemize}
        \item $a,$
        \item exactly $2$ of $b_1, b_2, b_3,$
        \item all but one of $c_1, \ldots, c_{r-1}.$
    \end{itemize}

    Then define $N_l$ to be the vector subspace of $S_2^{(n-r,1^r)}$ generated by $\{\sum_{t \in T_X} e_t \, | \, X \in \Bar{X}_l\}.$

    We need to show that $N_{l-1} \subseteq N_l.$ We do this in three cases:
    \begin{enumerate}
        \item To show that $N_0 \subseteq N_1,$ fix an \[X = (a_1, b_1, \ldots, a_{n-r-1},b_{n-r-1}, c_{1}, \ldots, c_{2r-n+2}) \in \Bar{X}_0.\] We want to show that $\sum_{t \in T_X} e_t \in N_1.$ Let \[Y_1 = (a_1, b_1, \ldots, a_{n-r-2},b_{n-r-2}, c_1, \ldots, c_{2r-n+2}, a_{n-r-1}) \in \Bar{X}_1\] and \[Y_2 = (a_1, b_1, \ldots, a_{n-r-2},b_{n-r-2}, c_1, \ldots, c_{2r-n+2}, b_{n-r-1}) \in \Bar{X}_1,\] then it is a straightforward calculation to show that \[\sum_{t \in T_X}e_t = \sum_{t \in T_{Y_1}}e_t + \sum_{t\in T_{Y_2}}e_t,\] and hence $N_0$ is a vector subspace of $N_1.$
        
        \item To show that $N_{l-1} \subseteq N_l$ when $0 < 2l < n-r,$ fix an \[X = (a_1, b_1, \ldots, a_{n-r-2l+2},b_{n-r-2l+2}, c_1, \ldots, c_{2r+2l-1-n}) \in \Bar{X}_{l-1}.\] We want to show that $\sum_{t \in T_X} e_t \in N_l.$ Let 
        \[Y_1 = (a_1, b_1, \ldots, a_{n-r-2l},b_{n-r-2l}, c_1, \ldots, c_{2r+2l-1-n}, a_{n-r-2l+1}, a_{n-r-2l+2})\in \Bar{X}_l,\]
        \[Y_2 = (a_1, b_1, \ldots, a_{n-r-2l},b_{n-r-2l}, c_1, \ldots, c_{2r+2l-1-n}, a_{n-r-2l+1}, b_{n-r-2l+2}) \in \Bar{X}_l,\] 
        \[Y_3 = (a_1, b_1, \ldots, a_{n-r-2l},b_{n-r-2l}, c_1, \ldots, c_{2r+2l-1-n}, b_{n-r-2l+1}, a_{n-r-2l+2}) \in \Bar{X}_l,\] 
        \[Y_4 = (a_1, b_1, \ldots, a_{n-r-2l},b_{n-r-2l}, c_1, \ldots, c_{2r+2l-1-n}, b_{n-r-2l+1}, b_{n-r-2l+2}) \in \Bar{X}_l,\] 
              
        then it is a straightforward calculation to show that \[\sum_{t \in T_X}e_t = \sum_{t \in T_{Y_1}}e_t + \sum_{t\in T_{Y_2}}e_t+ \sum_{t\in T_{Y_3}}e_t+ \sum_{t\in T_{Y_4}}e_t,\] and hence $N_{l-1}$ is a vector subspace of $N_l.$

        \item  To show that $N_{l-1} \subseteq N_l$ when $2l = n-r,$ fix an \[X = (a_1, b_1, a_{2},b_{2}, c_1, \ldots, c_{r-1}) \in \Bar{X}_{l-1}.\] We want to show that $\sum_{t \in T_X} e_t \in N_l.$ Let 
        \[Y_1 = (a_2, a_1, b_1, b_{2}, c_1, \ldots, c_{r-1}) \in \Bar{X}_{l},\]
        \[Y_2 = (b_2, a_1, b_1, a_{2}, c_1, \ldots, c_{r-1}) \in \Bar{X}_{l},\] then it is a straightforward calculation to show that \[\sum_{t \in T_X}e_t = \sum_{t \in T_{Y_1}}e_t + \sum_{t\in T_{Y_2}}e_t,\] and hence $N_{l-1}$ is a vector subspace of $N_l.$
    \end{enumerate}

    It is clear that each $N_l$ is actually a submodule of $S_2^{(n-r,1^r)},$ as $\Bar{X}_k$ is fixed under the natural action of $S_n$ on tuples of elements $\{1, \ldots, n\},$ and for all $\pi \in S_n,$ \[\pi \sum_{t \in T_X} e_t = \sum_{t \in T_{\pi X}}e_t,\]  so the given spanning set for $N_l$ is closed under the action of $S_n,$ and hence so is $N_l.$

    Now we construct homomorphisms $\phi_l, 0 \leq 2l \leq n-r,$ where $\phi_0 : S_2^{(r+1,n-r-1)} \to N_0/N_{-1}$ and for $0 < 2l \leq n-r ,$ \[\phi_l: S_2^{(r+2l, n-r-2l)}\to \frac{N_l}{N_{l-1}},\] and we set $N_{-1} \equiv 0.$

    We define $\phi_l$ in the following way. 
    \begin{enumerate}
        \item For $l=0$ and an $(r+1,n-r-1)$-tableau $s,$ let $a_i$ be the entry in the $i$th column first row, $b_i$ be the entry in the $i$th column second row, then for some choice of $c_i$ not equal to $a_1, b_1, \ldots, a_{n-r-1}, b_{n-r-1},$ set $X = (a_1, b_1, \ldots, a_{n-r-1}, b_{n-r-1}, c_1, \ldots, c_{2r-n+2}).$
        \item For $0 < 2l < n-r$ and an $(r+2l, n-r-2l)$-tableau $s,$ let $a_i$ be the entry in the $i$th column first row, $b_i$ be the entry in the $i$th column second row, then for some choice of $c_i$ not equal to $a_1, b_1, \ldots, a_{n-r-2l}, b_{n-r-2l},$ set $X = (a_1, b_1, \ldots, a_{n-r-2l}, b_{n-r-2l}, c_1, \ldots, c_{2r+2l+1-n}).$
        \item For $2l = n-r$ and an $(n)$-tableau $s,$ let $a, b_1, b_2, b_3, c_1, \ldots, c_{r-1}$ be distinct entries, and let $X = (a, b_1, b_2, b_3, c_1, \ldots, c_{r-1}).$
    \end{enumerate}

    Now define the maps \[\phi_l(e_s) := \sum_{t \in T_X}e_t + N_{l-1},\] and extend linearly.

    First we will show that this is independent of the choice of $c_i$'s, then we will show that $\phi_l$ is a $KS_n$-module homomorphism, and that $N_{\floor{(n-r)/2}} = S_2^{(n-r,1^r)}.$

    It is clear that reordering the $c_i$'s does not change $\phi_l$ for any $l,$ and when $l=0, X$ has exactly $n$ entries and so this is well-defined. For $0 < 2l < n-r,$ it suffices to also check that if some $c_1$ is relabelled by some $k\leq n, k \not\in X,$ that the sum is still well-defined.
    
    For $l=1,$ let $\phi^1_1(e_s)$ be the sum as defined above, and let $\phi^2_1(e_s)$ be the same sum but with $k$ in place of $c_1.$ Now define $s'$ to be a $(r+1, n-r-1)$-tableau which has $a_i, b_i$ in the first $n-r-2$ columns, and $c_1$ and $k$ in the $n-r-1$st column. Then $\phi^1_1(e_s) + \phi^2_1(e_s) = \phi_0(e_{s'}).$

    For $0 < 2l < n-r,$ let $\phi^1_l(e_s)$ be the sum as defined above, and let $\phi^2_l(e_s)$ be the same sum but with $k$ in place of $c_1.$ Now for $j$ from $1$ to $(r+l-n/2)$ (recalling that $n$ is even by assumption), write $s_j$ for the $(r+2l-2, n-r-2l+2)$-tableau which has $a_i, b_i$ in the first $n-r-2l$ columns, $k$ and $c_1$ in the $n-r-2l+1$st column, and $c_{2j}, c_{2j+1}$ in the $n-r-2l+2$nd column. Then $\phi^1(e_s) + \phi^2(e_s) = \sum_{j = 1}^{r + l - n/2} \phi_{l-1}(e_{s_j}).$

    For $2l = n-r,$ it is clear that relabelling $b_i$ with another $b_{i'}$ or relabelling $c_i$ with another $c_{i'}$ does not change the sum. It suffices to show that relabelling $a$ or $b_i$ or $c_j$ with some other $k \leq n, k \not\in X$ remains well-defined. Once this has been shown, it follows that permuting say $a$ and $b_i$ is well-defined, as this is the same as permuting $a$ and $k,$ and then $k$ with $b_i.$

    First, the case of permuting $a$ with some $k.$ Let $\phi^1_l(e_s)$ be the sum as defined above, and let $\phi^2_l(e_s)$ be the same sum but with $k$ in place of $a.$ Now define $s_1$ to be an $(n-2,2)$-tableau which has $a$ and $b_1$ in the first column, and $b_2$ and $b_3$ in the second column; define $s_2$ to be an $(n-2,2)$-tableau which has $b_1$ and $k$ in the first column, and $b_2$ and $b_3$ in the second column; and define $s_3$ to be an $(n-2,2)$-tableau which has $b_2$ and $b_3$ in the first column, and $a$ and $k$ in the second column. Then $\phi^1_l(e_s) + \phi^2_l(e_s) = \phi_{l-1}(e_{s_1}) + \phi_{l-1}(e_{s_2}) + \phi_{l-1}(e_{s_3}).$ 

    Second, the case of permuting say $b_1$ with some $k.$ Let $\phi^1_l(e_s)$ be the sum as defined above, and let $\phi^2_l(e_s)$ be the same sum but with $k$ in place of $b_1.$ Now define $s_1$ to be an $(n-2,2)$-tableau which has $a$ and $b_1$ in the first column, and $b_2$ and $b_3$ in the second column; and $s_2$ to be an $(n-2)$-tableau which has $a$ and $k$ in the first column, and $b_2$ and $b_3$ in the second column. Then $\phi^1_l(e_s) + \phi^2_l(e_s) = \phi_{l-1}(e_{s_1}) + \phi_{l-1}(e_{s_2}).$  

    Finally, the case of permuting say $c_1$ with some $k.$ Let $\phi^1_l(e_s)$ be the sum as defined above, and let $\phi^2_l(e_s)$ be the same sum but with $k$ in place of $c_1.$ Now for $j$ from $1$ to $(r+l-n/2)$ (recalling that $n$ is even by assumption), write $s_j$ for the $(n-2,2)$-tableau which has $a$ and $c_1$ in the first column, and $c_{2j}, c_{2j+1}$ in the second column; and write $s'_j$ for the $(n-2,2)$-tableau which has $a$ and $k$ in the first column, and $c_{2j}, c_{2j+1}$ in the second column. Then $\phi^1_l(e_s) + \phi^2_l(e_s) = \sum_{j = 1}^{r+l-n/2} e_{s_j} + e_{s'_j}.$

    Hence the maps are well-defined. Now we will show that each $\phi_l$ is a $KS_n$-module homomorphism. Clearly if $\sigma$ is a column stabiliser of $s,$ then $(1 + \sigma)\phi_l(e_s) = 0.$ We check the Garnir relations based on $l.$

    First assume $l=0.$ $S_2^{(r+1, n-r-1)}$ is cyclic, so we can choose $c_i = a_{n-r-1+i},$ and fix $\Tilde{t}$ to be the column-standard tableau that has in its first column $b_1, a_i$ for $i = 1, \ldots, n-r-1,$ and $c_j$ for $j = 2, \ldots, 2r+2-n.$ For $i = 2, \ldots, n-r-1$ define $\Hat{h}_i = (1 + (a_i, b_i))$ and define $h = 1 + \sum_{j = 2}^{2r+2-n} (c_1, c_j).$ Then, 
    \begin{align*}
    \phi_l(e_s) &= \sum_{t\in T_X}e_t + N_{-1}\\
    &= (\prod_{i=2}^{n-r-1}\hat{h}_i) h e_{\Tilde{t}} + N_{-1}.
    \end{align*}
    Note that the collection of $\hat{h}_i$ and $h$ all pairwise commute. 

    There are $6$ types of Garnir relations:
    \begin{enumerate}
    \item The Garnir relation on the set $\{a_1, b_1, a_2\}$ which has the Garnir element $g_{1,a} := 1 + (b_1, a_2) + (a_1, b_1, a_2).$
    \item The Garnir relation on the set $\{b_1, a_2, b_2\}$ which has the Garnir element $g_{1,b} := 1 + (b_1, b_2) + (b_1, a_2, b_2).$
    \item For $2 \leq i \leq n-r-2,$ the Garnir relation on the set $\{a_i, b_i, a_{i+1}\}$ which has the Garnir element $g_{i, a} := 1 + (b_i, a_{i+1}) + (a_i, b_i, a_{i+1}).$
    \item For $2 \leq i \leq n-r-2,$ the Garnir relation on the set $\{b_i, a_{i+1}, b_{i+1}\}$ which has the Garnir element $g_{i,b} := 1 + (b_i, b_{i+1}) + (b_i, a_{i+1}, b_{i+1}).$ 
    \item The Garnir relation on the set $\{a_{n-r-1}, b_{n-r-1}, c_1\}$ which has the Garnir element $g_{n-r-1, a} := 1 + (b_{n-r-1}, c_1) + (a_{n-r-1}, b_{n-r-1}, c_1).$
    \item For $i > n-r-1,$ the Garnir relation on the set $\{a_i,a_{i+1}\}$ which has the Garnir element $1 + (a_i, a_{i+1}).$
    \end{enumerate}

    We will show that each of these Garnir elements also eliminates $\phi_0(e_s)$ and hence the homomorphism is well defined.

    \begin{enumerate}
        \item Note that $g_{1,a}$ commutes with every $\hat{h}_i$ and $h$ except for $\hat{h}_2 = (1 + (a_2, b_2)),$ and that $g_{1,a}\hat{h}_2$ commutes with every other $\hat{h}_i$ and $h.$ Then 
        \begin{align*}
            g_{1,a}\hat{h}_2 &= (1 + (b_1, a_2) + (a_1, b_1, a_2))(1 + (a_2, b_2))\\
            &= (1 + (b_1, a_2) + (a_1, b_1, a_2) + (a_2, b_2) + (b_1, a_2, b_2) + (a_1, b_1, a_2, b_2)),
        \end{align*}

        Let $X = \{a_3, \ldots, a_{n-r-1}, c_2, \ldots, c_{2r+2-n}\}$ and for $x \in X$ let $e_{t/x}$ be the polytabloid for the tableau which has in the first column \[\{a_1, b_1, b_2, a_2, \ldots, a_{n-r-1}, c_2, \ldots, c_{2r+2-n}\}\setminus\{x\}.\] Then \[ g_{1,a}\hat{h}_2 e_{\Tilde{t}} = \sum_{x \in X} e_{t/x}\] by the Garnir relations on $S^{(n-r,1^r)},$ and for each $x,$ if $x = a_i$ then $\hat{h}_ig_{1,a}\hat{h}_2e_{t/x} = 0$ and if $x = c_j,$ \[h (\sum_{j = 2}^{2r+2-n} e_{t/c_j}) = 0.\]

        \item Similar to above.

        \item Note that for $2 \leq i \leq n-r-2,$ $g_{i,a}$ commutes with $h$ and every $\hat{h}_j$ except for $\hat{h}_i$ and $\hat{h}_{i+1}.$ Then
        \begin{align*}
        &g_{i, a}\hat{h}_i\hat{h}_{i+1} \\
        &= (1 + (b_i, a_{i+1}) + (a_i, b_i, a_{i+1}))(1 + (a_i, b_i))(1 + (a_{i+1}, b_{i+1}))\\
        &= 1 + (b_i, a_{i+1}) + (a_i, b_i, a_{i+1}) + (a_i, b_i) + (a_i, a_{i+1}, b_i) + (a_i, a_{i+1}) \\
        &+ (a_{i+1}, b_{i+1}) + (b_i, a_{i+1}, b_{i+1}) + (a_i, b_i, a_{i+1}, b_{i+1}) + (a_i, b_i)(a_{i+1}, b_{i+1}) \\
        &+ (a_i, a_{i+1}, b_{i+1}, b_i) + (a_i, a_{i+1}, b_{i+1}),
        \end{align*}  
        and
        \begin{align*}
        g_{i, a}\hat{h}_i\hat{h}_{i+1}e_{\Tilde{t}} = 0.
        \end{align*}

        \item Similar to above.

        \item Note that $g_{n-r-1,a}$ commutes with every $\hat{h}_j$ except for $\hat{h}_{n-r-1},$ and does not commute with $h.$ Then
        \begin{align*}
        &g_{n-r-1, a}\hat{h}_{n-r-1}h \\
        &= (1 + (b_{n-r-1}, c_1) + (a_{n-r-1}, b_{n-r-1}, c_1))(1 + (a_{n-r-1}, b_{n-r-1}))(1 + \sum_{j=2}^{2r+2-n}(c_1, c_j))\\
        &= (1 + (b_{n-r-1}, c_1) + (a_{n-r-1}, b_{n-r-1}, c_1) + (a_{n-r-1}, b_{n-r-1}) + (a_{n-r-1}, c_1, b_{n-r-1}) \\ &+ (a_{n-r-1}, c_1))(1 + \sum_{j=2}^{2r+2-n}(c_1, c_j)),
        \end{align*}  
        and
        \begin{align*}
        g_{i, a}\hat{h}_i\hat{h}_{i+1}e_{\Tilde{t}} = 0.
        \end{align*}

        \item As $i > n-r-1, a_i$ and $a_{i+1}$ are just $c_{i'}$ and $c_{i'+1}$ where $i' = i-n+r+1,$ and the case of relabelling to $c_j$s has already been dealt with.
        
        Also, $1 + (c_{i'}, c_{i'+1})$ commutes with $\hat{h}_j$ for every $j,$ and $(1+ (c_{i'},c_{i'+1}))(c_1, c_j) e_{\tilde{t}} = 0$ unless $j=i'$ or $j=i'+1.$ But then $(1 + (c_{i'}, c_{i'+1}))((c_1,c_{i'}) + (c_1, c_{i'+1}))e_{\tilde{t}}=0$ and $(1 + (c_{1}, c_{2}))((1 + (c_1, c_{2}))e_{\tilde{t}}=0,$ so $(1 + (a_i, a_{i+1}))\phi_l(e_s) = 0$ as required.
    \end{enumerate}

Now assume $0 < 2l < n-r.$ $S^{(r+2l, n-r-2l)}$ is cyclic, so we can choose $c_i = a_{n-r-1+i},$ and fix $\Tilde{t}$ to be the column-standard tableau that has in its first column $b_1, a_i,$ for $i = 1, \ldots, n-r-2l,$ and $c_j$ for $j = 2, \ldots, 2r+2l+1-n.$ For $i = 2, \ldots, n-r-2l$ define $\hat{h}_i = (1 + (a_i, b_i))$ and define $h = 1+ \sum_{j=2}^{2r+2l+1-n} (c_1, c_j).$ Then,
\begin{align*}
    \phi_l(e_s) &= \sum_{t \in T_X} e_t + N_{l-1}\\
    &= (\prod_{i=2}^{n-r-2l}\hat{h}_i) h e_{\Tilde{t}} + N_{l-1}.
\end{align*}
Note that the collection of $\hat{h}_i$ and $h$ all pairwise commute.

There are $6$ types of Garnir relations:
    \begin{enumerate}
    \item The Garnir relation on the set $\{a_1, b_1, a_2\}$ which has the Garnir element $g_{1,a} := 1 + (b_1, a_2) + (a_1, b_1, a_2).$
    \item The Garnir relation on the set $\{b_1, a_2, b_2\}$ which has the Garnir element $g_{1,b} := 1 + (b_1, b_2) + (b_1, a_2, b_2).$
    \item For $2 \leq i \leq n-r-2l-1,$ the Garnir relation on the set $\{a_i, b_i, a_{i+1}\}$ which has the Garnir element $g_{i, a} := 1 + (b_i, a_{i+1}) + (a_i, b_i, a_{i+1}).$
    \item For $2 \leq i \leq n-r-2l-1,$ the Garnir relation on the set $\{b_i, a_{i+1}, b_{i+1}\}$ which has the Garnir element $g_{i,b} := 1 + (b_i, b_{i+1}) + (b_i, a_{i+1}, b_{i+1}).$ 
    \item The Garnir relation on the set $\{a_{n-r-2l}, b_{n-r-2l}, a_{n-r-2l+1}\}$ which has the Garnir element $g_{n-r-2l, a} := 1 + (b_{n-r-2l}, a_{n-r-2l+1}) + (a_{n-r-2l}, b_{n-r-2l}, a_{n-r-2l+1}).$
    \item For $i > n-r-2l,$ the Garnir relation on the set $\{a_i,a_{i+1}\}$ which has the Garnir element $1 + (a_i, a_{i+1}).$
    \end{enumerate}
    
    We will show that each of these Garnir elements also eliminates $\phi_l(e_s)$ and hence the homomorphism is well defined.

\begin{enumerate}
    \item Note that $g_{1,a}$ commutes with every $\hat{h}_i$ and $h$ except for $\hat{h}_2 = (1 + (a_2, b_2)),$ and that $g_{1,a}\hat{h}_2$ commutes with every other $\hat{h}_i$ and $h.$ Then 
        \begin{align*}
            g_{1,a}\hat{h}_2 &= (1 + (b_1, a_2) + (a_1, b_1, a_2))(1 + (a_2, b_2))\\
            &= (1 + (b_1, a_2) + (a_1, b_1, a_2) + (a_2, b_2) + (b_1, a_2, b_2) + (a_1, b_1, a_2, b_2)),
        \end{align*}
        Let $X = \{a_3, \ldots, a_{n-r-2l}, c_2, \ldots, c_{2r+2l+1-n}\}$ and for $x \in X$ let $e_{t/x}$ be the polytabloid for the tableau which has in the first column \[\{a_1, b_1, b_2, a_2, \ldots, a_{n-r-2l}, c_2, \ldots, c_{2r+2l+1-n}\}\setminus\{x\}.\] Then \[ g_{1,a}\hat{h}_2 e_{\Tilde{t}} = \sum_{x \in X} e_{t/x},\] and for each $x,$ if $x = a_i$ then $\hat{h}_ig_{1,a}\hat{h}_2e_{t/x} = 0$ and if $x = c_j,$ \[h (\sum_{j = 2}^{2r+2-n} e_{t/c_j}) = 0.\]

    \item Similar to above.

    \item Note that for $2 \leq i \leq n-r-2l,$ $g_{i,a}$ commutes with $h$ and every $\hat{h}_j$ except for $\hat{h}_i$ and $\hat{h}_{i+1}.$ Then
        \begin{align*}
        &g_{i, a}\hat{h}_i\hat{h}_{i+1} \\
        &= (1 + (b_i, a_{i+1}) + (a_i, b_i, a_{i+1}))(1 + (a_i, b_i))(1 + (a_{i+1}, b_{i+1}))\\
        &= 1 + (b_i, a_{i+1}) + (a_i, b_i, a_{i+1}) + (a_i, b_i) + (a_i, a_{i+1}, b_i) + (a_i, a_{i+1}) \\
        &+ (a_{i+1}, b_{i+1}) + (b_i, a_{i+1}, b_{i+1}) + (a_i, b_i, a_{i+1}, b_{i+1}) + (a_i, b_i)(a_{i+1}, b_{i+1}) \\
        &+ (a_i, a_{i+1}, b_{i+1}, b_i) + (a_i, a_{i+1}, b_{i+1}),
        \end{align*}  
        and
        \begin{align*}
        g_{i, a}\hat{h}_i\hat{h}_{i+1}e_{\Tilde{t}} = 0.
        \end{align*}

    \item Similar to above.

    \item Note that as $0 < 2l,$ $n-r-2l+1 < n-r$ so $a_{n-r-2l+1} \neq c_j$ for any $j,$ hence $g_{n-r-2l,a}$ commutes with every $\hat{h}_j$ except for $\hat{h}_{n-r-2l},$ and \emph{does} commute with $h.$ Then
        \begin{align*}
        &g_{n-r-2l, a}\hat{h}_{n-r-2l} \\
        &= (1 + (b_{n-r-2l}, a_{n-r-2l+1}) + (a_{n-r-2l}, b_{n-r-2l}, a_{n-r-2l+1}))(1 + (a_{n-r-2l}, b_{n-r-2l}))\\
        &= 1 + (b_{n-r-2l}, a_{n-r-2l+1}) + (a_{n-r-2l}, b_{n-r-2l}, a_{n-r-2l+1}) + (a_{n-r-2l}, b_{n-r-2l})\\ &+ (a_{n-r-2l}, a_{n-r-2l+1}, b_{n-r-2l}) + (a_{n-r-2l}, a_{n-r-2l+1})
        \end{align*}  
        and
        \begin{align*}
        g_{n-r-2l, a}\hat{h}_{n-r-2l}e_{\Tilde{t}} = 0.
        \end{align*}

    \item As $i > n-r-1$, $\{a_i, a_{i+1}\} \cap \{c_j\}_{j=1}^{2r+2l+1-n} = 2$ or $1.$ In the first case, $(1+(a_i, a_{i+1}))$ corresponds to relabelling two $c_j$s, and in the latter case this corresponds to relabelling $c_j$ with some other element not in $X.$ Both cases have been dealt with.
\end{enumerate}

Now assume $l = n-r.$ Any Garnir relation is of the form $1 + \sigma$ where $\sigma$ is a transposition on the set $\{a, b_1, b_2, b_3, c_1, \ldots, c_{r-1}, k\}$ where $k \leq n$ is distinct from the previous terms. However, we have already shown that $(1 + \sigma)\phi_{n-r}(e_s) \in N_{n-r-1}.$

Hence $\phi_l$ is a $KS_n$ homomorphism, because the construction is invariant under the action of $S_n.$ That is, for all $\pi \in S_n, \phi_l(\pi \theta_s) = \phi_l(e_{\pi s}) = \pi \phi_l(e_s).$

Now we show that $N_{\floor{(n-r)/2}} = S_2^{(n-r,1^r)}.$ We split this into the case that $2l= n-r-1$ and $2l = n-r.$

Fix $l$ such that $2l = n-r-1,$ and observe that $r$ must be odd as $n$ is even. We want to show that $N_l = S_2^{(n-r, 1^r)}.$ As $S_2^{(n-r,1^r)}$ is generated by $e_t$ for any tableau $t,$ it suffices to show that $e_t \in N_l,$ for some $(n-r,1^r)$-tableau $t.$

Fix a subset $X = \{a, b, c_1, \ldots, c_{r}\}.$ Then $\sum_{t \in T_X}e_t \in N_l$ by definition. Fix $\Tilde{t}$ to be a tableau which has in its first column $a, b, c_2, \ldots, c_r,$ and write $h = 1 + \sum_{j = 2}^{r} (c_1, c_j)$ then \[ \sum_{t \in T_X} e_t = h e_{\tilde{t}}.\] Now write $h' = 1 + \sum_{k = 2}^r (b, c_{k})$ and consider the expression \[ h' h e_{\tilde{t}}.\] Note that $h'h = h' + h'\sum_{j=2}^r(c_1, c_j),$ and $h'e_{\Tilde{t}} = e_{\Tilde{t}}$ as each transposition in $h'$ is a column stabiliser of $\tilde{t}$ and there are an odd number of terms in $h'.$ Now consider $h'\sum_{j=2}^r(c_1, c_j)e_{\tilde{t}} = (1 + \sum_{k = 2}^r (b, c_{k}))\sum_{j=2}^r(c_1, c_j)e_{\tilde{t}}.$ For each fixed $j,$ if $k\neq j,$ then $(b,c_k)(c_1,c_j)e_{\tilde{t}}$ is the polytabloid which has $\{a, c_1, \ldots, c_r\}\setminus\{c_j\}$ in the first column; this happens $r-1$ times, which is even and hence the terms all cancel out. If $k = j,$ then $(b, c_k)(c_1,c_k)e_{\Tilde{t}}$ is the polytabloid which has $\{a, c_1, \ldots, c_r\}$ in the first column; this happens $r-1$ times. Hence $h'he_{\tilde{t}} = e_{\Tilde{t}}$ and $N_{l} = S_2^{(n-r,1^r)}.$

Now fix $l$ such that $2l = n-r,$ and observe that $r$ must be even as $n$ is even. We want to show that $N_l = S_2^{(n-r,1^r)}.$ As $S_2^{(n-r,1^r)}$ is generated by $e_t$ for any tableau $t,$ it suffices to show that $e_t \in N_l,$ for some $(n-r,1^r)$-tableau $t.$

Fix a subset $X = \{a, b_1, b_2, b_3, c_1, \ldots, c_{r-1}\}.$ Then $\sum_{t \in T_X}e_t \in N_l$ by definition. Fix $\Tilde{t}$ to be a tableau which has in its first column $a, b_1, b_2, c_2, \ldots, c_r,$ and write $h_1 = (1 + (b_1, b_3) + (b_2, b_3))$ and $h_2 = 1 + \sum_{j = 2}^{r-1}(c_1, c_j)$ then \[ \sum_{t \in T_X} e_t = h_1 h_2 e_{\tilde{t}}.\] Now write $h'_1 = (a,b_1) + (a,b_2) + (a,b_3)$ and $h'_2 = 1 + \sum_{k=2}^{r-1}(b,c_k)$ and consider the expression \[h'_2 h'_1 h_1 h_2 e_{\Tilde{t}}.\] Note that $h_2$ commutes with both $h'_1$ and $h_1,$ so we can rewrite the expression as \[h'_2 h_2 h'_1 h_1 e_{\Tilde{t}}.\] One can quickly verify that $h'_1 h_1 e_{\Tilde{t}} = e_{\Tilde{t'}}$ where $e_{\Tilde{t'}}$ is the polytabloid that has in the first column $b_1, b_2, b_3, c_2, \ldots, c_{r-1}.$ Similarly, one can verify that $h'_2 h_2 e_{\Tilde{t'}} = e_{\Tilde{t'}}$ by a calculation similar to the case $2l = n-r-1.$

All that remains is to show that the submodule $S^{*(r,n-r)}_2$ of $S^{(n-r,1^r)}_2$ which has been described as a quotient of $S^{(r+1,n-r-1)}_2$ (that is, as $\Im(\phi_0)$) is isomorphic to a submodule of $S^{(r,n-r)}_2,$ and that for $0 \leq 2l \leq n-r,$ the maps $\phi_l$ are isomorphisms.

Recall from \Cref{2_part_exact_sequence} that, after writing $n=2k,$ there is an exact sequence \[0 \xrightarrow{\hat{\Theta}_{-1}} S^{(n)}_2 \xrightarrow{\hat{\Theta}_0} S^{(n-1,1)}_2 \xrightarrow{\hat{\Theta}_1} \ldots \xrightarrow{\hat{\Theta}_{k-2}} S^{(k+1,k-1)}_2 \xrightarrow{\hat{\Theta}_{k-1}} S^{(k,k)}_2 \xrightarrow{\hat{\Theta}_k} 0.\]

Consider the maps $\phi_0 : S^{(r+1,n-r-1)}_2 \to S^{(n-r,1^r)}_2$ and $\hat{\Theta}_{n-r-1} : S^{(r+1,n-r-1)}_2 \to S^{(r,n-r)}_2.$ In particular, by the first isomorphism theorem, we have $\Im(\phi_0) \isom S^{(r+1,n-r-1)}_2 / \Ker(\phi_0)$ and $\Im(\hat{\Theta}_{n-r-1}) \isom S^{(r+1,n-r-1)}_2 / \Ker(\hat{\Theta}_{n-r-1}),$ so if $\Ker(\phi_0) = \Ker(\hat{\Theta}_{n-r-1})$ as submodules of $S^{(r+1,n-r-1)}_2$ then $S^{*(r,n-r)}_2$ is isomorphic to both a quotient of $S^{(r+1,n-r-1)}_2$ and a submodule of $S^{(r,n-r)}_2.$

First we show that $\Ker(\phi_0) \supseteq \Ker(\hat{\Theta}_{n-r-1}) = \Im(\hat{\Theta}_{n-r-2})$ where the equality comes from the exact sequence. Recalling the construction from \Cref{2_part_explicit_hom}, we let $s$ be an $(r+2,n-r-2)$-tableau which has $a_1, \ldots, a_{n-r-2}, c_1, \ldots, c_{2r-n+4}$ in the first row (from left to right), $b_1, \ldots, b_{n-r-2}$ in the second row (from left to right). For $1 \leq j \leq r-k+2$ write $t_j$ for the $(r+1,n-r-1)$-tableau which has $a_1, \ldots, a_{n-r-2}, c_{2j-1}$ in the first row (from left to right), $b_1, \ldots, b_{n-r-2}, c_{2j}$ in the second row (from left to right), and $c_1, \ldots, \hat{c}_{2j-1}, \hat{c}_{2j}, \ldots, c_{2r-n+4}$ in the first row from (from left to right) after $c_{2j-1},$ where $\hat{c}_{2j-1}$ and $\hat{c}_{2j}$ means we skip over $c_{2j-1}$ and $c_{2j}.$ Then \[\hat{\Theta}_{n-r-2}(e_s) = \sum_{1 \leq j \leq r-k+2} e_{t_j}.\] Now for some fixed $j,$ consider $\phi_0(e_{t_j}).$ This is a sum of polytabloids for the $(n-r,1^r)$-tableaux which have in the first column:
\begin{itemize}
    \item $a_1, b_1$
    \item exactly one of $a_i$ or $b_i$ for $i = 2, \ldots, n-r-2,$
    \item exactly one of $c_{2j-1}$ or $c_{2j}$
    \item all but one of $c_l$ for $l \in \{1, \ldots, 2r-n+4\}\setminus\{2j-1, 2j\}.$
\end{itemize}

Fix one such tableau $u.$ If the first column is missing $c_l$ for $l$ even, then $e_u$ appears $\phi_0(e_{t_{l/2}});$ if the first column is missing $c_l$ for $l$ odd, then $e_u$ appears in $\phi_0(e_{t_{(l+1)/2}}).$ Hence $\phi_0\circ\hat{\Theta}_{n-r-2} = 0$ and $\Ker(\phi_0) \supseteq \Ker(\hat{\Theta}_{n-r-1}).$

We show that $\Ker(\phi_0) \subseteq \Ker(\hat{\Theta}_{n-r-1})$ by showing that $\Dim(\Ker(\phi_0)) \leq \Dim(\Ker(\hat{\Theta}_{n-r-1})).$ For $\lambda \vdash n,$ we will write $d_{\lambda}$ for $\Dim(S^{\lambda}_2),$ and for $0 \leq 2l \leq n-r$ write $k_l$ for $\Dim(\Ker(\phi_l)).$ Then we have
\begin{align*}
    d_{(n-r,1^r)} &= \sum_{0 \leq l \leq \floor{(n-r)/2}} \Dim(N_l) \\
                &= d_{(r+1,n-r-1)} - k_0 +  \sum_{1 \leq l \leq \floor{(n-r)/2}} d_{(r+2l, n-r-2l)} - k_l.   
\end{align*}

By restricting $S^{(r+2l, n-r-2l)}_2$ we get $d_{(r+2l, n-r-2l)} = d_{r+2l-1, n-r-2l} + d_{r+2l, n-r-2l-1},$ and so 
\begin{align}\label{dim_sum_1}
d_{(n-r,1^r)} = d_{(r+1, n-r-2)} + \sum_{0 \leq k \leq n-r-1} d_{(r+k,n-r-1-k)} - \sum_{0 \leq l \leq \floor{(n-r)/2}}k_l.    
\end{align}

Additionally, by restricting $S^{(n-r,1^r)}_2$ we also get $d_{(n-r,1^r)} = d_{(n-r-1,1^r)} + d_{(n-r,1^{r-1})}.$ Note that $n-1$ is odd, and hence $S^{(n-r-1,1^r)}_2$ and $S^{(n-r,1^{r-1})}_2$ are both self-dual, so $d_{(n-r-1,1^r)} = d_{(r+1, 1^{n-r-2})}$ and $d_{(n-r,1^{r-1})} = d_{(r, 1^{n-r-1})}.$ By the filtration \Cref{FPL} we also get \[d_{(r+1, 1^{n-r-2})} = \sum_{0 \leq k \leq \floor{(n-r-2)/2}} d_{(r+1+2k, n-r-2 - 2k)}\] and \[d_{(r, 1^{n-r-1})} = \sum_{0 \leq k \leq \floor{(n-r-1)/2}} d_{(r+2k, n-r-1 - 2k)},\] hence 
\begin{align}\label{dim_sum_2}
d_{(n-r,1^r)} = \sum_{0 \leq k \leq n-r-1} d_{(r+k, n-r-1-k)}.    
\end{align}

By equating \eqref{dim_sum_1} and \eqref{dim_sum_2} we get $d_{(r+1, n-r-2)} = \sum_{0 \leq l \leq \floor{(n-r)/2}} k_l,$ hence $k_0 \leq d_{(r+1, n-r-2)} = 
\Dim(\Im(\hat{\Theta}'_{n-r-2})) = \Dim(\Ker(\hat{\Theta}'_{n-r-1})) = \Dim(\Ker(\hat{\Theta}_{n-r-1}))$ as required, and $\Ker(\phi_0) = \Ker(\hat{\Theta}_{n-r-1}).$

Hence $N_0 = S^{*(r,n-r)}_2 \subseteq S^{(n-r,1^r)}_2$ is a quotient of $S^{(r+1, n-r-1)}_2$ isomorphic to a submodule of $S^{(r,n-r)}_2.$ 

Finally, we need to show that in fact for $0 < l \leq \floor{(n-r)/2} , \phi_l$ are isomorphisms. Given the exact sequence
\[0 \xrightarrow{\hat{\Theta}_{-1}} S^{(n)}_2 \xrightarrow{\hat{\Theta}_0} S^{(n-1,1)}_2 \xrightarrow{\hat{\Theta}_1} \ldots \xrightarrow{\hat{\Theta}_{k-2}} S^{(k+1,k-1)}_2 \xrightarrow{\hat{\Theta}_{k-1}} S^{(k,k)}_2 \xrightarrow{\hat{\Theta}_k} 0,\] and the above result, we obtain the exact sequence 
\[0 \xrightarrow{\hat{\Theta}_{-1}} S^{(n)}_2 \xrightarrow{\hat{\Theta}_0} S^{(n-1,1)}_2 \xrightarrow{\hat{\Theta}_1} \ldots \xrightarrow{\hat{\Theta}_{n-r-2}} S^{(r+1,n-r-1)}_2 \xrightarrow{\phi_0} S^{*(r,n-r)}_2 \xrightarrow{0} 0.\] 

We get from the exact sequence that \begin{align}\label{dim_sum_3}
\Dim(S^{*(r,n-r)}_2) + \sum_{1 \leq k \leq \floor{(n-r)/2}}d_{(r+2k,n-r-2k)} = \sum_{0 \leq k \leq \floor{(n-r-1)/2}}d_{(r+1+2k, n-r-1-2k)}.    
\end{align}
From the maps $\phi_l,$ we have 
\begin{align}\label{dim_ineq_1}
d_{(n-r,1^r)} \leq \Dim(S^{*(r,n-r)}) + \sum_{1 \leq k \leq \floor{(n-r)/2}}d_{(r+2k,n-r-2k)} = \sum_{0 \leq k \leq \floor{(n-r-1)/2}}d_{(r+1+2k, n-r-1-2k)}
\end{align} 
where the equality in \eqref{dim_ineq_1} comes from \eqref{dim_sum_3}. However, by taking the dual and the filtration as in \Cref{FPL} we also have \[d_{(n-r,1^r)} = d_{(r+1, 1^{n-r-1})} = \sum_{0 \leq k \leq \floor{(n-r-1)/2}}d_{(r+1+2k, n-r-1-2k)},\] hence \eqref{dim_ineq_1} must in fact be an equality, and for $0 < l \leq (n-r)/2$ the maps $\phi_l$ are surjections, so the quotients $N_l / N_{l-1} \isom S^{(r+2l, n-r-2l)}_2$ as required.   

Hence, when $\lambda = (n-r,1^r)$ with $n$ even and $n-r \leq r, S_2^\lambda$ has a filtration (from bottom to top) by the modules $S_2^{*(r,n-r)}, S_2^{(r+2,n-r-2)}, \ldots, S_2^{(r+2l, n-r-2l)}, \ldots,$ where $S_2^{*(r,n-r)}$ is a quotient of $S_2^{(r+1,n-r-1)}$ isomorphic to a submodule of $S_2^{(r,n-r)}.$

\end{proof}

\begin{example}\label{Example_S_2^(6_1^2)}
    With \Cref{FPL,second_filtration}, we can calculate the full submodule structure of $S_2^{(6,1^2)}.$ In particular, $S_2^{(6,1^2)}$ is filtered, from bottom to top, by $S_2^{(6,2)}$ and $S_2^{(8)}.$ $S_2^{(6,2)}$ is uniserial with socle isomorphic to $D_2^{(7,1)}$ and head isomorphic to $D_2^{(6,2)},$ and $S_2^{(8)}$ is $D_2^{(8)}.$ Hence the socle of $S_2^{(6,1^2)}$ contains $D_2^{(7,1)}$ but can not contain $D_2^{(6,2)},$ and the head of $S_2^{(6,1^2)}$ contains $D_2^{(8)}$ but can not contain $D_2^{(7,1)}.$
    
    $S_2^{(6,1^2)}$ is also filtered by, from bottom to top, the dual of $S_2^{(7,1)}$ and the dual of $S_2^{*(5,3)}.$ The dual of $S_2^{(7,1)}$ is uniserial with socle isomorphic to $D_2^{(7,1)}$ and head isomorphic to $D_2^{(8)},$ and the dual of $S_2^{*(5,3)}$ is isomorphic to $D_2^{(6,2)}.$ Hence the socle of $S_2^{(6,1^2)}$ contains $D_2^{(7,1)}$ but can not contain $D_2^{(8)},$ and the head of $S_2^{(6,1^2)}$ contains $D_2^{(6,2)}$ but can not contain $D_2^{(7,1)}.$

    So there are $3$ composition factors, $D_2^{(8)} ,D_2^{(7,1)}$ and $D_2^{(6,2)},$ each with multiplicity $1.$ The only factor that can appear in the socle is $D_2^{(7,1)},$ and the head must be exactly $D_2^{(8)} \oplus D_2^{(6,2)}.$ This gives the full submodule structure of $S_2^{(6,1^2)},$ which we illustrate below with the submodule lattice. 
    
    In the diagram below, the vertices represent submodules of $S_2^{(6,1^2)},$ and the nodes are labelled by the dimension of the submodule. The directed edge from a vertex $u$ to a vertex $v$ indicates that the submodule associated to $u$ is a maximal submodule of the submodule associated to $v$, and the edges are labelled by their irreducible subquotients. We also label the submodules $M_{-1}, M_0, M_1$ and the corresponding duals of $N_{-1}, N_0, N_1$. 

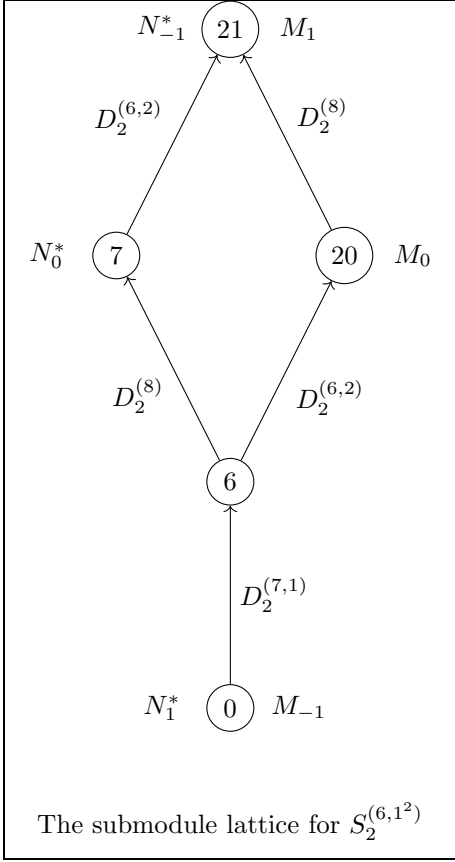
\begin{tabular}{|c|}
    \hline
    \begin{tikzpicture}[main/.style = {draw, circle}, scale = 3]
    \node at (0,-1/2) {The submodule lattice for $S^{( 6, 1^2 )}_2$};
    
    \node[main] (0) at (0,0) {$0$};
    \node[main] (6) at (0,1) {$6$};
    \node[main] (7) at (-1/2,2) {$7$};
    \node[main] (20) at (1/2,2) {$20$};
    \node[main] (21) at (0,3) {$21$};
    
    \draw[->] (0) edge ["$D_2^{(7,1)}$"', pos = 0.5] (6);
    
    \draw[->] (6) edge ["$D_2^{(8)}$", pos = 0.5] (7);
    \draw[->] (6) edge ["$D_2^{(6,2)}$"', pos = 0.5] (20);
    
    \draw[->] (7) edge ["$D_2^{(6,2)}$", pos = 0.5] (21);
    
    \draw[->] (20) edge ["$D_2^{(8)}$"', pos = 0.5] (21);

    \node at (0.3,0) {$M_{-1}$};
    \node at (0.8,2) {$M_0$};
    \node at (0.3,3) {$M_1$};
    \node at (-0.3,0) {$N_1^*$};
    \node at (-0.8,2) {$N_0^*$};
    \node at (-0.3, 3) {$N_{-1}^*$};
    \end{tikzpicture}\\
    \hline
\end{tabular}
\end{example}

\newpage
\section{Classification of Uniserial Hook Specht Modules in Characteristic 2}\label{Classification of Uniserial Hook Specht Modules in Characteristic 2}

In this section, our goal is to classify all uniserial hook Specht modules in characteristic $2.$ We use \Cref{FPL,2_part_uniserial} to find which hook Specht modules are filtered only by uniserial $2$-part Specht modules, and then refine to hook Specht modules which have simple socle \cite[Lemma 4.1, Theorem 4.2, Theorem 4.4]{murphy1982submodule}. This gives us a finite list of families of hook Specht modules to check.

First we observe that any subquotient of a uniserial module must also be uniserial, hence for a hook Specht module in characteristic $2$ to be uniserial, it is necessary that all of the $2$-part Specht modules that appear in its filtration must be uniserial. The result below shows that if $n-r \geq r \geq 30,$ then the hook Specht module $S_2^{(n-r,1^r)}$ has at least one non-uniserial $2$-part Specht module in its filtration, and hence is not uniserial.

\begin{lemma}\label{not_uniserial_great_than_29}
    Let $\lambda = (n-r,1^r) \vdash n,$ with $0 \leq r \leq n-r.$ If $r \geq 30$ then $S^\lambda_2$ is not uniserial.
\end{lemma}
\begin{proof}

Since $S_2^{(n-r,1^r)}$ has subquotients isomorphic to $S_2^{(n-r+2k, r-2k)}$ for $0 \leq 2k \leq r,$ it suffices to show that if $r$ is even, then at least one of $S_2^{(n)}, S_2^{(n-2,2)}, \ldots, S_2^{(n-30,30)}$ is not uniserial; and if $r$ is odd then at least one of $S_2^{(n-1,1)}, S_2^{(n-3,3)}, \ldots, S_2^{(n-31,31)}$ is not uniserial. By \Cref{period_2_part}, for fixed $s\leq 31,$ the uniseriality of $S_2^{(n-s,s)}$ depends only on $n \mod 32,$ which can then be checked using \Cref{2_part_uniserial}. So there are only finitely many cases to check, and for each congruence class modulo $32$ we give the smallest value of $s \leq 31$ with $s \equiv r \mod 2$ for which $S_2^{(n-s,s)}$ is not uniserial. We have two tables, one for $r$ even and one for $r$ odd.

\[\begin{array}{|c|c|c|c|c|}
\cline{1-2} \cline{4-5}
\multicolumn{2}{|c|}{r \text{ is even}} & & \multicolumn{2}{|c|}{r\text{ is odd}} \\
\cline{1-2} \cline{4-5}
n \mod 32 & s &  & n \mod 32 & s  \\
\cline{1-2} \cline{4-5}
 0 & 6     && 0 & 7     \\
 1 & 12    && 1 & 23    \\
 2 & 8     && 2 & 7     \\
 3 & 24    && 3 & 13    \\
 4 & 4     && 4 & 9     \\
 5 & 14    && 5 & 25    \\
 6 & 6     && 6 & 5     \\
 7 & 26    && 7 & 15    \\
 8 & 6     && 8 & 7     \\
 9 & 8     && 9 & 27    \\
10 & 8     && 10 & 7    \\
11 & 28    && 11 & 9    \\
12 & 4     && 12 & 9    \\
13 & 10    && 13 & 29   \\
14 & 6     && 14 & 5    \\
15 & 30    && 15 & 11   \\
16 & 6     && 16 & 7    \\
17 & 12    && 17 & 31   \\
18 & 8     && 18 & 7    \\
19 & 16    && 19 & 13   \\
20 & 4     && 20 & 9    \\
21 & 14    && 21 & 17   \\
22 & 6     && 22 & 5    \\
23 & 18    && 23 & 15   \\
24 & 6     && 24 & 7    \\
25 & 8     && 25 & 19   \\
26 & 8     && 26 & 7    \\
27 & 20    && 27 & 9    \\
28 & 4     && 28 & 9    \\
29 & 10    && 29 & 21   \\
30 & 6     && 30 & 5    \\
31 & 22    && 31 & 11   \\
\cline{1-2} \cline{4-5}
\end{array}\]

\end{proof}

\begin{remark}
    Note that the $2$-part Specht modules that appear in the filtrations of $S^{(52,1^{29})}$ and $S^{(51,1^{28})}$ are all uniserial, and so $30$ is the smallest $r$ for which this result holds for general $n$.
\end{remark}

We get the following results from \cite[Lemma 4.1, Theorem 4.2, Theorem 4.4]{murphy1982submodule}. Note that Murphy gives some results $\mod 2^{L(r-1)},$ which can be written $\mod 2^{L(r)}$ as in the third and fourth bullet points below:

\begin{lemma}\label{hook_unique_minimal_submodule}
    Let $0 \leq r \leq n-r.$ Then $S^{(n-r,1^r)}$ has a unique minimal submodule if and only if one of the following occurs:
    \begin{itemize}
        \item $r \equiv -1, n \equiv -1 \mod 2^{L(r)}$ (hence $n-r \equiv 0 \mod 2^{L(r)}),$
        \item $r \equiv -1, n \equiv -2 \mod 2^{L(r)}$ (hence $n-r \equiv -1 \mod 2^{L(r)}),$
        \item $r \equiv -2, n \equiv -3 \mod 2^{L(r)}$ (hence $n-r \equiv -1 \mod 2^{L(r)}),$
        \item $r \equiv 2^{L(r)-1}, n \equiv 0 \mod 2^{L(r)}$ (hence $n-r \equiv 2^{L(r)-1} \mod 2^{L(r)}),$
        \item $r \equiv 2^{L(r)-1}, n \equiv 2^{L(r)-1} \mod 2^{L(r)}$ (hence $n-r \equiv 0 \mod 2^{L(r)}).$
    \end{itemize}
\end{lemma}

Combining \Cref{hook_unique_minimal_submodule,not_uniserial_great_than_29}, we get the remaining following cases for when $S_2^{(n-r,1^r)}$ has a unique minimal submodule and $r < 30$. As before, we only need to define $n$ modulo $2^{L(r)}.$

\[
\begin{array}{|c|c|}
    \hline
         n\mod 2^{L(r)} & r\\
         \hline
         0 & 0 \\
         0 & 1 \\ 
         1 & 1 \\
         0 & 2 \\
         1 & 2 \\
         2 & 2 \\
         2 & 3 \\
         3 & 3 \\
         0 & 4 \\
         4 & 4 \\
         5 & 6 \\
         6 & 7 \\
         7 & 7 \\
         0 & 8 \\
         8 & 8 \\
         13 & 14 \\
         14 & 15 \\
         15 & 15 \\
         0 & 16 \\
         16 & 16 \\
\hline         
\end{array}
\]
\qquad 

By applying the same logic as in \Cref{not_uniserial_great_than_29}, we can show that some of these hook Specht modules contain a non-uniserial $2$-part Specht module in its filtration. Given $r,$ and $n \mod 2^{L(r)},$ we find the smallest $s$ such that $S_2^{(n-s,s)}$ appears in the filtration of $S_2^{(n-r,1^r)}$ is not uniserial.

\[
    \begin{array}{|c|c|c|c|}
    \hline
         n \mod 2^{L(r)} & r & s  \\
         \hline
         4 & 4 & 4  \\
         6 & 7 & 5  \\
         0 & 8 & 6  \\         
         8 & 8 & 6  \\
         13 & 14 & 10  \\
         14 & 15 & 5  \\
         15 & 15 & 11 \\
         0 & 16 & 6 \\
         16 & 16 & 6 \\
         \hline         
    \end{array}
\]
\qquad

That leaves the following cases, listed in the table below. 

\[
\begin{array}{|c|c|}
    \hline
         n\mod 2^{L(r)} & r\\
         \hline
         0 & 0 \\
         0 & 1 \\ 
         1 & 1 \\
         0 & 2 \\
         1 & 2 \\
         2 & 2 \\
         2 & 3 \\
         3 & 3 \\
         0 & 4 \\
         5 & 6 \\
         7 & 7 \\
\hline         
\end{array}
\]
\qquad

For the final proof, the following lemma will be useful.
\begin{lemma}\label{simple_in_socle}
Let $M$ be a module over an algebra, with filtration $M = M_r > \ldots > M_0 = 0.$ If $T$ is a module and $\Hom( T , M) \neq 0,$ then there is some $s$ such that $\Hom ( T , M_s / M_{s-1}) \neq 0.$
\end{lemma}

\begin{proof}
    Let $\theta : T \to M $ be a non-zero homomorphism, and let $s$ be minimal such that $\Im(\theta) \subseteq M_s.$ Then, after composing with the quotient of $M_{s-1},$ there is a map $\theta' : T \to M_s / M_{s-1}.$ Assume for contradiction that this map is $0,$ then $\Im(\theta) \subseteq M_{s-1},$ a contradiction with the minimality of $s.$ Hence $\Hom ( T , M_s / M_{s-1} ) \neq 0.$
\end{proof}

We finish the proof by considering cases in the above table. Recall that if $n-r < r,$ the submodule structure of $S_2^{(n-r,1^r)}$ is dual to that of $S_2^{(r+1, n-r-1)}$ and $r+1 > n-r-1,$ and so $S_2^{(n-r,1^r)}$ is uniserial if and only if $S_2^{(r+1,n-r-1)}$ is. Hence there is no loss in generality in assuming that $n-r\geq r.$

\begin{theorem}\label{classification_of_uniserial_hooks}
    Let $\lambda = (n-r,1^r) \vdash n, 0 \leq r < n-r.$ Then $S^\lambda_2$ is uniserial if and only if one of the following occurs:
    \begin{itemize}
        \item $r=0;$
        \item $r=1;$ 
        \item $r=2$ and $n \equiv 1,2 \mod 4;$
        \item $r=3$ and $n \equiv 3 \mod 4.$
    \end{itemize}
\end{theorem}

\begin{proof}
    In the case $r=0$ or $r=1, \lambda$ has at most $2$ parts and so \Cref{2_part_submod_structure_new_proof} gives the full combinatorial description. These Specht modules are either simple, or have a unique proper non-zero submodule.

    In the case $r=2$ and $n\equiv 1 \mod 4,$ or $r=3$ and $n \equiv 3 \mod 4, S_2^\lambda$ has a filtration from bottom to top with factors isomorphic to $S_2^{(n-r,r)}$ and $S_2^{(n-r+2,r-2)}$. $S_2^{(n-r,r)}$ itself is uniserial with $2$ composition factors; with $D_2^{(n-r+2,r-2)}$ in the socle and $D_2^{(n-r,r)}$ in the head, and $S_2^{(n-r+2,r-2)}$ is simple. As $S_2^\lambda$ has $3$ composition factors, has simple socle by \Cref{hook_unique_minimal_submodule} and is self-dual by \Cref{odd_self_dual} hence has simple head, $S_2^\lambda$ must therefore be uniserial.

    In the case $r=2$ and $n \equiv 2 \mod 4, S^\lambda_2$ has a filtration, from bottom to top, via $S^{(n-2,2)}_2$ and $S^{(n)}_2.$ $S^{(n-2,2)}_2$ itself is uniserial with composition factors, from bottom to top, $D^{(n-1,1)}_2, D^{(n)}_2,$ and $D^{(n-2,2)}_2$ by \Cref{2_part_submod_structure}. By considering the dual, we also get that $S^{\lambda}_2$ has a filtration from, bottom to top, by the dual of $S^{(n-1,1)}_2$ and the dual of $S^{*(n-3,3)}_2.$ The dual of $S^{(n-1,1)}_2$ is uniserial with composition factors, from bottom to top, $D^{(n-1,1)}_2$ and $D^{(n)}_2,$ and the dual of $S^{*(n-3,3)}_2$ is uniserial with composition factors, from bottom to top, $D^{(n-2,2)}_2$ and $D^{(n)}_2.$ Hence we must have that $S^\lambda_2$ is uniserial with composition factors, from bottom to top, $D^{(n-1,1)}_2, D^{(n)}_2, D^{(n-2,2)}_2, D^{(n)}_2.$

    We also need to justify that no other hook Specht modules in characteristic $2$ are uniserial. In particular, we need to show that $S_2^{(n-r,1^r)}$ is not uniserial in the following cases:
    \begin{itemize}
        \item $r=2, n \equiv 0 \mod 4;$
        \item $r=3, n \equiv 2 \mod 4;$
        \item $r=4, n \equiv 0 \mod 8;$
        \item $r=6, n \equiv 5 \mod 8;$
        \item $r=7, n \equiv 7 \mod 8.$
    \end{itemize}

    Assume $r=2$ and $n\equiv 0 \mod 4$ or $r=4$ and $n \equiv 0 \mod 8,$ and assume for contradiction that $S_2^{(n-r,1^r)}$ is uniserial. $S_2^\lambda$ has a filtration with $S_2^{(n)}$ at the top; and a filtration with the dual of $S_2^{*(n-r-1,r+1)}$ at the top. But $S_2^{*(n-r-1,r+1)},$ a submodule of $S_2^{(n-r-1,r+1)}$ does not have $S_2^{(n)}$ as a submodule for these values of $r$ and $n.$ This is a contradiction, and $S_2^\lambda$ must not be uniserial, in particular, the head is not simple.

    Similarly, let $r=3$ and $n \equiv 2 \mod 4$ and assume for contradiction that $S_2^{(n-3,3)}$ is uniserial. $S_2^\lambda$ has a filtration with $S_2^{(n-1,1)}$ at the top; and a filtration with the dual of $S_2^{*(n-4,4)}$ at the top. But $S_2^{*(n-4,4)},$ a submodule of $S_2^{(n-4,4)}$ does not have $D_2^{(n-1,1)}$ as a submodule for $r=3$ and $n \equiv 2 \mod 4.$ This is a contradiction, and $S_2^\lambda$ must not be uniserial.

    Finally, let $r=6$ and $n\equiv 5 \mod 8$ or $r=7$ and $n \equiv 7 \mod 8.$ Assume for a contradiction that $S_2^{(n-r,r)}$ is uniserial. Note that $S_2^\lambda$ is self-dual as $n$ is odd by \Cref{odd_self_dual}, and contains a submodule isomorphic to $S_2^{(n-r,r)},$ which is itself uniserial and the composition series has factors (from bottom to top) isomorphic to $D_2^{(n-r+6,r-6)}, D_2^{(n-r+2,r-2)}, D_2^{(n-r,r)}$ by \Cref{2_part_submod_structure_new_proof}. $D_2^{(n-r,r)}$ appears as a composition factor of $S_2^\lambda$ precisely once, so by self-duality and assumption of uniseriality, the composition series of $S_2^\lambda$ has simple factors, in order, $D_2^{(n-r+6,r-6)}, D_2^{(n-r+2,r-2)}, D_2^{(n-r,r)}, D_2^{(n-r+2,r-2)}, D_2^{(n-r+6,r-6)}.$ This is a contradiction as $[S_2^\lambda : D_2^{(n-r+4,r-4)}] \neq 0.$
\end{proof}

\newpage
\section{Further Questions}\label{Further Questions}

There are many potential avenues to continue the research in this paper.

\Cref{classification_of_uniserial_hooks} classifies all uniserial hook Specht modules in characteristic $2$; can this be extended to look at which hook Specht modules can be written as a direct sum of uniserial parts? Can one go further and give a full combinatorial description of the submodule structure of hook Specht modules in characteristic $2?$ 

Another nice class of modules over a ring are those whose submodule lattices are distributive lattices, we call such modules \emph{distributive modules}. More precisely, we say that a module $M$ is a distributive module if and only if for all submodules $U, V, W$ of $M,$ we have that $U \cap (V+W) = (U\cap V) + (U \cap W)$. 

\begin{theorem}
    Let $R$ be a ring and $M$ an $R$-module of finite length $n$. Then the following are equivalent.
    \begin{enumerate}
        \item $M$ is distributive. That is, the submodule lattice of $M$ is a distributive lattice.
        \item $M$ does not have a subquotient which is a direct sum of two isomorphic non-zero modules.
        \item $M$ does not have a subquotient which is a direct sum of two isomorphic simple modules. 
        \item $M$ has exactly $n$ simple-headed submodules.
        \item For each simple $R$-module $S$ and each integer $m \in \{1, \ldots, [M:S]\},$ there is a unique $S$-simple-headed submodule $N \subseteq M$ such that $[N:S] = m.$
        \item If $K, L \subseteq M$ and $K$ and $L$ have exactly the same composition factors with multiplicity, then $K=L.$ 
    \end{enumerate}
    If $R$ is an algebra over an infinite field, then the above are equivalent to the following condition.
    \begin{enumerate}
        \setcounter{enumi}{6}
        \item $M$ has only finitely many submodules.
    \end{enumerate}
\end{theorem}

For a general module $M,$ the submodule lattice can be very large and complicated, but if $M$ is distributive, there are nice diagrams which contain the same information as the submodule lattice of $M$ but with far fewer nodes and edges \cite{benson1985diagrams}.

It is clear that if $\lambda$ is a $2$-part partition, then $S_p^\lambda$ is distributive, as it is multiplicity-free. We have the following conjecture for hook Specht modules in characteristic $2.$

\begin{conjecture}
    Let $\lambda = (n-r,1^r) \vdash n$ with $0 \leq r \leq n-r.$ If $r \geq 10,$ then $S_2^\lambda$ is not distributive. 
\end{conjecture}

We also conjecture that there is a similar result to \Cref{period_2_part} which holds for hook Specht modules in characteristic $2:$

\begin{conjecture}\label{hook_periodic_conjecture}
    Let $\lambda = (n-r,1^r) \vdash n$ and $\mu = (m-r,1^r) \vdash m,$ with $n-r > r$ and $m-r > r.$ If $n \equiv m \mod 2^{L(r)}$ then the submodule lattices of $S_2^\lambda$ and $S_2^\mu$ are isomorphic.
\end{conjecture}

We demonstrate an example of \Cref{hook_periodic_conjecture} by comparing the submodule lattices of $S^{(5,1^3)}_2, S^{(9,1^3)}_2$ and $S^{(13,1^3)}_2.$ In the diagrams below, the vertices represent submodules of $S^\lambda_2,$ and the nodes are labelled by the dimension of the submodule. The directed edge from a vertex $u$ to a vertex $v$ indicates that the submodule associated to $u$ is a maximal submodule of the submodule associated to $v,$ and the edges are labelled by their irreducible subquotients. \\
\newline
\scalebox{0.8}{
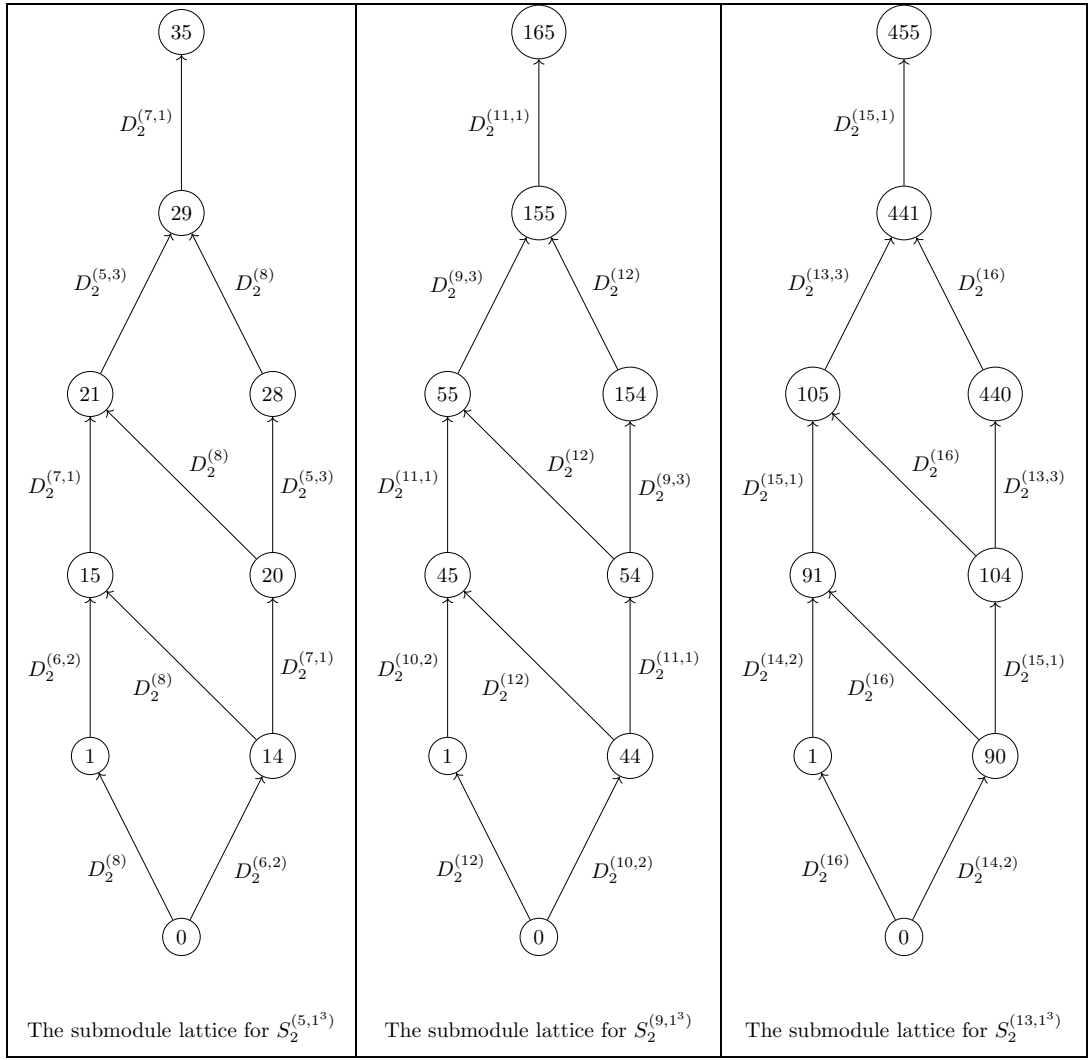
\begin{tabular}{|c|c|c|}
         \hline
        \begin{tikzpicture}[main/.style = {draw, circle}, scale = 3]
            \node at (0,-1/2) {The submodule lattice for $S^{( 5, 1^3 )}_2$};
            
            \node[main] (0) at (0,0) {$0$};
            \node[main] (1) at (-1/2,1) {$1$};
            \node[main] (14) at (1/2,1) {$14$};
            \node[main] (15) at (-1/2,2) {$15$};
            \node[main] (20) at (1/2,2) {$20$};
            \node[main] (21) at (-1/2,3) {$21$};
            \node[main] (28) at (1/2,3) {$28$};
            \node[main] (29) at (0,4) {$29$};
            \node[main] (35) at (0,5) {$35$};
            
            \draw[->] (0) edge ["$D_2^{(8)}$", pos = 0.5] (1);
            \draw[->] (0) edge ["$D_2^{(6,2)}$"', pos = 0.5] (14);
            
            \draw[->] (1) edge ["$D_2^{(6,2)}$", pos = 0.5] (15);
            
            \draw[->] (14) edge ["$D_2^{(8)}$", pos = 0.5] (15);
            \draw[->] (14) edge ["$D_2^{(7,1)}$"', pos = 0.5] (20);
            
            \draw[->] (15) edge ["$D_2^{(7,1)}$", pos = 0.5] (21);
            
            \draw[->] (20) edge ["$D_2^{(8)}$"', pos = 0.5] (21);
            \draw[->] (20) edge ["$D_2^{(5,3)}$"', pos = 0.5] (28);
            
            \draw[->] (21) edge ["$D_2^{(5,3)}$", pos = 0.5] (29);
            
            \draw[->] (28) edge ["$D_2^{(8)}$"', pos = 0.5] (29);
            
            \draw[->] (29) edge ["$D_2^{(7,1)}$", pos = 0.5] (35);
            
        \end{tikzpicture}
         &  
        \begin{tikzpicture}[main/.style = {draw, circle}, scale = 3]
            \node at (0,-1/2) {The submodule lattice for $S^{( 9, 1^3 )}_2$};
            
            \node[main] (0) at (0,0) {$0$};
            \node[main] (1) at (-1/2,1) {$1$};
            \node[main] (44) at (1/2,1) {$44$};
            \node[main] (45) at (-1/2,2) {$45$};
            \node[main] (54) at (1/2,2) {$54$};
            \node[main] (55) at (-1/2,3) {$55$};
            \node[main] (154) at (1/2,3) {$154$};
            \node[main] (155) at (0,4) {$155$};
            \node[main] (165) at (0,5) {$165$};
            
            \draw[->] (0) edge ["$D_2^{(12)}$", pos = 0.5] (1);
            \draw[->] (0) edge ["$D_2^{(10,2)}$"', pos = 0.5] (44);
            
            \draw[->] (1) edge ["$D_2^{(10,2)}$", pos = 0.5] (45);
            
            \draw[->] (44) edge ["$D_2^{(12)}$", pos = 0.5] (45);
            \draw[->] (44) edge ["$D_2^{(11,1)}$"', pos = 0.5] (54);
            
            \draw[->] (45) edge ["$D_2^{(11,1)}$", pos = 0.5] (55);
            
            \draw[->] (54) edge ["$D_2^{(12)}$"', pos = 0.5] (55);
            \draw[->] (54) edge ["$D_2^{(9,3)}$"', pos = 0.5] (154);
            
            \draw[->] (55) edge ["$D_2^{(9,3)}$", pos = 0.5] (155);
            
            \draw[->] (154) edge ["$D_2^{(12)}$"', pos = 0.5] (155);
            
            \draw[->] (155) edge ["$D_2^{(11,1)}$", pos = 0.5] (165);
            
        \end{tikzpicture}
        &
        \begin{tikzpicture}[main/.style = {draw, circle}, scale = 3]
            \node at (0,-1/2) {The submodule lattice for $S^{( 13, 1^3 )}_2$};
            
            \node[main] (0) at (0,0) {$0$};
            \node[main] (1) at (-1/2,1) {$1$};
            \node[main] (90) at (1/2,1) {$90$};
            \node[main] (91) at (-1/2,2) {$91$};
            \node[main] (104) at (1/2,2) {$104$};
            \node[main] (105) at (-1/2,3) {$105$};
            \node[main] (440) at (1/2,3) {$440$};
            \node[main] (441) at (0,4) {$441$};
            \node[main] (455) at (0,5) {$455$};
            
            \draw[->] (0) edge ["$D_2^{(16)}$", pos = 0.5] (1);
            \draw[->] (0) edge ["$D_2^{(14,2)}$"', pos = 0.5] (90);
            
            \draw[->] (1) edge ["$D_2^{(14,2)}$", pos = 0.5] (91);
            
            \draw[->] (90) edge ["$D_2^{(16)}$", pos = 0.5] (91);
            \draw[->] (90) edge ["$D_2^{(15,1)}$"', pos = 0.5] (104);
            
            \draw[->] (91) edge ["$D_2^{(15,1)}$", pos = 0.5] (105);
            
            \draw[->] (104) edge ["$D_2^{(16)}$"', pos = 0.5] (105);
            \draw[->] (104) edge ["$D_2^{(13,3)}$"', pos = 0.5] (440);
            
            \draw[->] (105) edge ["$D_2^{(13,3)}$", pos = 0.5] (441);
            
            \draw[->] (440) edge ["$D_2^{(16)}$"', pos = 0.5] (441);
            
            \draw[->] (441) edge ["$D_2^{(15,1)}$", pos = 0.5] (455);
            
        \end{tikzpicture} 
        \\
         \hline
    \end{tabular}
    }

Can one extend the combinatorial description of the submodule structure of $2$-part Specht modules as in \Cref{2_part_submod_structure_new_proof} to a similar construction for the submodule structure of Specht modules in general?

\newpage
\printbibliography 

@book{James1978,
Author = {James, G.D.},
ISBN = {3-540-08948-9},
Publisher = {Springer, Berlin},
Series = {Lecture Notes in Mathematics, 682.},
Title = {The Representation Theory of the Symmetric Groups},
Year = {1978},
}

@article{JamesMathas1999,
author = {James, G.D. and Mathas, A.},
title = {The Irreducible Specht Modules in Characteristic 2},
journal = {Bulletin of the London Mathematical Society},
volume = {31},
number = {4},
pages = {457-462},
year = {1999}
}

@article{Fayers2005,
title = {Irreducible Specht Modules for Hecke Algebras of Type A},
journal = {Advances in Mathematics},
volume = {193},
number = {2},
pages = {438-452},
year = {2005},
issn = {0001-8708},
author = {Fayers, M.},
}

@article{kleshchev1999extensions,
  title={On Extensions of Simple Modules Over Symmetric and Algebraic Groups},
  author={Kleshchev, A.S. and Sheth, J.},
  journal={Journal of Algebra},
  volume={221},
  number={2},
  pages={705--722},
  year={1999},
  publisher={Elsevier Science}
}

@article{murphy1980decomposability,
  title={On Decomposability of Some Specht Modules for Symmetric Groups},
  author={Murphy, G.},
  journal={Journal of Algebra},
  volume={66},
  number={1},
  pages={156--168},
  year={1980}
}

@article{murphy1982submodule,
  title={Submodule Structure of Some Specht Modules},
  author={Murphy, G.},
  journal={Journal of Algebra},
  volume={74},
  number={2},
  pages={524--534},
  year={1982},
  publisher={Academic Press}
}

@phdthesis{sutton2017graded,
  title={Graded Representations of Khovanov-Lauda-Rouquier Algebras},
  author={Sutton, L.},
  year={2017},
  school={Queen Mary University of London}
}

@article{dodge2012some,
  title={Some New Decomposable Specht Modules},
  author={Dodge, C.J. and Fayers, M.},
  journal={Journal of Algebra},
  volume={357},
  pages={235--262},
  year={2012},
  publisher={Elsevier}
}

@article{peel1971hook,
  title={Hook Representations of the Symmetric Groups},
  author={Peel, M.H.},
  journal={Glasgow Mathematical Journal},
  volume={12},
  number={2},
  pages={136--149},
  year={1971},
  publisher={Cambridge University Press}
}

@article{benson1985diagrams,
  title={Diagrams for Modular Lattices},
  author={Benson, D.J. and Conway, J.H.},
  journal={Journal of Pure and Applied Algebra},
  volume={37},
  pages={111--116},
  year={1985},
  publisher={North-Holland}
}

@article{brundan2000translation,
  title={On Translation Functors for General Linear and Symmetric Groups},
  author={Brundan, J. and Kleshchev, A.S.},
  journal={Proceedings of the London Mathematical Society},
  volume={80},
  number={1},
  pages={75--106},
  year={2000},
  publisher={Cambridge University Press}
}

@inproceedings{brauer1947conjecture,
  title={On a Conjecture by Nakayama},
  author={Brauer, R.},
  year={1947},
  organization={Royal Society of Canada}
}

@article{robinson1947conjecture,
  title={On a Conjecture by Nakayama},
  author={Robinson, G de B},
  journal={Trans. Roy. Soc. Canada. Sect. III},
  volume={41},
  pages={20--25},
  year={1947}
}

@article{littlewood1951modular,
  title={Modular Representations of Symmetric Groups},
  author={Littlewood, D.E. and Aitken, A.C.},
  journal={Proceedings of the Royal Society of London. Series A. Mathematical and Physical Sciences},
  volume={209},
  number={1098},
  pages={333--353},
  year={1951},
  publisher={The Royal Society London}
}
\end{document}